\documentclass[12pt]{article}%
\usepackage{amssymb}
\usepackage{amsfonts}
\usepackage{amsmath}
\usepackage{hyperref}
\usepackage{tikz}
\usepackage{CJK}%
\usepackage{graphicx}
\providecommand{\U}[1]{\protect\rule{.1in}{.1in}}
\usetikzlibrary{matrix,arrows}
\numberwithin{equation}{section}
\newtheorem{theorem}{Theorem}[section]

\newtheorem{corollary}[theorem]{Corollary}

\newtheorem{definition}[theorem]{Definition}
\newtheorem{example}[theorem]{Example}

\newtheorem{lemma}[theorem]{Lemma}

\newtheorem{proposition}[theorem]{Proposition}
\newtheorem{remark}[theorem]{Remark}

\newenvironment{proof}[1][Proof]{\noindent\textbf{#1.} }{\ \rule{0.5em}{0.5em}}

\usepackage{tikzexternal}
\tikzsetexternalprefix{figures/}

\tikzsetfigurename{center_}

\begin{document}

\title{Centers of Braided Tensor Categories}
\author{Zhimin Liu\thanks{E-mail: zhiminliu13@fudan.edu.cn} \thanks{Project funded by China Postdoctoral Science Foundation grant 2019M661327}\hspace{1cm} Shenglin Zhu
\thanks{CONTACT: mazhusl@fudan.edu.cn} \thanks{This work was
supported by NNSF of China (No. 11331006).}\\Fudan University, Shanghai 200433, China}
\date{}
\maketitle

\begin{abstract}
Let $\mathcal{C}$ be a finite braided multitensor category. Let $B$ be Majid's
automorphism braided group of $\mathcal{C}$, then $B$ is a cocommutative Hopf
algebra in $\mathcal{C}$. We show that the center of $\mathcal{C}$ is
isomorphic to the category of left $B$-comodules in $\mathcal{C}$, and the
decomposition of $B$ into a direct sum of indecomposable $\mathcal{C}%
$-subcoalgebras leads to a decomposition of $B$-$\operatorname*{Comod}%
_{\mathcal{C}}$ into a direct sum of indecomposable $\mathcal{C}$-module subcategories.

As an application, we present an explicit characterization of the structure of
irreducible Yetter-Drinfeld modules over semisimple quasi-triangular weak Hopf
algebras. Our results generalize those results on finite groups and on
quasi-triangular Hopf algebras.

\end{abstract}


\textbf{2000 MATHEMATICS SUBJECT CLASSIFICATION}: 16W30

\section{Introduction}

\label{section-intro}

The theory of module categories over a tensor category was introduced
respectively by Bernstein's \cite{Bernstein1995Sackler}, by Crane and Frenkel
\cite{Crane1994Four}, and well-developed by Ostrik \cite{Ostrik2003module}, by
Etingof and Ostrik \cite{Etingof2004finite}.

Let $\left(  \mathcal{M},\otimes,a,\ell\right)  $ be a semisimple module
category over a finite multitensor category $\mathcal{C}$, and $M\in
\mathcal{M}$ be a generator of $M$. It is proved in
\cite{Ostrik2003module,Etingof2004finite,Etingof2015tensor} that
$A=\underline{\operatorname{Hom}}\left(  M,M\right)  $ is a semisimple algebra
in $\mathcal{C}$, and the internal Hom functor $F=\underline
{\operatorname{Hom}}\left(  M,\bullet\right)  :\mathcal{M}\rightarrow
\operatorname{Mod}_{\mathcal{C}}$-$A$ induces a $\mathcal{C}$-module category
equivalence. The proof is based on the fact that $F$ is faithful and full, and
essentially surjective on objects. In \cite{LiuZhu2019On}, for a right
$A$-module $\left(  U,q\right)  $ in $\mathcal{C}$ and a left $A$-module
$(N,p)$ in $\mathcal{M}$ the authors defined the tensor product $U\otimes
_{A}N$ and proved that the functor $G=\bullet\otimes_{A}M:\operatorname{Mod}%
_{\mathcal{C}}$-$A\rightarrow\mathcal{M}$ is a quasi-inverse of $F$.

Let $\mathcal{C}$ be a monoidal category. There is a well-known braided
category construction $\mathcal{Z}_{l}\left(  \mathcal{C}\right)  $, called
the Drinfeld center of $\mathcal{C}$ (see \cite{Joyal1991Tortile}). The
objects of $\mathcal{Z}_{l}\left(  \mathcal{C}\right)  $ are those objects of
$\mathcal{C}$ together with natural transformations satisfying a hexagon
axiom. The center is a categorical version of the Hopf algebraic construction
of the Drinfeld double. If $H$ is a finite dimensional Hopf algebra over a
field and $\mathcal{C=}${}$_{H}\mathcal{M}$, then $\mathcal{Z}_{l}\left(
\mathcal{C}\right)  $ is equivalent to the Yetter-Drinfeld module category
${}_{H}^{H}\mathcal{YD}$.

Assume further that $\mathcal{C}$ is braided. The center $\mathcal{Z}%
_{l}\left(  \mathcal{C}\right)  $ can be viewed as a right module category
over $\mathcal{C}$. If $\mathcal{C}$ is multitensor with certain additional
assumption, then there is a cocommutative $\mathcal{C}$-Hopf algebra $U\left(
\mathcal{C}\right)  $, coming from the braided reconstruction theory, which is
named the automorphism braided group of $\mathcal{C}$ by Majid
\cite{Majid1991Reconstruction,Majid1991Braided}.

Let $\left(  H,R\right)  $ be a quasi-triangular Hopf algebra over a field $k$
and $\mathcal{C}$ be the braided tensor category {}$_{H}\mathcal{M}$. Then
$U\left(  \mathcal{C}\right)  =H$ \cite{Majid1991Braided}, with the same
algebra structure of $H$ and an $R$-twisted coalgebra structure $\Delta_{R}$.
(For brevity, we denote the $\mathcal{C}$-coalgebra $\left(  H,\Delta
_{R}\right)  $ by $H_{R}$.) The Yetter-Drinfeld module category${}$ $_{H}%
^{H}\mathcal{YD}$ is equivalent to the relative module category {}$_{H}%
^{H_{R}}\mathcal{M}$ \cite{Zhu2015Relative}. In \cite{LiuZhu2019On}, the
authors have proved that each Yetter-Drinfeld submodule of $H\in{}_{H}%
^{H}\mathcal{YD}$ is a subcoalgebra of $H_{R}$, and $H$ admits a unique
decomposition into the direct sum of indecomposable Yetter-Drinfeld
submodules, while this decomposition coincides with the direct sum $\left(
H,\Delta_{R}\right)  =D_{1}\oplus\cdots\oplus D_{r}$ of the indecomposable
$\mathcal{C}$-subcoalgebras of $H_{R}$. Furthermore, the tensor category
\[
{}_{H}^{H}\mathcal{YD}\cong{}_{H}^{H_{R}}\mathcal{M=}\oplus_{i=1}^{r}{}%
_{H}^{D_{i}}\mathcal{M}%
\]
is a canonical direct sum of indecomposable module categories over
$\mathcal{C}$, and by
\cite{Ostrik2003module,Etingof2004finite,Etingof2015tensor} each category
${}_{H}^{D_{i}}\mathcal{M}$ is equivalent to the category $A_{i}%
$-$\operatorname{Mod}_{\mathcal{C}}$, where $A_{i}=\underline
{\operatorname{Hom}}\left(  M_{i},M_{i}\right)  $ for a nonzero object
$M_{i}\in{}_{H}^{D_{i}}\mathcal{M} $.

Moreover, ${}_{H}^{H}\mathcal{YD}$ can also be viewed as a left module
category over $\mathcal{C}^{\prime}=Vec_{k}$. In this case, internal Homs in
${}_{H}^{H}\mathcal{YD}$ are constructed concretely, and the structure of
irreducible objects of${}_{H}^{H}\mathcal{YD}$ are given in
\cite{LiuZhu2019On}. This structure theorem deduces the classical results on
finite groups.

This paper is devoted to the study of the center $\mathcal{Z}_{l}\left(
\mathcal{C}\right)  $ of a finite braided multitensor category $\mathcal{C}$.
We develop a purely categorical version of the structure theorem on
Yetter-Drinfeld modules for quasi-triangular Hopf algebras, which appeared in
\cite{LiuZhu2019On}, extend the results to the center of finite braided
multitensor categories. Explicitly, we prove that as module categories over
$\mathcal{C}$, $\mathcal{Z}_{l}\left(  \mathcal{C}\right)  $ is equivalent to
the category $U\left(  \mathcal{C}\right)  $-$\operatorname*{Comod}%
_{\mathcal{C}}$ of left $U\left(  \mathcal{C}\right)  $-comodules in
$\mathcal{C}$, and the decomposition of $U\left(  \mathcal{C}\right)  $ into a
direct sum of indecomposable $\mathcal{C}$-subcoalgebras leads to a
decomposition of $U\left(  \mathcal{C}\right)  $-$\operatorname*{Comod}%
_{\mathcal{C}}$ into a direct sum of indecomposable $\mathcal{C}$-module
subcategories, and each such indecomposable $\mathcal{C}$-module subcategory
is equivalent to the category of left modules over a $\mathcal{C}$-algebra.
And we present a characterization of the internal Hom for $U\left(
\mathcal{C}\right)  $-$\operatorname*{Comod}_{\mathcal{C}}$.

It is known that any finite multifusion category is equivalent to the category
of finite dimensional representations of a regular semisimple weak Hopf
algebra \cite{Hayashi1999ACanonical}\cite{Szlachanyi2000Finite}. The main
results of this paper are applied to the study of Yetter-Drinfeld module for
quasi-triangular weak Hopf algebras. An explicit characterization of the
structure of irreducible Yetter-Drinfeld modules over semisimple
quasi-triangular weak Hopf algebras will be given, which generalize those
results on finite groups \cite{dijkgraaf1992quasi,gould1993quantum} and on
quasi-triangular Hopf algebras \cite{LiuZhu2019On}.

The paper is organized as follows. Section \ref{sec-preliminary} recalls
module categories, Drinfeld centers of monoidal categories. Section
\ref{Section Center of BRC} discusses the center $\mathcal{Z}_{l}\left(
\mathcal{C}\right)  $ of a braided rigid category $\mathcal{C}$. Using
graphical calculus, we prove that when the automorphism braided group
$U\left(  \mathcal{C}\right)  $ exists, the category $\mathcal{Z}_{l}\left(
\mathcal{C}\right)  $ is equivalent to the category $U\left(  \mathcal{C}%
\right)  $-$\operatorname*{Comod}_{\mathcal{C}}$ of left $U\left(
\mathcal{C}\right)  $-comodules in $\mathcal{C}$. In Section \ref{sec-decomp
thm} we show that a decomposition of the automorphism braided group induces a
decomposition of $\mathcal{Z}_{l}\left(  \mathcal{C}\right)  $ as
$\mathcal{C}$-module subcategories. Section \ref{sec-appl to WHA} is devoted
to an application of the theory developed to weak Hopf algebras.

\section{Preliminaries}

\label{sec-preliminary}

\subsection{Notations and Conventions}

Throughout this paper, $k$ denotes a field, and $\mathrm{Vec}_{k}$ denotes the
category of finite dimensional vector spaces over $k$. For the basic theory of
monoidal categories, the reader is referred to \cite{Etingof2015tensor}. It is
well-known that any monoidal category is equivalent to a strict one by
MacLane's strictness theorem \cite{MacLane1998Categories}, we assume that the
monoidal categories considered are all strict.

Let $\left(  \mathcal{C},\otimes,1\right)  $ be a monoidal category. We will
use graphical calculus to calculate in $\mathcal{C}$, representing morphisms
by diagrams to be read downwards. Our references are
\cite{Yetter1990Quantum,Kassel1995Quantum,NeuchlSchauenburg1998Reconstruction}%
. We denote respectively the evaluations, the coevaluations for left dual
$X^{\ast}$ and right dual $^{\ast}X$ of an object $X\in\mathcal{C}$ by
\[
ev_{X}%
=\tikz[baseline=(current bounding box.west),scale=0.7,samples=100,thick] { \halfcircle{(0,0)}{(1,0)}; \draw (1,0) node[above] {$X$} (0,0) node[above] {$X^{*}$}; },\quad
coev_{X}%
=\tikz[baseline=(current bounding box.west),scale=0.7,samples=100,thick] { \halfcircle{(1,0)}{(0,0)}; \draw (0,0) node[below] {$X$} (1,0) node[below] {$X^{*}$}; },\quad
ev_{X}^{\prime}%
=\tikz[baseline=(current bounding box.west),scale=0.7,samples=100,thick] { \halfcircle{(0,0)}{(1,0)}; \draw (1,0) node[above] {${^{*}\hspace{-0.5ex}X}$} (0,0) node[above] {$X$}; },\quad
coev_{X}^{\prime}%
=\tikz[baseline=(current bounding box.west),scale=0.7,samples=100,thick] { \halfcircle{(1,0)}{(0,0)}; \draw (0,0) node[below] {${^{*}\hspace{-0.5ex}X}$} (1,0) node[below] {$X$}; }.
\]

If $\mathcal{C}$ is also braided, the braiding $c$ and its inverse $c^{-1}$
are denoted respectively by
\[
c_{X,Y}%
=\tikz[baseline=(current bounding box.west),scale=0.7,samples=100,thick] {\drawcrossing[crosstyle=mn,inlefthandlen=0,inrighthandlen=0,outlefthandlen=0,outrighthandlen=0,yscale=.7]{(0,0)}{(1,0)}; \draw (M1) node[above] {$X$} (N1) node[above] {$Y$} (M2) node[below] {$X$} (N2) node[below] {$Y$};},\quad
c_{X,Y}^{-1}%
=\tikz[baseline=(current bounding box.west),scale=0.7,samples=100,thick] {\drawcrossing[crosstyle=nm,inlefthandlen=0,inrighthandlen=0,outlefthandlen=0,outrighthandlen=0,yscale=.7]{(0,0)}{(1,0)};\draw (M1) node[above] {$Y$} (N1) node[above] {$X$} (M2) node[below] {$Y$} (N2) node[below] {$X$};}.
\]
If $B$ is a Hopf algebra in $\mathcal{C}$, we denote its multiplication
$m_{B}$, unit $u_{B}$, comultiplication $\Delta_{B}$, counit $\varepsilon_{B}%
$, antipode $S_{B}$ and the inverse $S_{B}^{-1}$ (if it exists) as follows:
\[
m_{B}%
=\tikz[baseline=(current bounding box.west),scale=1,samples=100,thick] { \drawcoop[CoProdStyle=coproduct,leftarmlen=0.7,rightarmlen=0.7,handlelen=.7,yscale=1]{(0.7,0)}{(0,0)}; \draw (T) node[below] {$B$} (M1) node[above] {$B$} (N1) node[above] {$B$}; },\quad
u_{B}%
=\tikz[baseline=(current bounding box.west),scale=1,samples=100,thick] { \linewithtext[text={{circle/0/\text{\normalsize $u$}}}]{(0,.8)}{(0,0)}; \draw (0,0) node[below] {$B$}; },\quad
\Delta_{B}%
=\tikz[baseline=(current bounding box.west),scale=1,samples=100,thick] { \drawcoop[CoProdStyle=coproduct,leftarmlen=0.7,rightarmlen=0.7,handlelen=.7,yscale=1]{(0,0)}{(0.7,0)}; \draw (T) node[above] {$B$} (M1) node[below] {$B$} (N1) node[below] {$B$}; },\quad
\varepsilon_{B}%
=\tikz[baseline=(current bounding box.west),scale=1,samples=100,thick] { \linewithtext[text={{circle/1/\text{\normalsize $\varepsilon$}}}]{(0,.8)}{(0,0)}; \draw (0,0.8) node[above] {$B$}; },\quad
S_{B}%
=\tikz[baseline=(current bounding box.west),scale=.8,samples=100,thick] { \linewithtext[text={{circle/.5/\text{\tiny{$+$}}}}]{(0,1)}{(0,0)}; \draw (0,1) node[above] {$B$} (0,0) node[below] {$B$}; },\quad
S_{B}^{-1}%
=\tikz[baseline=(current bounding box.west),scale=.8,samples=100,thick] { \linewithtext[text={{circle/.5/\text{\tiny{$-$}}}}]{(0,1)}{(0,0)}; \draw (0,1) node[above] {$B$} (0,0) node[below] {$B$}; }.
\]
\label{graphical notation}

\subsection{Module Categories}

The Morita theory of module categories over a monoidal category was well
developed by Ostrik and Etingof. For references, one can see
\cite{Ostrik2003module,Etingof2015tensor}.

A left module category over a monoidal category $\mathcal{C}$ is a category
$\mathcal{M}$ endowed with an action bifunctor $\otimes:\mathcal{C}%
\times\mathcal{M}\rightarrow\mathcal{M}$, an associativity constraint
$a_{X,Y,M}:\left(  X\otimes Y\right)  \otimes M\rightarrow X\otimes\left(
Y\otimes M\right)  $ and a functorial unit isomorphism $\ell_{M}:1\otimes
M\rightarrow M$, for $X,Y\in\mathcal{C}$, $M\in\mathcal{M}$, satisfying a
pentagon axiom and a triangle axiom.

Similarly, one can define the notion of right module category over
$\mathcal{C}$. Denote the opposite monoidal category of $\mathcal{C}$ by
$\mathcal{C}^{op}$, which is the category $\mathcal{C}$ with reversed order of
tensor product and inverted associativity isomorphism. Then a right
$\mathcal{C}$-module category is a left module category over $\mathcal{C}%
^{op}$.

In the case that $\mathcal{C}$ is a multitensor category, we are interested in
module categories over $\mathcal{C}$ with additional properties in the sense
of \cite[Definition 7.3.1]{Etingof2015tensor}. That is, if we say
$\mathcal{M}$ is a left module category over $\mathcal{C}$, we mean that
$\mathcal{M}$ is a locally finite abelian category equipped with a structure
of a left $\mathcal{C}$-module category, such that the module product
bifunctor $\otimes:\mathcal{C\times M}\rightarrow\mathcal{M}$ is bilinear on
morphisms and exact in the first variable.

In 2003, Ostrik \cite{Ostrik2003module} characterized semisimple
indecomposable module categories over a fusion category $\mathcal{C}$. Later,
Etingof and Ostrik \cite{Etingof2004finite} generalized that result to
nonsemisimple case.

In the study of the structure of a module category $\mathcal{M}$ over a
multitensor category $\mathcal{C}$, a basic tool is the internal Hom. We first
recall this notion here. For objects $M_{1},M_{2},M_{3}$ of $\mathcal{M}$, the
\textit{internal} Hom of $M_{1}$ and $M_{2}$ is an object $\underline
{\operatorname{Hom}}\left(  M_{1},M_{2}\right)  $ of $\mathcal{C}$
representing the contravariant functor $X\mapsto\operatorname{Hom}%
_{\mathcal{M}}\left(  X\otimes M_{1},M_{2}\right)  :\mathcal{C}\rightarrow
\mathrm{Vec}_{k}$, i.e., there exists a natural isomorphism%

\begin{equation}
\eta_{\bullet,M_{1},M_{2}}:\operatorname{Hom}_{\mathcal{M}}\left(
\bullet\otimes M_{1},M_{2}\right)  \overset{\cong}{\longrightarrow
}\operatorname{Hom}_{\mathcal{C}}\left(  \bullet,\underline{\operatorname{Hom}%
}\left(  M_{1},M_{2}\right)  \right)  .\label{def-in-hom}%
\end{equation}
The evaluation morphism $ev_{M_{1},M_{2}}=\eta^{-1}\left(  id_{\underline
{\operatorname{Hom}}\left(  M_{1},M_{2}\right)  }\right)  :\underline
{\operatorname{Hom}}\left(  M_{1},M_{2}\right)  \otimes M_{1}\rightarrow
M_{2}$ is obtained from the isomorphism
\[
\operatorname{Hom}_{\mathcal{C}}\left(  \underline{\operatorname{Hom}}\left(
M_{1},M_{2}\right)  ,\underline{\operatorname{Hom}}\left(  M_{1},M_{2}\right)
\right)  \overset{\cong}{\longrightarrow}\operatorname{Hom}_{\mathcal{M}%
}\left(  \underline{\operatorname{Hom}}\left(  M_{1},M_{2}\right)  \otimes
M_{1},M_{2}\right)  .
\]
The multiplication (composition)
\[
\mu_{M_{1},M_{2},M_{3}}:\underline{\operatorname{Hom}}\left(  M_{2}%
,M_{3}\right)  \otimes\underline{\operatorname{Hom}}\left(  M_{1}%
,M_{2}\right)  \rightarrow\underline{\operatorname{Hom}}\left(  M_{1}%
,M_{3}\right)
\]
is defined as the image of the morphism%

\[
ev_{M_{2},M_{3}}\left(  id\otimes ev_{M_{1},M_{2}}\right)  a_{\underline
{\operatorname{Hom}}\left(  M_{2},M_{3}\right)  ,\underline{\operatorname{Hom}%
}\left(  M_{1},M_{2}\right)  ,M_{1}}%
\]
under the isomorphism
\begin{align*}
&  \operatorname{Hom}_{\mathcal{M}}\left(  \left(  \underline
{\operatorname{Hom}}\left(  M_{2},M_{3}\right)  \otimes\underline
{\operatorname{Hom}}\left(  M_{1},M_{2}\right)  \right)  \otimes M_{1}%
,M_{3}\right) \\
&  \overset{\cong}{\longrightarrow}\operatorname{Hom}_{\mathcal{C}}\left(
\underline{\operatorname{Hom}}\left(  M_{2},M_{3}\right)  \otimes
\underline{\operatorname{Hom}}\left(  M_{1},M_{2}\right)  ,\underline
{\operatorname{Hom}}\left(  M_{1},M_{3}\right)  \right)  .
\end{align*}
Then $A=\left(  \underline{\operatorname{Hom}}\left(  M_{1},M_{1}\right)
,\mu_{M_{1},M_{1},M_{1}}\right)  $ is an algebra in $\mathcal{C}$ with unit
morphism $u_{M_{1}}:1\rightarrow\underline{\operatorname{Hom}}\left(
M_{1},M_{1}\right)  $ obtained from the isomorphism $\operatorname{Hom}%
_{\mathcal{M}}\left(  M_{1},M_{1}\right)  \overset{\cong}{\longrightarrow
}\operatorname{Hom}_{\mathcal{C}}\left(  1,\underline{\operatorname{Hom}%
}\left(  M_{1},M_{1}\right)  \right)  $, and $\left(  \underline
{\operatorname{Hom}}\left(  M_{1},M_{2}\right)  ,\mu_{M_{1},M_{1},M_{2}%
}\right)  $ is a natural right $A$-module in $\mathcal{C}$.

\begin{theorem}
[\cite{Ostrik2003module,Etingof2004finite,{Etingof2015tensor}}]\label{thm
Ostrik}Let $\mathcal{M}$ be a semisimple module category over a finite
multitensor category $\mathcal{C}$. If $M\in\mathcal{M}$ is a generator, then
$A=\underline{\operatorname{Hom}}\left(  M,M\right)  $ is a semisimple algebra
in $\mathcal{C}$. The functor $F=\underline{\operatorname{Hom}}\left(
M,\bullet\right)  :\mathcal{M}\rightarrow\mathrm{Mod}_{\mathcal{C}}$-$A$ given
by $V\mapsto\underline{\operatorname{Hom}}\left(  M,V\right)  $ is an
equivalence of $\mathcal{C}$-module categories.

If assume further that $\mathcal{M}$ is indecomposable, then every nonzero
object $M$ generates $\mathcal{M}$, and the functor $F=\underline
{\operatorname{Hom}}\left(  M,\bullet\right)  :\mathcal{M}\rightarrow
\mathrm{Mod}_{\mathcal{C}}$-$A$ is an equivalence of $\mathcal{C}$-module categories.
\end{theorem}

Let $\left(  A,m,u\right)  $ be a $\mathcal{C}$-algebra. In
\cite{LiuZhu2019On}, the authors defined left $A$-modules in $\mathcal{M}$, by
using the module category tensor, and give the $A$-tensor product of a right
$A$-module $\left(  U,q\right)  $ in $\mathcal{C}$ and a left $A$-module
$\left(  M,p\right)  $ in $\mathcal{M}$. Explicitly, a left $A$-module in
$\mathcal{M}$ is a pair $\left(  M,p\right)  $, where $M$ is an object of
$\mathcal{M}$ and $p:A\otimes M\rightarrow M$ is a morphism (in $\mathcal{M}$)
satisfying two natural axioms,
\[
p\left(  m\otimes id_{M}\right)  =p\left(  id_{A}\otimes p\right)
a_{A,A,M},\quad p\left(  u\otimes id_{M}\right)  =id_{M},
\]
where $a$ is the associativity constraint for $\mathcal{M}$. For right
$A$-module $\left(  U,q\right)  $ in $\mathcal{C}$ and left $A$-module
$\left(  M,p\right)  $ in $\mathcal{M}$, the tensor product $U\otimes_{A}M$ is
the co-equalizer of the morphisms
\[
(U\otimes A)\otimes
M\;\tikz[baseline=-.3ex] {\draw[->] (0,.8ex) -- node[above]{$q\otimes id_M$}(3cm,0.8ex); \draw[->] (0,0ex) -- node[below]{$(id_U\otimes p) a_{U,A,M}$}(3cm,0ex);}\;U\otimes
M\longrightarrow U\otimes_{A}M,
\]
i.e., the cokernel of the morphism $q\otimes id_{M}-\left(  id_{U}\otimes
p\right)  a_{U,A,M}$.

With all these terms, the authors presented a quasi-inverse for the
equivalence $F=\underline{\operatorname{Hom}}\left(  M,\bullet\right)
:\mathcal{M}\rightarrow\mathrm{Mod}_{\mathcal{C}}$-$A$ given in
Theorem~\ref{thm Ostrik}.

\begin{theorem}
[{\cite[Theorem 4.3]{LiuZhu2019On}}]\label{thm liu zhu}Let $\mathcal{M}$ be a
semisimple module category over a finite multitensor category $\mathcal{C}$.
Let $M$ be a generator of $\mathcal{M}$, then the functor $G=\bullet
\otimes_{A}M:\mathrm{Mod}_{\mathcal{C}}$-$A\rightarrow\mathcal{M}$ is a
quasi-inverse to the equivalence $F=\underline{\operatorname{Hom}}\left(
M,\bullet\right)  :\mathcal{M}\rightarrow\mathrm{Mod}_{\mathcal{C}}$-$A$.
\end{theorem}

\subsection{The Drinfeld Center}

Recall that the left Drinfeld center (left center) of a monoidal category
$\mathcal{C}=\left(  \mathcal{C},\otimes,1\right)  $ is a category
$\mathcal{Z}_{l}\left(  \mathcal{C}\right)  $. An object of $\mathcal{Z}%
_{l}\left(  \mathcal{C}\right)  $ is a pair $\left(  Z,\gamma_{Z,\bullet
}\right)  $ consisting of an object $Z\in\mathcal{C}$ and a natural
isomorphism $\gamma_{Z,X}:Z\otimes X\rightarrow X\otimes Z$, $X\in\mathcal{C}%
$, such that
\begin{equation}
\gamma_{Z,X\otimes Y}=\left(  id_{X}\otimes\gamma_{Z,Y}\right)  \left(
\gamma_{Z,X}\otimes id_{Y}\right)  ,\quad\text{for all }X,Y\in\mathcal{C}%
\text{.}\label{condition for Z(C)}%
\end{equation}
A morphism from $\left(  Z,\gamma_{Z,\bullet}\right)  $ to $\left(  Z^{\prime
},\gamma_{Z^{\prime},\bullet}\right)  $ is a morphism $f\in\operatorname{Hom}%
_{\mathcal{C}}\left(  Z,Z^{\prime}\right)  $ such that
\[
\left(  id_{X}\otimes f\right)  \gamma_{Z,X}=\gamma_{Z^{\prime},X}\left(
f\otimes id_{X}\right)  ,
\]
for all $X\in\mathcal{C}$. The right center $\mathcal{Z}_{r}\left(
\mathcal{C}\right)  $ of $\mathcal{C}$ is a similar category with reversed
order of tensor product in its definition.

For any $Z\in\mathcal{Z}_{r}\left(  \mathcal{C}\right)  $, the objects
$Z\otimes1$ and $1\otimes Z$ are identified with $Z$. Then (\ref{condition for
Z(C)}) implies that $\gamma_{Z,1}=\gamma_{Z,1\otimes1}=\left(  \gamma
_{Z,1}\right)  ^{2}$. Hence, one has
\[
\gamma_{Z,1}=id_{Z}.
\]

The center $Z_{l}\left(  \mathcal{C}\right)  $ is a braided monoidal category
with braiding $c$ given by $c_{Z,Z^{\prime}}=\gamma_{Z,Z^{\prime}}$. Also,
$Z_{r}\left(  \mathcal{C}\right)  $ is a braided monoidal category, which is
isomorphic to $Z_{l}\left(  \mathcal{C}\right)  $ with the inverse braiding.

For the basic theory of the Drinfeld center, we refer to
\cite{Joyal1991Tortile,Kassel1995Quantum}.

By definition, an object of $Z_{l}\left(  \mathcal{C}\right)  $ is an object
in $\mathcal{C}$ with a natural isomorphism satisfying (\ref{condition for
Z(C)}). In fact if every object of $\mathcal{C}$ has a right dual, then
identity (\ref{condition for Z(C)}) is sufficient for a natural transformation
$\gamma_{Z,\bullet}$ to be a natural isomorphism.

\begin{lemma}
\label{lemma gama is iso}Let $\gamma_{Z,\bullet}:Z\otimes\bullet
\rightarrow\bullet\otimes Z$ be a natural transformation satisfying
\[
\gamma_{Z,X\otimes Y}=\left(  id_{X}\otimes\gamma_{Z,Y}\right)  \left(
\gamma_{Z,X}\otimes id_{Y}\right)  ,\quad\text{for all }X,Y\in\mathcal{C}%
\text{.}%
\]
If $X\in\mathcal{C}$ has a right dual $^{\ast}X$, then
\begin{align}
\left(  ev_{X}^{\prime}\otimes id_{Z}\right)  \left(  id_{X}\otimes
\gamma_{Z,X^{\ast}}\right)  \left(  \gamma_{Z,X}\otimes id_{^{\ast}X}\right)
& =id_{Z}\otimes ev_{X}^{\prime},\label{eq braided&ev}\\
\left(  id_{^{\ast}X}\otimes\gamma_{Z,X}\right)  \left(  \gamma_{Z,^{\ast}%
X}\otimes id_{X}\right)  \left(  id_{Z}\otimes coev_{X}^{\prime}\right)   &
=coev_{X}^{\prime}\otimes id_{Z}.\label{eq braided&coev}%
\end{align}
\ Moreover, if every object of $\mathcal{C}$ has a right dual, then
$\gamma_{Z,\bullet}$ is a natural isomorphism.

\begin{proof}
By the naturality of $\gamma_{Z,\bullet}$, the equalities (\ref{eq
braided&ev}) and (\ref{eq braided&coev}) are hold. If we denote $\gamma
_{Z,X}%
=\tikz[baseline=(current bounding box.west),scale=0.7,samples=100,thick,font=\scriptsize] {
	\drawcrossing[crosstyle=mn,inlefthandlen=0,inrighthandlen=0,outlefthandlen=0,outrighthandlen=0,yscale=0.7,decoration=circle]{(0,0)}{(1,0)};
	\draw (N1) node[above] {$X$} (M1) node[above] {$Z$} (N2) node[below] {$X$} (M2) node[below] {$Z$};
}$, then the pictorial transcriptions of (\ref{eq braided&ev}) and (\ref{eq
braided&coev}) are respectively%
\[
\tikz[baseline=(current bounding box.west),scale=0.5,samples=100,thick,font=\scriptsize] { \drawcrossing[crosstyle=mn,inlefthandlen=0,inrighthandlen=0,outlefthandlen=0.7,outrighthandlen=0.8,yscale=0.8,decoration=circle]{(0,0)}{(1,0)};
\draw (M1) node[above] {$Z$} (N1) node[above] {$X$}; 
\coordinate (*X1) at ($(M2)+(1,0)$); 
\coordinate (X1) at (N2); 
\coordinate (X2) at ($(M2)+(2,0)$); 
\drawcrossing[crosstyle=mn,inlefthandlen=0,inrighthandlen=1.6,outlefthandlen=0.05,outrighthandlen=0.7,yscale=0.7,decoration=circle]{(M2)}{(*X1)}; 
\coordinate (X3) at (X1 |- N2); 
\draw (X1)--(X3) (N1)node[above]{$^{\ast }X$}; 
\halfcircle{(X3)}{(N2)}; 
\draw(M2) node[below] {$Z$}; }
=\tikz[baseline=(current bounding box.west),scale=0.5,samples=100,thick,font=\scriptsize] { 
\draw (0,0)node[above]{$X$}--(0,-1) (-1,0) node[above]{$Z$} --(-1,-2)node[below]{$Z$}; 
\halfcircle {(0,-1)}{(1,-1)}; 
\draw (1,-1)--(1,0)node[above]{$^{\ast }X$}; },\quad
\tikz[baseline=(current bounding box.west),scale=0.5,samples=100,thick,font=\scriptsize] { 
\coordinate (*X1) at ($(1,0)+(1,0)$); 
\coordinate (X1) at (0,0.7); 
\coordinate (X2) at ($(1,0)+(2,0)$); 
\halfcircle{(X2)}{(*X1)}; 
\drawcrossing[crosstyle=mn,inlefthandlen=0.7,inrighthandlen=0,outlefthandlen=1.45,outrighthandlen=0.7,yscale=0.7,decoration=circle]{(1,0)}{(*X1)}; 
\coordinate (X3) at (X1 |- N2); 
\draw (N2) node[below] {${}^*X$} (M1) node[above] {$Z$}; 
\coordinate (X4) at (X2 |- M2); 
\draw (X2)--(X4); 
\drawcrossing[crosstyle=mn,inlefthandlen=0.5,inrighthandlen=0.5,outlefthandlen=0,outrighthandlen=0,yscale=0.7,decoration=circle]{(M2)}{(X4)}; 
\draw (N2) node [below] {$X$} (M2) node [below] {$Z$}; }=\tikz[baseline=(current bounding box.west),scale=0.5,samples=100,thick,font=\scriptsize] { 
\halfcircle[]{(2,-1)}{(1,-1)}; 
\draw (2,-1)--(2,-2)node[below]{$X$}; 
\draw (1,-1)--(1,-2)node[below]{${}^*X$} (3,0)node[above]{$Z$} --(3,-2)node[below]{$Z$};}.
\]
In addition, the inverse of $\gamma_{Z,X}$ is given by%
\[
\gamma_{Z,X}^{-1}=\left(  ev_{X}^{\prime}\otimes id_{Z}\otimes id_{X}\right)
\left(  id_{X}\otimes\gamma_{Z,^{\ast}X}\otimes id_{X}\right)  \left(
id_{X}\otimes id_{Z}\otimes coev_{X}^{\prime}\right)  .
\]
The compositions of $\gamma_{Z,X}$ and $\gamma_{Z,X}^{-1}$ are computed as
follows:%
\[
\tikz[baseline=(current bounding box.west),scale=0.5,samples=100,thick,font=\scriptsize] { 
\drawcrossing[crosstyle=mn,inlefthandlen=0,inrighthandlen=0,outlefthandlen=0.7,outrighthandlen=0.8,yscale=0.8,decoration=circle]{(0,0)}{(1,0)}; 
\draw (M1) node[above] {$Z$} (N1) node[above] {$X$}; \coordinate (*X1) at ($(M2)+(1,0)$); \coordinate (X1) at (N2); \coordinate (X2) at ($(M2)+(2,0)$); 
\halfcircle{(X2)}{(*X1)}; 
\drawcrossing[crosstyle=mn,inlefthandlen=0,inrighthandlen=0,outlefthandlen=0.05,outrighthandlen=0.7,yscale=0.7,decoration=circle]{(M2)}{(*X1)}; 
\coordinate (X3) at (X1 |- N2); \draw (X1)--(X3); 
\halfcircle{(X3)}{(N2)}; \coordinate (X4) at (X2 |- M2); 
\draw (X2)--(X4) node[below] {$X$} (M2) node[below] {$Z$}; 
\draw[thin,dashed]($(X1)+(-.2,.5)$) rectangle ($(X4)+(0.2,0.05)$); }=\tikz[baseline=(current bounding box.west),scale=0.5,samples=100,thick,font=\scriptsize] { 
\draw (0,0)node[above]{$X$}--(0,-1) (-1,0) node[above]{$Z$} --(-1,-2)node[below]{$Z$}; 
\halfcircle {(0,-1) }{(1,-1)}; 
\halfcircle {(2,-1) }{(1,-1)}; \draw (2,-1)--(2,-2)node[below]{$X$}; }=id_{Z\otimes
X},\quad
\tikz[baseline=(current bounding box.west),scale=0.5,samples=100,thick,font=\scriptsize] { 
\coordinate (*X1) at ($(1,0)+(1,0)$); 
\coordinate (X1) at (0,0.7); \coordinate (X2) at ($(1,0)+(2,0)$); \halfcircle{(X2)}{(*X1)}; 
\drawcrossing[crosstyle=mn,inlefthandlen=0.7,inrighthandlen=0,outlefthandlen=0.05,outrighthandlen=0.7,yscale=0.7,decoration=circle]{(1,0)}{(*X1)}; 
\coordinate (X3) at (X1 |- N2); \draw (X1) node[above] {$X$}--(X3) (M1) node[above] {$Z$}; 
\halfcircle{(X3)}{(N2)}; \coordinate (X4) at (X2 |- M2); \draw (X2)--(X4); 
\drawcrossing[crosstyle=mn,inlefthandlen=0.5,inrighthandlen=0.5,outlefthandlen=0,outrighthandlen=0,yscale=0.7,decoration=circle]{(M2)}{(X4)}; 
\draw (N2) node [below] {$X$} (M2) node [below] {$Z$}; \draw[thin,dashed]($(X1)+(-.2,-.1)$) rectangle ($(N1)+(0.2,-0.45)$); }
=\tikz[baseline=(current bounding box.west),scale=0.5,samples=100,thick,font=\scriptsize] { 
\halfcircle {(0,-1) }{(1,-1)}; 
\halfcircle {(2,-1) }{(1,-1)}; 
\draw (2,-1)--(2,-2)node[below]{$X$}; \draw (0,0)node[above]{$X$}--(0,-1) (3,0) node[above]{$Z$} --(3,-2)node[below]{$Z$};}=id_{X\otimes
Z},
\]
where $\gamma_{Z,X}^{-1}$ is represented by the morphism in the dashed box.
\end{proof}
\end{lemma}

Let $H$ be a finite dimensional Hopf algebra over $k$, and $\mathcal{C}={}%
_{H}\mathcal{M}$ be the category of finite dimensional left $H$-modules. Then
the center $Z_{l}\left(  _{H}\mathcal{M}\right)  $ of the monoidal category
$_{H}\mathcal{M}$ is isomorphic to the Yetter-Drinfeld~category ${}_{H}%
^{H}\mathcal{YD}$ over $H$.

\section{The Center of Braided Rigid Categories}

\label{Section Center of BRC}

In this section, $\mathcal{C}$ will be a braided rigid category with a
braiding $c$. We will show that under some representability assumption the
Drinfeld center $\mathcal{Z}_{l}\left(  \mathcal{C}\right)  $ of $\mathcal{C}$
is equivalent to the left comodule category of the automorphism braided group
$U\left(  \mathcal{C}\right)  $ of $\mathcal{C}$, as right $\mathcal{C}%
$-module categories.

Firstly, let's recall Majid's reconstruction
\cite{Majid1991Reconstruction,Majid1991Braided,Majid1995foundations} of the
automorphism braided group $U\left(  \mathcal{C}\right)  $. We require the
\textbf{representability assumption for modules} \cite[\S \ 9.4]%
{Majid1995foundations}. That is, there exist an object $U\left(
\mathcal{C}\right)  \in\mathcal{C}$, and a natural isomorphism
\[
\theta_{V}:\operatorname{Hom}_{\mathcal{C}}\left(  V,U\left(  \mathcal{C}%
\right)  \right)  \rightarrow\operatorname{Nat}\left(  V\otimes
id_{\mathcal{C}},id_{\mathcal{C}}\right)
\]
for any $V\in\mathcal{C}$, and the maps
\begin{align*}
\theta_{V}^{2}  & :\operatorname{Hom}_{\mathcal{C}}\left(  V,U\left(
\mathcal{C}\right)  \otimes U\left(  \mathcal{C}\right)  \right)
\rightarrow\operatorname{Nat}\left(  V\otimes id_{\mathcal{C}}{}^{\otimes
2},id_{\mathcal{C}}{}^{\otimes2}\right)  ,\\
\theta_{V}^{3}  & :\operatorname{Hom}_{\mathcal{C}}\left(  V,U\left(
\mathcal{C}\right)  \otimes U\left(  \mathcal{C}\right)  \otimes U\left(
\mathcal{C}\right)  \right)  \rightarrow\operatorname{Nat}\left(  V\otimes
id_{\mathcal{C}}{}^{\otimes3},id_{\mathcal{C}}{}^{\otimes3}\right)  ,
\end{align*}
induced by $\alpha=\theta_{U\left(  \mathcal{C}\right)  }\left(  id_{U\left(
\mathcal{C}\right)  }\right)  $ and the braiding $c$, are bijective. With
graphical convention, we denote
\[
\alpha_{X}%
=\tikz[baseline=(current bounding box.west),yscale=0.8,samples=100,thick,font=\scriptsize] { 
\drawcoop[CoProdStyle=comodule,leftarmlen=0.7,rightarmlen=0.7,handlelen=1.5,yscale=0.8]{(0.8,0)}{(0,0)};
 \draw (T) node[below] {$X$} (M1) node[above] {$X$} (N1) node[above] {$U\left( \mathcal{C}\right)$} ($(T)+(0.3,0.3)$) node {$\alpha_{X}$}; }\text{,
or simply }\alpha_{X}%
=\tikz[baseline=(current bounding box.west),yscale=0.8,samples=100,thick,font=\scriptsize] { 
\drawcoop[CoProdStyle=comodule,leftarmlen=0.7,rightarmlen=0.7,handlelen=1.5,yscale=0.8]{(0.8,0)}{(0,0)};
 \draw (T) node[below] {$X$} (M1) node[above] {$X$} (N1) node[above] {$U\left( \mathcal{C}\right)$}; }\text{,
for }X\in\mathcal{C}\text{.}%
\]
Then $\theta_{V}^{2}$ and $\theta_{V}^{3}$ can be expressed graphically as
follows. For $X,Y,Z\in\mathcal{C}$, $t\in\operatorname{Hom}_{\mathcal{C}%
}\left(  V,U\left(  \mathcal{C}\right)  \otimes U\left(  \mathcal{C}\right)
\right)  $ and $s\in\operatorname{Hom}_{\mathcal{C}}\left(  V,U\left(
\mathcal{C}\right)  \otimes U\left(  \mathcal{C}\right)  \otimes U\left(
\mathcal{C}\right)  \right)  $,%

\begin{equation}
\theta_{V}^{2}\left(  t\right)  _{X,Y}%
=\tikz[baseline=(current bounding box.west),scale=1,samples=100,thick,font=\scriptsize] { 
\drawcoop[CoProdStyle=comodule,leftarmlen=0.9,rightarmlen=0.7,handlelen=1,yscale=.5,leftarmpt=P]{(0,0)}{(0.5,0)};  
\draw (T) node[above] {$V$}; 
\drawcrossing[crosstyle=mn,inlefthandlen=0,inrighthandlen=1.2,outlefthandlen=0.2,outrighthandlen=0.2]{(N1)}{($(N1)+(.5,0)$)}; \draw (N1) node[above] {$X$}; \drawcoop[CoProdStyle=comodule,leftarmlen=0.1,rightarmlen=1,handlelen=.7,yscale=.5]{(N2)}{(P |- N2)}; \draw (T) node[below] {$X$};  \drawcoop[CoProdStyle=comodule,leftarmlen=2.4,rightarmlen=0.1,handlelen=.7,yscale=.5]{($(M2)+(.5,0)$)}{(M2)}; \draw (M1) node[above] {$Y$} (T) node[below] {$Y$}; \draw[fill=white,thin] (-0.1,0.1) rectangle(0.6,-0.3); \draw ($(-0.1,0.1)!0.5!(0.6,-0.3)$) node {$t$}; },\quad
\theta_{V}^{3}\left(  s\right)  _{X,Y,Z}%
=\tikz[baseline=(current bounding box.west),scale=1,samples=100,thick,font=\scriptsize] { \drawcrossing[crosstyle=mn,inlefthandlen=0,inrighthandlen=1.1,outlefthandlen=0.1,outrighthandlen=0.1]{(0,0)}{(.5,0)}; \draw (N1) node[above] {$X$}; \coordinate (A) at (N2); \drawcrossing[crosstyle=mn,inlefthandlen=0,inrighthandlen=2.2,outlefthandlen=0.2,outrighthandlen=0.2]{(M2)}{($(M2)+(.5,0)$)}; \draw (N1) node[above] {$Y$}; \drawcoop[CoProdStyle=comodule,leftarmlen=3.4,rightarmlen=0.1,handlelen=.7,yscale=.5]{($(M2)+(.5,0)$)}{(M2)}; \draw (T) node[below] {$Z$}; \draw (M1) node[above] {$Z$}; \drawcrossing[crosstyle=mn,inlefthandlen=2.2,inrighthandlen=0,outlefthandlen=0.2,outrighthandlen=0.2]{($(A)+(-.5,0)$)}{(A)}; \draw (M1) node[above] {$V$}; \drawcoop[CoProdStyle=comodule,leftarmlen=0.1,rightarmlen=0.1,handlelen=.7,yscale=.5]{($(M2)+(.5,0)$)}{(M2)}; \draw (T) node[below] {$Y$}; \drawcoop[CoProdStyle=comodule,leftarmlen=0.1,rightarmlen=3,handlelen=.7,yscale=.5]{(N2)}{($(N2)+(-.5,0)$)}; \draw (T) node[below] {$X$}; \coordinate (B) at ($(N1)+(-0.1,0)$); \draw[fill=white,thin] (B|- 0,0.4) rectangle(0.1,0); \draw ($(B|- 0,0.4)!0.5!(0.1,0)$) node {$s$}; }.\label{eq deftheta}%
\end{equation}
Then $U\left(  \mathcal{C}\right)  $ is a Hopf algebra in the braided monoidal
category $\mathcal{C}$, named the automorphism braided group of $\mathcal{C}$,
which acts canonically on every object $X\in\mathcal{C}$ via $\alpha_{X}$.
Write $B=U\left(  \mathcal{C}\right)  $. Then with the graphical notations
(see Page \pageref{graphical notation}) the Hopf algebra structure on $B$ is
determined by the diagrams (see \cite[\S \ 9.4]{Majid1995foundations})
\begin{align}
& \theta_{B\otimes B}\left(  m_{B}\right)  _{X}%
=\tikz[baseline=(current bounding box.west),scale=1,samples=100,thick,font=\scriptsize] { \drawcoop[CoProdStyle=coproduct,leftarmlen=0.84,rightarmlen=0.84,handlelen=.7,yscale=.5, labels={\leftArmPt/above/B,\rightArmPt/above/B}]{(0.5,0)}{(0,0)};   \drawcoop[CoProdStyle=comodule,leftarmlen=.6cm,rightarmlen=0.05,handlelen=.7,yscale=.5, labels={\leftArmPt/above/X}]{($(T)+(.7,0)$)}{(T)}; \draw (T) node[below] {$X$}; }=\tikz[baseline=(current bounding box.west),scale=1,samples=100,thick,font=\scriptsize] { \drawcoop[CoProdStyle=copmodule,leftarmlen=0.84,rightarmlen=0.84,handlelen=.7,yscale=.5, labels={\leftArmPt/above/X,\rightArmPt/above/B}]{(0.5,0)}{(0,0)}; \drawcoop[CoProdStyle=comodule,leftarmlen=.05,rightarmlen=0.6cm,handlelen=.7,yscale=.5, labels={\rightArmPt/above/B}]{(T)}{($(T)-(.7,0)$)}; \draw (T) node[below] {$X$}; },\quad
\theta_{1}\left(  u_{B}\right)  _{X}%
=\tikz[baseline=(current bounding box.west),scale=1,samples=100,thick,font=\scriptsize] { \drawcoop[CoProdStyle=comodule,leftarmlen=1.2,rightarmlen=1,handlelen=.7,yscale=0.5,rightarmtext={{circle/1/\text{\normalsize $u$}}}]{(0.5,0)}{(0,0)};   \draw (M1) node[above] {$X$} (T)node[below]{$X$}; }=\tikz[baseline=(current bounding box.west),scale=1,samples=100,thick,font=\scriptsize] { \draw (0,0.8) node [above]{$X$} -- (0,0)node[below]{$X$};  },\label{Fig1 stru of B}%
\\
& \theta_{B}^{2}\left(  \Delta_{B}\right)  _{X,Y}%
=\tikz[baseline=(current bounding box.west),scale=1,samples=100,thick,font=\scriptsize] { \drawcoop[CoProdStyle=coproduct,leftarmlen=0.6,rightarmlen=0.4,handlelen=.7,yscale=.5,leftarmpt=P]{(0,0)}{(0.5,0)};   \draw (T) node[above] {$B$}; \drawcrossing[crosstyle=mn,inlefthandlen=0,inrighthandlen=.8,outlefthandlen=0.2,outrighthandlen=0.2]{(N1)}{($(N1)+(.5,0)$)}; \draw (N1) node[above] {$X$}; \drawcoop[CoProdStyle=comodule,leftarmlen=0.1,rightarmlen=1,handlelen=.7,yscale=.5]{(N2)}{(P |- N2)}; \draw (T) node[below] {$X$};  \drawcoop[CoProdStyle=comodule,leftarmlen=2,rightarmlen=0.1,handlelen=.7,yscale=.5]{($(M2)+(.5,0)$)}{(M2)}; \draw (M1) node[above] {$Y$} (T) node[below] {$Y$}; }=\tikz[baseline=(current bounding box.west),scale=1,samples=100,thick,font=\scriptsize] { \drawcoop[CoProdStyle=comodule,leftarmlen=1.3,rightarmlen=1.3,handlelen=.6,yscale=1]{(0.7,0)}{(0,0)};   \draw (N1) node[above] {$B$} (M1) node[above] {$X\otimes Y$} (T) node[below]{$X\otimes Y$}; }=\alpha
_{X\otimes Y},\quad\varepsilon_{B}=\alpha_{1},\label{Fig2 stru of B}\\
& \theta_{B}\left(  S_{B}\right)  _{X}%
=\tikz[baseline=(current bounding box.west),scale=1,samples=100,thick,font=\scriptsize] { \drawcoop[CoProdStyle=comodule,leftarmlen=1.3,rightarmlen=1.3,handlelen=.6,yscale=1,rightarmtext={{circle/.5/\text{\tiny{$+$}}}}]{(0.7,0)}{(0,0)};
 \draw (M1) node [above] {$X$} (N1) node [above] {$B$} (T) node[below]{$X$}; }=\tikz[baseline=(current bounding box.west),xscale=0.9,yscale=0.6,samples=100,thick,font=\scriptsize] { \drawcrossing[crosstyle=mn,inlefthandlen=.6,inrighthandlen=.2,outlefthandlen=0.8,outrighthandlen=0.2,yscale=.5]{(0,0)}{(1,0)}; \draw (M1) node[above] {$B$} (N2) node[below] {$X$} ($(N1)+(0.9,0)$) node[right] {$X^{*}$}; \halfcircle{($(N1)+(1,0)$)}{(N1)}; \drawcoop[CoProdStyle=comodule,leftarmlen=.9,rightarmlen=.1,handlelen=.3,yscale=.5]{($(M2)+(1,0)$)}{(M2)}; \draw ($(T)+(0.5,-0.2)$) node {$\alpha_{X^*}$}; \halfcircle{(T)}{($(T)+(1.5,0)$)}; \draw ($(T)+(1.5,0)$) -- ($(T)+(1.5,1.3)$) node[above] {$X$}; }.\label{Fig3 stru of B}%
\end{align}

It is known that if an object $X$ of a braided monoidal category has a left
dual $X^{\ast}$, then $X^{\ast}$ is naturally a right dual of $X$ with
$ev_{X}^{\prime}=ev_{X}\circ c_{X,X^{\ast}}$ and $coev_{X}^{\prime
}=c_{X,X^{\ast}}^{-1}\circ coev_{X}$. We will use the right rigidity of
$\mathcal{C}$ to construct the inverse of $S_{B}$.

\begin{proposition}
The antipode of $B$ is an isomorphism with its inverse $T_{B}%
=\tikz[baseline=(current bounding
box.west),scale=.8,samples=100,thick,font=\scriptsize] {
\linewithtext[text={{circle/.5/\text{\tiny{$-$}}}}]{(0,1)}{(0,0)}; \draw
(0,1) node[above] {$B$} (0,0) node[below] {$B$}; }$ given by
\[
\tikz[baseline=(current bounding box.west),scale=1,samples=100,thick,font=\scriptsize] { \drawcoop[CoProdStyle=comodule,leftarmlen=1.3,rightarmlen=1.3,handlelen=.6,yscale=1,rightarmtext={{circle/.5/\text{\tiny{$-$}}}}]{(0.7,0)}{(0,0)};   \draw (M1) node [above] {$X$} (N1) node [above] {$B$} (T) node[below]{$X$}; }=\tikz[baseline=(current bounding box.west),yscale=1,xscale=1.5,samples=100,thick,font=\scriptsize] { \drawcrossing[crosstyle=mn,inlefthandlen=.3,inrighthandlen=.3,outlefthandlen=0.3,outrighthandlen=0.2,yscale=.5]{(0,0)}{(.6,0)}; \draw (M1) node[above] {$B$} (N1) node[above] {$X$};  \drawcoop[CoProdStyle=comodule,leftarmlen=.7,rightarmlen=.2,handlelen=.1,yscale=.5]{($(M2)+(.6,0)$)}{(M2)};   \halfcircle{(N2)}{(T)}; \draw ($(M1)+(-0.15,0.15)$) node[] {${^{*}\hspace{-0.5ex}X}$} ($(T)+(0.25,-0.2)$) node {$\alpha_{^{*}\hspace{-0.5ex}X}$}; \halfcircle{($(M1)+(.6,0)$)}{(M1)}; \draw ($(M1)+(.6,0)$) -- ($(M1)+(.6,-1)$)node[below]{$X$};}.
\]

\begin{proof}
To show that $T_{B}$ be the inverse of $S_{B}$ is to show
\[
\theta_{B}\left(  S_{B}\circ T_{B}\right)  =\theta_{B}\left(  id_{B}\right)
=\theta_{B}\left(  T_{B}\circ S_{B}\right)  .
\]
The following graphical calculus yields the result:
\begin{align*}
\tikz[baseline=(current bounding box.west),yscale=1.2,samples=100,thick,font=\scriptsize] { \drawcoop[CoProdStyle=comodule,leftarmlen=1.3,rightarmlen=1.3,handlelen=.6,yscale=1,rightarmtext={{circle/.75/\text{\tiny{$+$}}},{circle/.3/\text{\tiny{$-$}}}}]{(0.7,0)}{(0,0)};   \draw (M1) node [above] {$X$} (N1) node [above] {$B$} (T) node[below]{$X$}; }
&
=\tikz[baseline=(current bounding box.west),xscale=1.5,yscale=1.2,samples=100,thick,font=\scriptsize] { \drawcrossing[crosstyle=mn,inlefthandlen=.2,inrighthandlen=.2,outlefthandlen=0.3,outrighthandlen=0.2,yscale=.5]{(0,0)}{(.5,0)}; \linewithtext[text={{circle/.3/\text{\tiny{$+$}}}}]{(M1)}{($(M1)+(0,0.4)$)}; \draw ($(M1)+(0,0.4)$) node[above] {$B$} (N1)--($(N1)+(0,0.4)$) node[above] {$X$};  \drawcoop[CoProdStyle=comodule,leftarmlen=.7,rightarmlen=.2,handlelen=.1,yscale=.5]{($(M2)+(.5,0)$)}{(M2)}; \halfcircle{(N2)}{(T)}; \draw ($(M1)+(-0.1,0.2)$) node[] {${^{*}\hspace{-0.5ex}X}$} ($(T)+(0.2,-0.2)$) node {$\alpha_{^{*}\hspace{-0.5ex}X}$}; \halfcircle{($(M1)+(.5,0)$)}{(M1)}; \draw ($(M1)+(.5,0)$) -- ($(M1)+(.5,-1)$)node[below]{$X$};}=\tikz[baseline=(current bounding box.west),xscale=1.5,yscale=1.2,samples=100,thick,font=\scriptsize] { \drawcrossing[crosstyle=mn,inlefthandlen=.2,inrighthandlen=.2,outlefthandlen=0.9,outrighthandlen=1,yscale=.5]{(0,0)}{(.5,0)}; \draw (M1) node[above] {$B$} (N1) node[above] {$X$};  \drawcoop[CoProdStyle=comodule,leftarmlen=.7,rightarmlen=.8,handlelen=.1,yscale=.5,rightarmtext={{circle/.6/\text{\tiny{$+$}}}}]{($(M2)+(.5,0)$)}{(M2)};   \halfcircle{(N2)}{(T)}; \draw ($(M1)+(-0.1,0.2)$) node[] {${^{*}\hspace{-0.5ex}X}$} ($(T)+(0.2,-0.2)$) node {$\alpha_{^{*}\hspace{-0.5ex}X}$}; \halfcircle{($(M1)+(.5,0)$)}{(M1)}; \draw ($(M1)+(.5,0)$) -- ($(M1)+(.5,-1)$)node[below]{$X$};}\\
&
=\tikz[baseline=(current bounding box.west),xscale=1.5,yscale=1.2,samples=100,thick,font=\scriptsize] { \drawcrossing[crosstyle=mn,inlefthandlen=.2,inrighthandlen=.2,outlefthandlen=2,outrighthandlen=0.8,yscale=.5,outleftnm=P]{(0,0)}{(.5,0)}; \draw (M1) node[above] {$B$} (N1) node[above] {$X$};  \drawcrossing[crosstyle=mn,inlefthandlen=0.5,inrighthandlen=0.2,outlefthandlen=0.8,outrighthandlen=0.2,yscale=.5]{(M2)}{($(M2)+(0.5,0)$)}; \draw ($(N2)+(0.15,0)$) node {${}{^{*}\hspace{-0.5ex}X}$}; \halfcircle{($(N1)+(0.5,0)$)}{(N1)}; \drawcoop[CoProdStyle=comodule,leftarmlen=.9,rightarmlen=.1,handlelen=.3,yscale=.5]{($(M2)+(0.5,0)$)}{(M2)}; \draw ($(T)+(0.25,-0.07)$) node {$\alpha_{X}$} ($(M1)+(0.15,0.07)$)node{$X$}; \halfcircle{(P)}{(N2)}; \halfcircle{(T)}{($(T)+(0.7,0)$)}; \draw ($(T)+(0.7,0)$) -- ($(T)+(0.7,1)$) ($(T)+(0.55,1.1)$) node{${}{^{*}\hspace{-0.5ex}X}$}; \halfcircle{($(T)+(1.2,1)$)}{($(T)+(0.7,1)$)}; \draw ($(T)+(1.2,1)$)-- ($(T)+(1.2,-0.5)$)node[below]{$X$}; }=\tikz[baseline=(current bounding box.west),xscale=1.5,yscale=1.2,samples=100,thick,font=\scriptsize] { \halfcircle {(0,-0.5) }{(0.5,-0.5)}; \halfcircle {(1,-0.5) }{(0.5,-0.5)}; \draw (1,-0.5)--(1,-1); \draw (0,0)node[above]{$X$}--(0,-0.5); \drawcoop[CoProdStyle=comodule,leftarmlen=0,rightarmlen=0.666,handlelen=.3,yscale=.5]{(1,-1)}{(-0.5,-1)}; \halfcircle{(T)}{($(T)+(1.1,0)$)}; \halfcircle{($(T)+(1.7,1)$)}{($(T)+(1.1,1)$)}; \draw (N1)node[above]{$B$} ($(T)+(0.25,0)$)node{$\alpha_{X}$} ($(T)+(1.1,0)$)--($(T)+(1.1,1)$) ($(T)+(1.7,1)$)--($(T)+(1.7,-0.7)$)node[below]{$X$}; }=\tikz[baseline=(current bounding box.west),xscale=1.5,yscale=1.2,samples=100,thick,font=\scriptsize] { \drawcoop[CoProdStyle=comodule,leftarmlen=1.5,rightarmlen=1.5,handlelen=1,yscale=1]{(0.5,0)}{(0,0)};   \draw (M1) node [above] {$X$} (N1) node [above] {$B$} (T) node[below]{$X$}; },
\end{align*}%
\begin{align*}
\tikz[baseline=(current bounding box.west),yscale=1.2,samples=100,thick,font=\scriptsize] { \drawcoop[CoProdStyle=comodule,leftarmlen=1.3,rightarmlen=1.3,handlelen=.6,yscale=1,rightarmtext={{circle/.75/\text{\tiny{$-$}}},{circle/.3/\text{\tiny{$+$}}}}]{(0.7,0)}{(0,0)};   \draw (M1) node [above] {$X$} (N1) node [above] {$B$} (T) node[below]{$X$}; }
&
=\tikz[baseline=(current bounding box.west),yscale=1.2,xscale=1.5,samples=100,thick,font=\scriptsize] { \drawcrossing[crosstyle=mn,inlefthandlen=.2,inrighthandlen=0.5,outlefthandlen=0.9,outrighthandlen=0.2,yscale=.5]{(0,0)}{(0.5,0)}; \linewithtext[text={{circle/.3/\text{\tiny{$-$}}}}]{(M1)}{($(M1)+(0,0.55)$)}; \draw ($(M1)+(0,0.55)$) node[above] {$B$}(N2) node[below] {$X$}; \halfcircle{($(N1)+(0.5,0)$)}{(N1)}; \drawcoop[CoProdStyle=comodule,leftarmlen=1.2,rightarmlen=.1,handlelen=.3,yscale=.5]{($(M2)+(0.5,0)$)}{(M2)}; \halfcircle{(T)}{($(T)+(0.75,0)$)}; \draw ($(T)+(0.75,0)$) -- ($(T)+(0.75,1.1)$) node[above] {$X$} ($(M1)+(0.15,0.2)$) node {$X^{*}$} ($(T)+(0.3,-0.1)$) node {$\alpha_{X^*}$}; }=\tikz[baseline=(current bounding box.west),yscale=1.2,xscale=1.5,samples=100,thick,font=\scriptsize] { \drawcrossing[crosstyle=mn,inlefthandlen=.5,inrighthandlen=0,outlefthandlen=1.8,outrighthandlen=0.9,yscale=.5]{(0,0)}{(0.5,0)}; \draw (M1) node[above] {$B$}(N2) node[below] {$X$}; \halfcircle{($(N1)+(0.5,0)$)}{(N1)}; \drawcoop[CoProdStyle=comodule,leftarmlen=1.5,rightarmlen=0.8,handlelen=.3,yscale=.5,rightarmtext={{circle/.6/\text{\tiny{$-$}}}}]{($(M2)+(0.5,0)$)}{(M2)}; \halfcircle{(T)}{($(T)+(0.75,0)$)}; \draw ($(T)+(0.75,0)$) -- ($(T)+(0.75,1.1)$) node[above] {$X$} ($(M1)+(0.15,0.2)$) node {$X^{*}$} ($(T)+(0.3,-0.1)$) node {$\alpha_{X^*}$}; }\\
&
=\tikz[baseline=(current bounding box.west),xscale=1.5,yscale=1.2,samples=100,thick,font=\scriptsize] { \drawcrossing[crosstyle=mn,inlefthandlen=.5,inrighthandlen=0,outlefthandlen=2.5,outrighthandlen=0.5,yscale=.5]{(0,0)}{(0.5,0)}; \draw (M1) node[above] {$B$}(N2) node[below] {$X$}; \halfcircle{($(N1)+(0.5,0)$)}{(N1)}; \draw ($(N1)+(0.7,0.15)$) node {$X^{*}$}; \drawcrossing[crosstyle=mn,inlefthandlen=.2,inrighthandlen=1,outlefthandlen=0.3,outrighthandlen=0.2,yscale=.5]{(M2)}{($(M2)+(.5,0)$)}; \drawcoop[CoProdStyle=comodule,leftarmlen=.7,rightarmlen=.2,handlelen=.1,yscale=.5]{($(M2)+(.5,0)$)}{(M2)}; \halfcircle{(N2)}{(T)}; \draw ($(T)+(0.25,-0.13)$) node {$\alpha_{X}$} ($(M1)+(-0.1,0.07)$)node{$X$}; \halfcircle{($(M1)+(.5,0)$)}{(M1)}; \draw ($(M1)+(.5,0)$) -- ($(M1)+(.5,-0.8)$); \halfcircle{($(M1)+(.5,-0.8)$)}{($(M1)+(1,-.8)$)}; \draw ($(M1)+(1,-.8)$) -- ($(M1)+(1,0.8)$) node[above] {$X$} ($(M1)+(.3,-0.8)$)node{$X^{*}$}; }=\tikz[baseline=(current bounding box.west),xscale=1.5,yscale=1.2,samples=100,thick,font=\scriptsize]{ \drawcoop[CoProdStyle=comodule,leftarmlen=.5,rightarmlen=1.7,handlelen=.5,yscale=.5]{(0.5,0)}{(0,0)}; \halfcircle{($(T)+(-.5,0)$)}{(T)}; \halfcircle{($(T)+(-.5,0)$)}{($(T)+(-1,0)$)}; \halfcircle{($(M1)+(.5,0)$)}{(M1)}; \halfcircle{($(M1)+(.5,0)$)}{($(M1)+(1,0)$)}; \draw ($(T)+(-1,0)$)--($(T)+(-1,-0.4)$)node[below]{$X$} (N1)node[above]{$B$} ($(M1)+(1,0)$)--($(M1)+(1,0.6)$)node[above]{$X$} ($(T)+(0.25,-0.03)$)node{$\alpha_{X}$}; }=\tikz[baseline=(current bounding box.west),xscale=1.5,yscale=1.2,samples=100,thick,font=\scriptsize] { \drawcoop[CoProdStyle=comodule,leftarmlen=1.5,rightarmlen=1.5,handlelen=1,yscale=1]{(0.5,0)}{(0,0)};   \draw (M1) node [above] {$X$} (N1) node [above] {$B$} (T) node[below]{$X$}; }.
\end{align*}

\end{proof}
\end{proposition}

For any $M,N\in\mathcal{C}$, we define%
\begin{align*}
\varphi_{M,N}  & :\operatorname{Hom}_{\mathcal{C}}\left(  M,B\otimes N\right)
\rightarrow\operatorname{Nat}\left(  M\otimes id_{\mathcal{C}},id_{\mathcal{C}%
}\otimes N\right)  ,\\
\varphi_{M,N}^{2}  & :\operatorname{Hom}_{\mathcal{C}}\left(  M,B\otimes
B\otimes N\right)  \rightarrow\operatorname{Nat}\left(  M\otimes
id_{\mathcal{C}}{}^{\otimes2},id_{\mathcal{C}}{}^{\otimes2}\otimes N\right)
\end{align*}
via
\begin{equation}
\varphi_{M,N}\left(  t\right)  _{X}%
=\tikz[baseline=(current bounding box.west),scale=1,samples=100,thick,font=\scriptsize] { \drawcoop[CoProdStyle=comodule,leftarmlen=0.9,rightarmlen=0.7,handlelen=1,yscale=.5,leftarmpt=P]{(0,0)}{(0.5,0)};   \draw (T) node[above] {$M$}; \drawcrossing[crosstyle=mn,inlefthandlen=0.4,inrighthandlen=1.2,outlefthandlen=0.2,outrighthandlen=0.55,yscale=.5]{(N1)}{($(N1)+(.5,0)$)}; \draw (N1) node[above] {$X$} (M2) node[below] {$N$}; \drawcoop[CoProdStyle=comodule,leftarmlen=0.1,rightarmlen=1,handlelen=.7,yscale=.5]{(N2)}{(P |- N2)}; \draw (T) node[below] {$X$};  \draw[fill=white,thin] (-0.1,0.1) rectangle(0.6,-0.3); \draw ($(-0.2,0.1)!0.5!(0.7,-0.3)$) node {$t$}; },\quad
\varphi_{M,N}^{2}\left(  s\right)  _{X,Y}%
=\tikz[baseline=(current bounding box.west),scale=1,samples=100,thick,font=\scriptsize] { \drawcrossing[crosstyle=mn,inlefthandlen=0,inrighthandlen=1.08,outlefthandlen=0.1,outrighthandlen=0.1,yscale=.5]{(0,0)}{(.5,0)}; \draw (N1) node[above] {$X$}; \coordinate (A) at (N2); \drawcrossing[crosstyle=mn,inlefthandlen=0,inrighthandlen=1.7,outlefthandlen=0.2,outrighthandlen=0.55,yscale=.5]{(M2)}{($(M2)+(.5,0)$)}; \draw (N1) node[above] {$Y$} (M2) node[below] {$N$}; \drawcrossing[crosstyle=mn,inlefthandlen=1.7,inrighthandlen=0,outlefthandlen=0.2,outrighthandlen=0.2,yscale=.5]{($(A)+(-.5,0)$)}{(A)}; \draw (M1) node[above] {$M$}; \drawcoop[CoProdStyle=comodule,leftarmlen=0.1,rightarmlen=0.1,handlelen=.7,yscale=.5]{($(M2)+(.5,0)$)}{(M2)}; \draw (T) node[below] {$Y$}; \drawcoop[CoProdStyle=comodule,leftarmlen=0.1,rightarmlen=1.8,handlelen=.7,yscale=.5]{(N2)}{($(N2)+(-.5,0)$)}; \draw (T) node[below] {$X$}; \coordinate (B) at ($(N1)+(-0.1,0)$); \draw[fill=white,thin] (B|- 0,0.4) rectangle(0.1,0); \draw ($(B|- 0,0.4)!0.5!(0.1,0)$) node{$s$}; },\ \label{eq defphi}%
\end{equation}
where $X,Y\in\mathcal{C}$, $t\in\operatorname{Hom}_{\mathcal{C}}\left(
M,B\otimes N\right)  $ and $s\in\operatorname{Hom}_{\mathcal{C}}\left(
M,B\otimes B\otimes N\right)  $.

The next two lemmas will show connection between the category $B$%
-$\operatorname*{Comod}_{\mathcal{C}}$ of left $B$-comodules in $\mathcal{C}$
and the center $\mathcal{Z}_{l}\left(  \mathcal{C}\right)  $ of $\mathcal{C}$.

\begin{lemma}
\label{Lemma nat iso}For any $M,N\in\mathcal{C}$, $\varphi_{M,N}$ and
$\varphi_{M,N}^{2}$ are isomorphisms natural in both variables.
\end{lemma}

\begin{proof}
Since $\mathcal{C}$ is rigid, there is a natural isomorphism%
\begin{align}
\operatorname{Hom}_{\mathcal{C}}\left(  M,X\otimes N\right)   & \rightarrow
\operatorname{Hom}_{\mathcal{C}}\left(  M\otimes{}{^{\ast}\hspace{-0.5ex}%
N},X\right)  ,\nonumber\\
\tikz[baseline=(current bounding box.west),scale=1,samples=100,thick,font=\scriptsize] { \drawcoop[CoProdStyle=comodule,leftarmlen=1,rightarmlen=1,handlelen=1,yscale=.5]{(0,0)}{(0.5,0)}; \draw (T) node[above] {$M$} (M1) node[below]{$X$} (N1)node[below]{$N$}; \draw[fill=white,thin] (-0.2,0.1) rectangle(0.7,-0.3); \draw ($(-0.2,0.1)!0.5!(0.7,-0.3)$) node {$t$}; }
& \mapsto
\tikz[baseline=(current bounding box.west),scale=1,samples=100,thick,font=\scriptsize] { \drawcoop[CoProdStyle=comodule,leftarmlen=1.3,rightarmlen=0.7,handlelen=1,yscale=.5]{(0,0)}{(0.5,0)}; \draw (T) node[above] {$M$} (M1) node[below]{$X$}; \draw[fill=white,thin] (-0.2,0.1) rectangle(0.7,-0.3); \draw ($(-0.2,0.1)!0.5!(0.7,-0.3)$) node {$t$}; \draw ($(N1)+(0.8,0)$)--(1.3,0.4|- T)node[above]{${}{^{*}\hspace{-0.5ex}N}$ } ($(N1)+(-0.15,-0.2)$)node{$N$}; \halfcircle{(N1)}{($(N1)+(0.8,0)$)};}\ ,\ t\in
\operatorname{Hom}_{\mathcal{C}}\left(  M,X\otimes N\right) \label{eq1}%
\end{align}
with its inverse being%
\[
\tikz[baseline=(current bounding box.west),scale=1,samples=100,thick,font=\scriptsize] { \drawcoop[CoProdStyle=comodule,leftarmlen=1,rightarmlen=1,handlelen=1,yscale=.5]{(0.5,0)}{(0,0)}; \draw (T) node[below] {$X$} (M1) node[above]{${}{^{*}\hspace{-0.5ex}N}$} (N1)node[above]{$M$}; \draw[fill=white,thin] (-0.2,0.3) rectangle(0.7,-0.1); \draw ($(-0.2,0.3)!0.5!(0.7,-0.1)$) node {$s$}; }\mapsto
\tikz[baseline=(current bounding box.west),scale=1,samples=100,thick,font=\scriptsize] { \drawcoop[CoProdStyle=comodule,leftarmlen=0.7,rightarmlen=1.3,handlelen=1.2,yscale=.5]{(0.5,0)}{(0,0)}; \draw (T) node[below] {$X$} ($(M1)+(-0.13,0.25)$) node{${}{^{*}\hspace{-0.5ex}N}$} (N1)node[above]{$M$}; \draw[fill=white,thin] (-0.2,0.3) rectangle(0.7,-0.1); \draw ($(-0.2,0.3)!0.5!(0.7,-0.1)$) node {$s$}; \halfcircle{($(M1)+(0.8,0)$)}{(M1)}; \draw ($(M1)+(0.8,0)$)--(1.3,0.4|- T)node[below]{$N$ };}\ ,\ s\in
\operatorname{Hom}_{\mathcal{C}}\left(  M\otimes{}{^{\ast}\hspace{-0.5ex}%
N},X\right)  ,
\]
which induces an isomorphism%
\begin{gather}
\operatorname{Nat}\left(  M\otimes{}{^{\ast}\hspace{-0.5ex}N}\otimes
id_{\mathcal{C}},id_{\mathcal{C}}\right)  \overset{\cong}{\longrightarrow
}\operatorname{Nat}\left(  M\otimes id_{\mathcal{C}}\otimes{}{^{\ast}%
\hspace{-0.5ex}N},id_{\mathcal{C}}\right)  \overset{\cong}{\longrightarrow
}\operatorname{Nat}\left(  M\otimes id_{\mathcal{C}},id_{\mathcal{C}}\otimes
N\right)  ,\nonumber\\
\tikz[baseline=(current bounding box.west),scale=1,samples=100,thick,font=\scriptsize] { \drawcoop[CoProdStyle=comodule,leftarmlen=0.6,rightarmlen=0.6,handlelen=0.7,yscale=.5]{(1,0)}{(0,0)}; \draw (M1) node[above] {$\bullet$} (T)node[below] {$\bullet$} (N1)node[above]{$M$}; \draw (0.5,0)--(0.5,0.4|- M1)node[above]{${}{^{*}\hspace{-0.5ex}N}$}; \draw[fill=white,thin] (-0.2,0.3) rectangle(1.2,-0.1); \draw ($(-0.2,0.3)!0.5!(1.2,-0.1)$) node {$\delta$}; }\mapsto
\tikz[baseline=(current bounding box.west),scale=1,samples=100,thick,font=\scriptsize] { \drawcrossing[crosstyle=nm,inlefthandlen=0.2,inrighthandlen=0.2,outlefthandlen=0.9,outrighthandlen=0.9,yscale=.5]{(0.5,1)}{(1,1)};\draw (M1)node[above] {$\bullet$} (N1)node [above] {${}{^{*}\hspace{-0.5ex}N}$}; \drawcoop[CoProdStyle=comodule,leftarmlen=0.35,rightarmlen=1.15,handlelen=0.5,yscale=.5]{(1,0)}{(0,0)}; \draw(T)node[below] {$\bullet$} (N1)node[above]{$M$}; \draw[fill=white,thin] (-0.2,0.3) rectangle(1.2,-0.1); \draw ($(-0.2,0.3)!0.5!(1.2,-0.1)$) node {$\delta$}; }\mapsto
\tikz[baseline=(current bounding box.west),scale=1,samples=100,thick,font=\scriptsize] { \drawcrossing[crosstyle=nm,inlefthandlen=0.3,inrighthandlen=0,outlefthandlen=0.9,outrighthandlen=0.9,yscale=.5]{(0.5,1)}{(1,1)};\draw (M1)node[above] {$\bullet$} ($(N1)+(-0.15,0.2)$)node {${}{^{*}\hspace{-0.5ex}N}$};  \coordinate (X) at ($(N1)+(0.5,0)$); \halfcircle{(X)}{(N1)}; \drawcoop[CoProdStyle=comodule,leftarmlen=0.35,rightarmlen=1.2,handlelen=0.5,yscale=.5]{(1,0)}{(0,0)}; \draw(T)node[below] {$\bullet$} (N1)node[above]{$M$}; \draw[fill=white,thin] (-0.2,0.3) rectangle(1.2,-0.1); \draw ($(-0.2,0.3)!0.5!(1.2,-0.1)$) node {$\delta$}; \draw (X)--(X |- T)node[below]{$N$}; }\ ,\ \delta
\in\operatorname{Nat}\left(  M\otimes{}{^{\ast}\hspace{-0.5ex}N}\otimes
id_{\mathcal{C}},id_{\mathcal{C}}\right)  .\label{eq2}%
\end{gather}
Then by a graphical calculation, $\varphi_{M,N}$ is the following composition%
\begin{gather*}%
\begin{split}
\operatorname{Hom}_{\mathcal{C}}\left(  M,B\otimes N\right)  \overset
{(\ref{eq1})}{-\!\!-\!\!\!\longrightarrow}\operatorname{Hom}_{\mathcal{C}%
}\left(  M\otimes{}{^{\ast}\hspace{-0.5ex}N},B\right)   &  \overset
{\theta_{M\otimes{}{^{\ast}\hspace{-0.5ex}N}}}{-\!\!-\!\!\!\longrightarrow
}\operatorname{Nat}\left(  M\otimes{}{^{\ast}\hspace{-0.5ex}N}\otimes
id_{\mathcal{C}},id_{\mathcal{C}}\right) \\
&  \overset{\left(  \ref{eq2}\right)  }{-\!\!-\!\!\!\longrightarrow
}\operatorname{Nat}\left(  M\otimes id_{\mathcal{C}},id_{\mathcal{C}}\otimes
N\right)  ,
\end{split}
\\
\tikz[baseline=(current bounding box.west),scale=1,samples=100,thick,font=\scriptsize] { \drawcoop[CoProdStyle=comodule,leftarmlen=1,rightarmlen=1,handlelen=1,yscale=.5]{(0,0)}{(0.5,0)}; \draw (T) node[above] {$M$} (M1) node[below]{$B$} (N1)node[below]{$N$}; \draw[fill=white,thin] (-0.2,0.1) rectangle(0.7,-0.3); \draw ($(-0.2,0.1)!0.5!(0.7,-0.3)$) node {$t$};}\mapsto
\tikz[baseline=(current bounding box.west),scale=1,samples=100,thick,font=\scriptsize] { \drawcoop[CoProdStyle=comodule,leftarmlen=1.3,rightarmlen=0.7,handlelen=1,yscale=.5]{(0,0)}{(0.5,0)}; \draw (T) node[above] {$M$} (M1) node[below]{$B$}; \draw[fill=white,thin] (-0.2,0.1) rectangle(0.7,-0.3); \draw ($(-0.2,0.1)!0.5!(0.7,-0.3)$) node {$t$}; \draw ($(N1)+(0.5,0)$)--(1,0.4|- T)node[above]{${}{^{*}\hspace{-0.5ex}N}$ } ($(N1)+(-0.15,-0.2)$)node{$N$}; \halfcircle{(N1)}{($(N1)+(0.5,0)$)};}\mapsto
\tikz[baseline=(current bounding box.west),scale=1,samples=100,thick,font=\scriptsize] { \drawcoop[CoProdStyle=comodule,leftarmlen=1.65,rightarmlen=0.7,handlelen=1,yscale=.5]{(0,0)}{(0.5,0)}; \draw (T) node[above] {$M$}; \draw[fill=white,thin] (-0.2,0.1) rectangle(0.7,-0.3); \draw ($(-0.2,0.1)!0.5!(0.7,-0.3)$) node {$t$}; \draw ($(N1)+(0.5,0)$)--(1,0.4|- T)node[above]{${}{^{*}\hspace{-0.5ex}N}$ } ($(N1)+(-0.15,-0.2)$)node{$N$}; \halfcircle{(N1)}{($(N1)+(0.5,0)$)}; \drawcoop[CoProdStyle=comodule,leftarmlen=0.72,rightarmlen=0.3,handlelen=0.25,yscale=.5]{($(M1)+(1.5,0)$)}{(M1)};\draw (M1)node[above]{$\bullet$} (T)node[below]{$\bullet$};}\mapsto
\tikz[baseline=(current bounding box.west),scale=1,samples=100,thick,font=\scriptsize] { \drawcoop[CoProdStyle=comodule,leftarmlen=1.65,rightarmlen=0.7,handlelen=2.5,yscale=.5]{(0,0)}{(0.5,0)}; \draw (T) node[above] {$M$}; \draw[fill=white,thin] (-0.2,0.1) rectangle(0.7,-0.3); \draw ($(-0.2,0.1)!0.5!(0.7,-0.3)$) node {$t$}; \draw ($(N1)+(-0.15,-0.2)$)node{$N$}; \halfcircle{(N1)}{($(N1)+(0.5,0)$)}; \drawcoop[CoProdStyle=comodule,leftarmlen=0.72,rightarmlen=0.3,handlelen=0.25,yscale=.5]{($(M1)+(1.5,0)$)}{(M1)};\draw (T)node[below]{$\bullet$}; \drawcrossing[crosstyle=nm,inlefthandlen=0.2,inrighthandlen=0,outlefthandlen=1.3,outrighthandlen=0.1,yscale=.5]{(1,0.55)}{(1.5,0.55)}; \coordinate (X) at ($(N1)+(0.5,0)$); \halfcircle{(X)}{(N1)}; \draw (M1)node[above] {$\bullet$} (X)--(X |- T)node[below]{$N$};}=\tikz[baseline=(current bounding box.west),scale=1,samples=100,thick,font=\scriptsize] { \drawcoop[CoProdStyle=comodule,leftarmlen=0.9,rightarmlen=0.7,handlelen=1,yscale=.5,leftarmpt=P]{(0,0)}{(0.5,0)};   \draw (T) node[above] {$M$}; \drawcrossing[crosstyle=mn,inlefthandlen=0.4,inrighthandlen=1.2,outlefthandlen=0.2,outrighthandlen=0.55,yscale=.5]{(N1)}{($(N1)+(.5,0)$)}; \draw (N1) node[above] {$\bullet$} ($(M2)+(-0.18,0)$)node{$N$}; \halfcircle{(M2)}{($(M2)+(0.5,0)$)}; \coordinate (X) at ($(M2)+(1,0)$); \halfcircle{(X)}{($(M2)+(0.5,0)$)}; \drawcoop[CoProdStyle=comodule,leftarmlen=0.1,rightarmlen=1,handlelen=1.5,yscale=.5]{(N2)}{(P |- N2)}; \draw (T) node[below] {$\bullet$} (X)--(X |-T)node[below]{$N$};  \draw[fill=white,thin] (-0.2,0.1) rectangle(0.7,-0.3); \draw ($(-0.2,0.1)!0.5!(0.7,-0.3)$) node {$t$};}=\tikz[baseline=(current bounding box.west),scale=1,samples=100,thick,font=\scriptsize] { \drawcoop[CoProdStyle=comodule,leftarmlen=0.9,rightarmlen=0.7,handlelen=1,yscale=.5,leftarmpt=P]{(0,0)}{(0.5,0)};   \draw (T) node[above] {$M$}; \drawcrossing[crosstyle=mn,inlefthandlen=0.4,inrighthandlen=1.2,outlefthandlen=0.2,outrighthandlen=0.55,yscale=.5]{(N1)}{($(N1)+(.5,0)$)}; \draw (N1) node[above] {$\bullet$} (M2) node[below] {$N$}; \drawcoop[CoProdStyle=comodule,leftarmlen=0.1,rightarmlen=1,handlelen=.7,yscale=.5]{(N2)}{(P |- N2)}; \draw (T) node[below] {$\bullet$};  \draw[fill=white,thin] (-0.2,0.1) rectangle(0.7,-0.3); \draw ($(-0.2,0.1)!0.5!(0.7,-0.3)$) node {$t$}; },
\end{gather*}
which is clearly natural in both variables $M$ and $N$.

Similarly, $\varphi_{M,N}^{2}$ is the composition of the isomorphisms%
\begin{align*}
\operatorname{Hom}_{\mathcal{C}}\left(  M,B\otimes B\otimes N\right)   &
\rightarrow\operatorname{Hom}_{\mathcal{C}}\left(  M\otimes{}{^{\ast}%
\hspace{-0.5ex}N},B\otimes B\right)  \overset{\theta_{M\otimes{}{^{\ast
}\hspace{-0.5ex}N}}^{2}}{-\!\!-\!\!\!\longrightarrow}\operatorname{Nat}\left(
M\otimes{}{^{\ast}\hspace{-0.5ex}N}\otimes id_{\mathcal{C}}{}^{\otimes
2},id_{\mathcal{C}}{}^{\otimes2}\right) \\
& \rightarrow\operatorname{Nat}\left(  M\otimes id_{\mathcal{C}}{}^{\otimes
2},id_{\mathcal{C}}{}^{\otimes2}\otimes N\right)  ,
\end{align*}
which is natural in $M$ and $N$.
\end{proof}

\begin{lemma}
\label{Lemma eqiv cnt and comod}Assume that $\mathcal{C}$ is a braided rigid
category, and the automorphism braided group $B=U\left(  \mathcal{C}\right)  $
exists. For any $M\in\mathcal{C}$ and a morphism $\rho_{M}:M\rightarrow
B\otimes M$, let $\gamma_{M,\bullet}=\varphi_{M,M}\left(  \rho_{M}\right)  $
be the natural transformation from $M\otimes id_{\mathcal{C}}$ to
$id_{\mathcal{C}}\otimes M$. Then%
\[
\left(  \Delta_{B}\otimes id_{M}\right)  \rho_{M}=\left(  id_{B}\otimes
\rho_{M}\right)  \rho_{M},
\]
if and only if
\[
\gamma_{M,X\otimes Y}=\left(  id_{X}\otimes\gamma_{M,Y}\right)  \left(
\gamma_{M,X}\otimes id_{Y}\right)  ,\ \text{for any }X,Y\in\mathcal{C}.
\]

\end{lemma}

\begin{proof}
It is clear that for any $X,Y\in\mathcal{C}$, by the definition of
$\varphi^{2} $ and (\ref{Fig2 stru of B}) the morphism $\varphi_{M,M}%
^{2}\left(  \left(  \Delta_{B}\otimes id_{M}\right)  \rho_{M}\right)  _{X,Y}$
is expressed by the diagram%
\[
\tikz[baseline=(current bounding box.west),scale=1,samples=100,thick,font=\scriptsize] { \coordinate (X) at (0,0); \coordinate (Y) at (0.75,0); \drawcoop[CoProdStyle=comodule,leftarmlen=0.7,rightarmlen=0.7,handlelen=1,yscale=.5,leftarmpt=P]{(X)}{(Y)}; \draw (T) node[above] {$M$}; \draw[fill=white,thin] ($(X)+(-0.1,0.1)$) rectangle($(Y)+(0.1,-0.3)$); \draw ($($(X)+(-0.1,0.1)$)!0.5!($(Y)+(0.1,-0.3)$)$) node {$\rho _{M}$}; \drawcrossing[crosstyle=mn,inlefthandlen=0.2,inrighthandlen=1.8,outlefthandlen=0.5,outrighthandlen=0.2,yscale=.5]{(N1)}{($(N1)+(.5,0)$)}; \draw (N1)node [above]{$X$}; \drawcoop[CoProdStyle=coproduct,leftarmlen=0.9,rightarmlen=0.8,handlelen=0.7,yscale=.5,leftarmpt=P]{($(P)+(-.25,0)$)}{($(P)+(0.25,0)$)}; \coordinate (M3) at (M2); \drawcrossing[crosstyle=mn,inlefthandlen=0.2,inrighthandlen=0.2,outlefthandlen=0.2,outrighthandlen=0.2,yscale=.5]{($(N2)+(-.5,0)$)}{(N2)}; \drawcoop[CoProdStyle=comodule,leftarmlen=0.1,rightarmlen=0.9,handlelen=1,yscale=.5]{(N2)}{($(N2)+(-.5,0)$)}; \draw (T)node[below]{$X$}; \drawcoop[CoProdStyle=comodule,leftarmlen=0.1,rightarmlen=0.1,handlelen=1,yscale=.5]{($(M2)+(.5,0)$)}{(M2)}; \draw (T)node[below]{$Y$}; \drawcrossing[crosstyle=mn,inlefthandlen=0.2,inrighthandlen=2.5,outlefthandlen=0.4,outrighthandlen=1,yscale=.5]{(M3)}{($(M3)+(0.5,0)$)}; \draw (M2)node[below]{$M$} (N1)node [above]{$Y$}; }=\tikz[baseline=(current bounding box.west),scale=1,samples=100,thick,font=\scriptsize] { \drawcoop[CoProdStyle=comodule,leftarmlen=0.9,rightarmlen=0.7,handlelen=1,yscale=.4,leftarmpt=P]{(0,0)}{(0.8,0)};   \draw (T) node[above] {$M$}; \drawcrossing[crosstyle=mn,inlefthandlen=0.4,inrighthandlen=1.2,outlefthandlen=0.2,outrighthandlen=0.9,yscale=.4]{(N1)}{($(N1)+(.8,0)$)}; \draw (N1) node[above] {$X\otimes Y$} (M2)node[below]{$M$}; \coordinate (X) at ($(M2)+(1,0)$); \drawcoop[CoProdStyle=comodule,leftarmlen=0.1,rightarmlen=1,handlelen=1.5,yscale=.4]{(N2)}{(P |- N2)}; \draw (T) node[below] {$X\otimes Y$} ($(T)+(0.45,0.45)$)node{$\alpha_{X\otimes Y}$}; \draw[fill=white,thin] (-0.2,0.1) rectangle(1,-0.3); \draw ($(-0.2,0.1)!0.5!(1,-0.3)$) node {$\rho _{M}$};}=\gamma
_{M,X\otimes Y},
\]
while the morphism $\varphi_{M,M}^{2}\left(  \left(  id_{B}\otimes\rho
_{M}\right)  \rho_{M}\right)  _{X,Y}$ is expressed by the diagram
\[
\tikz[baseline=(current bounding box.west),scale=1,samples=100,thick,font=\scriptsize] { \coordinate (X) at (0,0); \coordinate (Y) at (0.75,0); \drawcoop[CoProdStyle=comodule,leftarmlen=1.7,rightarmlen=0.9,handlelen=0.8,yscale=.5,rightarmpt=P]{(X)}{(Y)}; \draw (T) node[above] {$M$}; \draw[fill=white,thin] ($(X)+(-0.1,0.1)$) rectangle($(Y)+(0.1,-0.3)$); \draw ($($(X)+(-0.1,0.1)$)!0.5!($(Y)+(0.1,-0.3)$)$) node {$\rho _{M}$}; \coordinate (X) at ($(P)+(-.25,0)$); \coordinate (Y) at ($(P)+(0.25,0)$); \drawcoop[CoProdStyle=comodule,leftarmlen=1,rightarmlen=0.9,handlelen=0.7,yscale=.5]{(X)}{(Y)}; \draw[fill=white,thin] ($(X)+(-0.1,0.1)$) rectangle($(Y)+(0.1,-0.3)$); \draw ($($(X)+(-0.1,0.1)$)!0.5!($(Y)+(0.1,-0.3)$)$) node {$\rho _{M}$}; \drawcrossing[crosstyle=mn,inlefthandlen=0.2,inrighthandlen=2.8,outlefthandlen=0.1,outrighthandlen=0.2,yscale=.5]{(N1)}{($(N1)+(.5,0)$)}; \draw (N1)node [above]{$X$}; \coordinate (M3) at (M2); \drawcrossing[crosstyle=mn,inlefthandlen=0.6,inrighthandlen=0.1,outlefthandlen=0.2,outrighthandlen=0.2,yscale=.5]{($(N2)+(-.5,0)$)}{(N2)}; \drawcoop[CoProdStyle=comodule,leftarmlen=0.1,rightarmlen=1.9,handlelen=0.8,yscale=.5]{(N2)}{($(N2)+(-.5,0)$)}; \draw (T)node[below]{$X$}; \drawcoop[CoProdStyle=comodule,leftarmlen=0.1,rightarmlen=0.1,handlelen=0.8,yscale=.5]{($(M2)+(.5,0)$)}{(M2)}; \draw (T)node[below]{$Y$}; \drawcrossing[crosstyle=mn,inlefthandlen=0.2,inrighthandlen=3.5,outlefthandlen=0,outrighthandlen=0.5,yscale=.5]{(M3)}{($(M3)+(0.5,0)$)}; \draw (M2)node[below]{$M$} (N1)node [above]{$Y$}; }=\tikz[baseline=(current bounding box.west),scale=1,samples=100,thick,font=\scriptsize] { \coordinate (X) at (0,0); \coordinate (Y) at (0.5,0); \drawcoop[CoProdStyle=comodule,leftarmlen=0.9,rightarmlen=0.9,handlelen=1,yscale=.5,rightarmpt=P]{(X)}{(Y)}; \draw (T) node[above] {$M$}; \draw[fill=white,thin] ($(X)+(-0.1,0.1)$) rectangle($(Y)+(0.1,-0.3)$); \draw ($($(X)+(-0.1,0.1)$)!0.5!($(Y)+(0.1,-0.3)$)$) node {$\rho _{M}$}; \drawcrossing[crosstyle=mn,inlefthandlen=0,inrighthandlen=1.5,outlefthandlen=0.1,outrighthandlen=0.5,yscale=.5]{(P)}{($(P)+(.5,0)$)}; \draw (N1)node [above]{$X$}; \coordinate (X) at ($(M2)+(-.25,0)$); \coordinate (Y) at ($(M2)+(0.25,0)$); \drawcoop[CoProdStyle=comodule,leftarmlen=0.1,rightarmlen=.9,handlelen=4.6,yscale=.5]{(N2)}{($(N2)+(-.5,0)$)}; \draw (T)node[below]{$X$}; \drawcoop[CoProdStyle=comodule,leftarmlen=1,rightarmlen=1,handlelen=0.7,yscale=.5,rightarmpt=P]{(X)}{(Y)}; \draw[fill=white,thin] ($(X)+(-0.1,0.1)$) rectangle($(Y)+(0.1,-0.3)$); \draw ($($(X)+(-0.1,0.1)$)!0.5!($(Y)+(0.1,-0.3)$)$) node {$\rho _{M}$}; \drawcrossing[crosstyle=mn,inlefthandlen=0.2,inrighthandlen=3.5,outlefthandlen=0.1,outrighthandlen=0.5,yscale=.5]{(P)}{($(P)+(0.5,0)$)}; \draw (M2)node[below]{$M$} (N1)node [above]{$Y$}; \drawcoop[CoProdStyle=comodule,leftarmlen=0.1,rightarmlen=0.8,handlelen=0.8,yscale=.5]{(N2)}{($(N2)-(.5,0)$)}; \draw (T)node[below]{$Y$}; }=\left(
id_{X}\otimes\gamma_{M,Y}\right)  \left(  \gamma_{M,X}\otimes id_{Y}\right)  .
\]
Since $\varphi_{M,M}^{2}$ is an isomorphism by Lemma \ref{Lemma nat iso}, the
equality
\[
\gamma_{M,X\otimes Y}=\left(  id_{X}\otimes\gamma_{M,Y}\right)  \left(
\gamma_{M,X}\otimes id_{Y}\right)
\]
holds for all $X,Y\in\mathcal{C}$ if and only if
\[
\left(  \Delta_{B}\otimes id_{M}\right)  \rho_{M}=\left(  id_{B}\otimes
\rho_{M}\right)  \rho_{M}.
\]

\end{proof}

For any coalgebra $D$ in $\mathcal{C}$, let $D$-$\operatorname*{Comod}%
\nolimits_{\mathcal{C}}$ be the category of left $D$-comodules in
$\mathcal{C}$. Then $D $-$\operatorname*{Comod}\nolimits_{\mathcal{C}}$ is a
natural right $\mathcal{C}$-module category, where for any $\left(  M,\rho
_{M}\right)  \in D$-$\operatorname*{Comod}\nolimits_{\mathcal{C}}$ and
$X\in\mathcal{C}$, the comodule morphism of $M\otimes X$ is
\begin{equation}
\rho_{M\otimes X}=\rho_{M}\otimes id_{X}:M\otimes X\rightarrow D\otimes
M\otimes X.\label{module cat stru on Dcomod}%
\end{equation}
Specially $B$-$\operatorname*{Comod}_{\mathcal{C}}$ can be viewed as a right
$\mathcal{C}$-module category in this way.

The category $\mathcal{Z}_{l}\left(  \mathcal{C}\right)  $ is also a right
$\mathcal{C}$-module category, via the tensor functor $\mathcal{C\rightarrow
Z}_{l}\left(  \mathcal{C}\right)  ,X\mapsto\left(  X,c_{X,\bullet}\right)  $.
Precisely, for an object $\left(  Z,\gamma_{Z,\bullet}\right)  $ of
$\mathcal{Z}_{l}\left(  \mathcal{C}\right)  $, $\left(  Z\otimes
X,\gamma_{Z\otimes X,\bullet}\right)  $ is an object of $\mathcal{Z}%
_{l}\left(  \mathcal{C}\right)  $ with $\gamma_{Z\otimes X,\bullet}=\left(
\gamma_{Z,\bullet}\otimes id_{X}\right)  \left(  id_{Z}\otimes c_{X,\bullet
}\right)  $.

Now we are ready to prove the main result of this section.

\begin{theorem}
\label{theorem center iso H-comod}Let $\mathcal{C}$ be a braided rigid
category with representability assumption for modules. Let $B=U\left(
\mathcal{C}\right)  $ be the automorphism braided group of $\mathcal{C}$. For
any $M\in\mathcal{C}$, we have the following statements.

\begin{enumerate}
\item If $\left(  M,\rho_{M}\right)  $ is a left $B$-comodule in $\mathcal{C}%
$, then $\left(  M,\varphi_{M,M}\left(  \rho_{M}\right)  \right)  $ is an
object of $\mathcal{Z}_{l}\left(  \mathcal{C}\right)  $.

\item If $\left(  M,\gamma_{M,\bullet}\right)  $ is an object of
$\mathcal{Z}_{l}\left(  \mathcal{C}\right)  $, then $\left(  M,\varphi
_{M,M}^{-1}\left(  \gamma_{M,\bullet}\right)  \right)  $ is a left
$B$-comodule in $\mathcal{C}$.
\end{enumerate}

Moreover, as right $\mathcal{C}$-module categories $\mathcal{Z}_{l}\left(
\mathcal{C}\right)  $ and $B$-$\operatorname*{Comod}_{\mathcal{C}}$ are
equivalent via%
\begin{align*}
F  & :B\text{-}\operatorname*{Comod}\nolimits_{\mathcal{C}}\rightarrow
\mathcal{Z}_{l}\left(  \mathcal{C}\right)  ,\ \left(  M,\rho_{M}\right)
\mapsto\left(  M,\varphi_{M,M}\left(  \rho_{M}\right)  \right)  ,\\
G  & :\mathcal{Z}_{l}\left(  \mathcal{C}\right)  \rightarrow B\text{-}%
\operatorname*{Comod}\nolimits_{\mathcal{C}},\ \left(  M,\gamma_{M,\bullet
}\right)  \mapsto\left(  M,\varphi_{M,M}^{-1}\left(  \gamma_{M,\bullet
}\right)  \right)  .
\end{align*}

\end{theorem}

\begin{proof}
If $\left(  M,\rho_{M}\right)  \in B$-$\operatorname*{Comod}_{\mathcal{C}}$,
then by Lemma \ref{Lemma eqiv cnt and comod}
\[
\varphi_{M,M}\left(  \rho_{M}\right)  _{X\otimes Y}=\left(  id_{X}%
\otimes\varphi_{M,M}\left(  \rho_{M}\right)  _{Y}\right)  \left(
\varphi_{M,M}\left(  \rho_{M}\right)  _{X}\otimes id_{Y}\right)  ,
\]
so $\varphi_{M,M}\left(  \rho_{M}\right)  $ is an isomorphism by the rigidity
of $\mathcal{C}$ and Lemma \ref{lemma gama is iso}. Thus $\left(
M,\varphi_{M,M}\left(  \rho_{M}\right)  \right)  $ is an object of
$\mathcal{Z}_{l}\left(  \mathcal{C}\right)  $.

Conversely, assume that $\left(  M,\gamma_{M,\bullet}\right)  $ is an object
of $\mathcal{Z}_{l}\left(  \mathcal{C}\right)  .$ Let $\rho_{M}=\varphi
_{M,M}^{-1}\left(  \gamma_{M,\bullet}\right)  $, then $\gamma_{M,\bullet
}=\varphi_{M,M}\left(  \rho_{M}\right)  $. Again by Lemma \ref{Lemma eqiv cnt
and comod}
\[
\left(  id_{B}\otimes\rho_{M}\right)  \rho_{M}=\left(  \Delta_{B}\otimes
id_{M}\right)  \rho_{M}.
\]
In addition,
\[
id_{M}=\gamma_{M,1}=\varphi_{M,M}\left(  \rho_{M}\right)  _{1}%
=\tikz[baseline=(current bounding box.west),scale=1,samples=100,thick,font=\scriptsize] { \coordinate (X) at (0,0); \coordinate (Y) at (0.6,0); \drawcoop[CoProdStyle=comodule,leftarmlen=1.2,rightarmlen=1.3,handlelen=1,yscale=.5,leftarmtext={{circle/1/\text{\normalsize $\varepsilon$}}}]{(X)}{(Y)}; \draw (T) node[above] {$M$} (N1)node[below]{$M$}; \draw[fill=white,thin] ($(X)+(-0.1,0.1)$) rectangle($(Y)+(0.1,-0.3)$); \draw ($($(X)+(-0.1,0.1)$)!0.5!($(Y)+(0.1,-0.3)$)$) node {$\rho _{M}$};}.
\]
Thus $\left(  M,\varphi_{M,M}^{-1}\left(  \gamma_{M,\bullet}\right)  \right)
$ is a $B$-comodule in $\mathcal{C}$.

Moreover, let $\left(  M,\rho_{M}\right)  $, $\left(  N,\rho_{N}\right)  \in
B$-$\operatorname*{Comod}\nolimits_{\mathcal{C}}$ and $f:\left(  M,\rho
_{M}\right)  \rightarrow\left(  N,\rho_{N}\right)  $ be a $B$-comodule map in
$\mathcal{C} $. Then we have%
\[
\varphi_{M,N}\left(  \rho_{N}f\right)  =\varphi_{M,N}\left(  \left(
id_{B}\otimes f\right)  \rho_{M}\right)  ,
\]
that is,
\[
\tikz[baseline=(current bounding box.west),scale=1,samples=100,thick,font=\scriptsize] { \coordinate (X) at (0,0); \coordinate (Y) at (0.8,0); \drawcoop[CoProdStyle=comodule,leftarmlen=0.5,rightarmlen=0.5,handlelen=2.5,yscale=.5,rightarmpt=P,handletext={rectangle/0.5/f}]{(X)}{(Y)}; \draw (T) node[above] {$M$}; \draw[fill=white,thin] ($(X)+(-0.1,0.1)$) rectangle($(Y)+(0.1,-0.3)$); \draw ($($(X)+(-0.1,0.1)$)!0.5!($(Y)+(0.1,-0.3)$)$) node {$\rho _{N}$}; \drawcrossing[crosstyle=mn,inlefthandlen=0,inrighthandlen=1.8,outlefthandlen=0.1,outrighthandlen=.5,yscale=.5]{(P)}{($(P)+(.8,0)$)}; \draw (N1)node [above]{$X$} (M2)node[below]{$N$}; \drawcoop[CoProdStyle=comodule,leftarmlen=0.1,rightarmlen=.7,handlelen=0.7,yscale=.5]{(N2)}{($(N2)+(-.8,0)$)}; \draw (T)node[below]{$X$};}=\tikz[baseline=(current bounding box.west),scale=1,samples=100,thick,font=\scriptsize] { \coordinate (X) at (0,0); \coordinate (Y) at (0.8,0); \drawcoop[CoProdStyle=comodule,leftarmlen=1.5,rightarmlen=1.3,handlelen=1,yscale=.5,rightarmpt=P,rightarmtext={rectangle/0.65/f}]{(X)}{(Y)}; \draw (T) node[above] {$M$}; \draw[fill=white,thin] ($(X)+(-0.1,0.1)$) rectangle($(Y)+(0.1,-0.3)$); \draw ($($(X)+(-0.1,0.1)$)!0.5!($(Y)+(0.1,-0.3)$)$) node {$\rho _{M}$}; \drawcrossing[crosstyle=mn,inlefthandlen=0,inrighthandlen=1.8,outlefthandlen=0.1,outrighthandlen=.5,yscale=.5]{(P)}{($(P)+(.8,0)$)}; \draw (N1)node [above]{$X$} (M2)node[below]{$N$}; \drawcoop[CoProdStyle=comodule,leftarmlen=0.1,rightarmlen=.7,handlelen=0.7,yscale=.5]{(N2)}{($(N2)+(-.8,0)$)}; \draw (T)node[below]{$X$};}=\tikz[baseline=(current bounding box.west),scale=1,samples=100,thick,font=\scriptsize] { \coordinate (X) at (0,0); \coordinate (Y) at (0.8,0); \drawcoop[CoProdStyle=comodule,leftarmlen=0.5,rightarmlen=0.5,handlelen=0.8,yscale=.5,rightarmpt=P]{(X)}{(Y)}; \draw (T) node[above] {$M$}; \draw[fill=white,thin] ($(X)+(-0.1,0.1)$) rectangle($(Y)+(0.1,-0.3)$); \draw ($($(X)+(-0.1,0.1)$)!0.5!($(Y)+(0.1,-0.3)$)$) node {$\rho _{M}$}; \drawcrossing[crosstyle=mn,inlefthandlen=0,inrighthandlen=0.8,outlefthandlen=0.1,outrighthandlen=0,yscale=.5]{(P)}{($(P)+(.8,0)$)}; \draw (N1)node [above]{$X$};  \coordinate (X) at ($(M2)+(-0.25,0.7)$); \drawcoop[CoProdStyle=comodule,leftarmlen=0.1,rightarmlen=.7,handlelen=2.5,yscale=.5,handlept=T2]{(N2)}{($(N2)+(-.8,0)$)}; \linewithtext[text={rectangle/0.6/f}]{(M2)}{(M2|-T2)}; \draw (T2)node[below]{$X$} (M2|-T2)node[below]{$N$}; },
\]
which implies that
\[
F\left(  f\right)  =f:\left(  M,\varphi_{M,M}\left(  \rho_{M}\right)  \right)
\rightarrow\left(  N,\varphi_{N,N}\left(  \rho_{N}\right)  \right)
\]
is a map in $\mathcal{Z}_{l}\left(  \mathcal{C}\right)  $.

If $\left(  M,\gamma_{M,\bullet}\right)  ,\left(  N,\gamma_{N,\bullet}\right)
\in$ $\mathcal{Z}_{l}\left(  \mathcal{C}\right)  $, and $f:\left(
M,\gamma_{M,\bullet}\right)  \rightarrow\left(  N,\gamma_{N,\bullet}\right)  $
is a map in $\mathcal{Z}_{l}\left(  \mathcal{C}\right)  $, then $G\left(
f\right)  =f:\left(  M,\varphi_{M,M}^{-1}\left(  \gamma_{M,\bullet}\right)
\right)  \rightarrow\left(  N,\varphi_{N,N}^{-1}\left(  \gamma_{N,\bullet
}\right)  \right)  $ is a map in $B$-$\operatorname*{Comod}_{\mathcal{C}} $.
Clearly, $FG=id$, and $GF=id$. This establishes the equivalence of
$\mathcal{Z}_{l}\left(  \mathcal{C}\right)  $ and $B$-$\operatorname*{Comod}%
_{\mathcal{C}}$.

Finally we show that $F$ is a $\mathcal{C}$-module functor. Let $\left(
M,\rho_{M}\right)  $ be an object of $B$-$\operatorname*{Comod}_{\mathcal{C}}%
$. For all $X\in\mathcal{C}$, observe that%
\[
\varphi_{M\otimes X,M\otimes X}\left(  \rho_{M}\otimes id_{X}\right)
=\tikz[baseline=(current bounding box.west),scale=1,samples=100,thick,font=\scriptsize] { \coordinate (X) at (0,0); \coordinate (Y) at (0.5,0); \drawcoop[CoProdStyle=comodule,leftarmlen=1.3,rightarmlen=1,handlelen=1,yscale=.5,leftarmpt=P]{(X)}{(Y)}; \draw (T) node[above] {$M$}; \draw[fill=white,thin] ($(X)+(-0.1,0.1)$) rectangle($(Y)+(0.1,-0.3)$); \draw ($($(X)+(-0.1,0.1)$)!0.5!($(Y)+(0.1,-0.3)$)$) node {$\rho _{M}$}; \drawcrossing[crosstyle=mn,inlefthandlen=0.2,inrighthandlen=0.7,outlefthandlen=0.5,outrighthandlen=1,yscale=.5]{(N1)}{($(N1)+(.5,0)$)}; \draw (M2)node[below]{$M$}; \coordinate (M3) at (N1); \drawcoop[CoProdStyle=comodule,leftarmlen=0.1,rightarmlen=0.9,handlelen=1,yscale=.5]{(N2)}{($(N2)+(-.5,0)$)}; \draw (T)node[below]{$\bullet$}; \drawcrossing[crosstyle=mn,inlefthandlen=2.15,inrighthandlen=0.2,outlefthandlen=0.3,outrighthandlen=0.3,yscale=.5]{($(M3)+(0.5,0)$)}{(M3)}; \draw (M2)node[above]{$X$} (M1)node[below]{$X$} (N2)node [above]{$\bullet$}; }=\left(
\varphi_{M,M}\left(  \rho_{M}\right)  \otimes id_{X}\right)  \left(
id_{M}\otimes c_{X,\bullet}\right)  .
\]
Then $F\left(  M\otimes X\right)  =F\left(  M\right)  \otimes X$, and thus
$\left(  F,s\right)  $ is a $\mathcal{C}$-module with $s_{M,X}=id_{F\left(
M\otimes X\right)  }$.
\end{proof}

\begin{remark}
If $\mathcal{C}$ is not strict, the same argument of Theorem \ref{theorem
center iso H-comod} is also true. The proof is similar but quite lengthy, we
leave this for an interested reader.
\end{remark}

\section{The Center of Braided Multifusion Categories --- A Decomposition
Theorem}

\label{sec-decomp thm}

In this section, $\mathcal{C}$ will be a finite braided multitensor category.
We assume that for $\mathcal{C}$ the module representability assumption holds.
Let $B$ be the automorphism braided group $U\left(  \mathcal{C}\right)  $. We
will show that any direct sum decomposition of $B$ in $B$%
-$\operatorname*{Comod}_{\mathcal{C}}$ induces a decomposition of the category
$B$-$\operatorname*{Comod}_{\mathcal{C}}$ into a direct sum of $\mathcal{C}%
$-module subcategories.

Let $i:D\rightarrow C$ be a monomorphism in $\mathcal{C}$. Since the bifunctor
$\otimes:\mathcal{C\times C}\rightarrow\mathcal{C}$ is exact in both factors,
$\left(  D\otimes C,i\otimes id_{C}\right)  $ and $\left(  C\otimes
D,id_{C}\otimes i\right)  $ are subobjects of $C\otimes C\in\mathcal{C}$. We
show that $\left(  D\otimes D,i\otimes i\right)  $ is the intersection of
subobjects $D\otimes C$ and $C\otimes D$ in the following lemma, i.e.,
$\left(  D\otimes D,i\otimes id_{D},id_{D}\otimes i\right)  $ is the pullback
of the monomorphisms $i\otimes id_{C}$ and $id_{C}\otimes i$.

\begin{lemma}
\label{lemma intersection}Let $0\rightarrow D_{j}\overset{i_{j}}{\rightarrow
}C_{j}\overset{f_{j}}{\rightarrow}E_{j}$ $\left(  j=1,2\right)  $ be exact
sequences in $\mathcal{C}$. If $g:X\rightarrow C_{1}\otimes C_{2}$ is a
morphism in $\mathcal{C}$ with $\left(  f_{1}\otimes id_{C_{2}}\right)  g=0$
and $\left(  id_{C_{1}}\otimes f_{2}\right)  g=0$, then there exists a unique
morphism $h:X\rightarrow D_{1}\otimes D_{2},$ such that $g=\left(
i_{1}\otimes i_{2}\right)  h$.

Moreover, $\left(  D_{1}\otimes D_{2},i_{1}\otimes i_{2}\right)  $ is the
intersection of the subobjects $D_{1}\otimes C_{2}$ and $C_{1}\otimes D_{2}$.
\end{lemma}

\begin{proof}
Consider the diagram
\[%
\def\mleftdelim{.}\def\mrightdelim{.}\def\mrowsep{1cm}\def\mcolumnsep{0.8cm}%
\begin{tikzpicture}[scale=1,samples=100,baseline,font=\scriptsize,]\matrix (m) [matrix of math nodes,left delimiter={\mleftdelim},right delimiter={\mrightdelim},row sep=\mrowsep,column sep=\mcolumnsep]{ & & X & && \\ & & &0 & &0 \\ & & D_1\otimes D_2& & C_1\otimes D_2& \\ 0& D_1\otimes C_2& &C_1\otimes C_2& & E_1\otimes C_2 \\ D_1\otimes E_2& &C_1\otimes E_2 & & & \\}; \begin{scope}[every node/.style={midway,auto,font=\scriptsize}] \draw[->] (m-3-3) --node{$i_1\otimes id_{D_2}$}(m-3-5); \draw[->](m-4-1)--(m-4-2); \draw[->](m-4-2) --node[below]{$i_1\otimes id_{C_2}$}(m-4-4);\draw[->](m-4-4)--node[below]{$f_1\otimes id_{C_2}$}(m-4-6); \draw[->] (m-2-4)--(m-3-3);\draw[->] (m-3-3) --node[right=0.1cm]{$id_{D_1}\otimes i_2$}(m-4-2); \draw[->] (m-4-2)--node[right=0.1cm]{$id_{D_1}\otimes f_2$}(m-5-1);\draw[->] (m-2-6)--(m-3-5);\draw[->] (m-3-5) --node[right=0.1cm]{$id_{C_1}\otimes i_2$}(m-4-4);\draw[->] (m-4-4)--node[right=0.1cm]{$id_{C_1}\otimes f_2$}(m-5-3); \draw[->] (m-5-1) --node[below]{$i_1\otimes id_{E_2}$}(m-5-3);\draw[white,line width=5pt] (m-1-3) --(m-4-4); \draw[->] (m-1-3) --node[above=0.25cm]{$g$}(m-4-4); \draw[->,dashed] (m-1-3) --node[left=0.1cm]{$g_{1}$}(m-4-2);\draw[->,dashed] (m-1-3) --node{$h$}(m-3-3); \end{scope} \end{tikzpicture}.
\]
It is trivial that the two bottom parallelograms commute. The exactness of the
tensor product implies $\left(  D_{1}\otimes C_{2},i_{1}\otimes id_{C_{2}%
}\right)  $, $\left(  C_{1}\otimes D_{2},id_{C_{1}}\otimes i_{2}\right)  $ are
respectively the kernel of $f_{1}\otimes id_{C_{2}}$ and the kernel of
$id_{C_{1}}\otimes f_{2}$. Since $\left(  f_{1}\otimes id_{C_{2}}\right)  g=0$
by assumption, there is a unique morphism $g_{1}:X\rightarrow D_{1}\otimes
C_{2}$ such that $g=\left(  i_{1}\otimes id_{C_{2}}\right)  g_{1}$. Then
\[
\left(  i_{1}\otimes id_{E_{2}}\right)  \left(  id_{D_{1}}\otimes
f_{2}\right)  g_{1}=\left(  id_{C_{1}}\otimes f_{2}\right)  \left(
i_{1}\otimes id_{C_{2}}\right)  g_{1}=\left(  id_{C_{1}}\otimes f_{2}\right)
g=0,
\]
and we have $\left(  id_{D_{1}}\otimes f_{2}\right)  g_{1}=0$, since
$i_{1}\otimes id_{E_{2}}$ is monic. So there exists a morphism $h:X\rightarrow
D_{1}\otimes D_{2}$ such that $g_{1}=\left(  id_{D_{1}}\otimes i_{2}\right)
h$. It is clear that
\[
g=\left(  i_{1}\otimes id_{C_{2}}\right)  g_{1}=\left(  i_{1}\otimes
id_{C_{2}}\right)  \left(  id_{D_{1}}\otimes i_{2}\right)  h=\left(
i_{1}\otimes i_{2}\right)  h.
\]
As the morphism $i\otimes i$ is monic, $h$ is unique.
\end{proof}

We have known from \cite{Majid1991Braided} that $B$ is $\mathcal{C}%
$-cocommutative in the sense that for every object $X\in\mathcal{C}$, the
$B$-action $\alpha_{X}$ on $X$ satisfies the following identity%
\[
\left(  id_{B}\otimes\alpha_{X}\right)  \left(  \Delta_{B}\otimes
id_{X}\right)  =\left(  id_{B}\otimes\alpha_{X}\right)  \left(  c_{B,B}\otimes
id_{X}\right)  \left(  id_{B}\otimes c_{X,B}c_{B,X}\right)  \left(  \Delta
_{B}\otimes id_{X}\right)  ,
\]
that is,
\begin{equation}
\begin{tikzpicture}[xscale=.6,samples=100,thick,font=\scriptsize,baseline=(current bounding box.west)] \drawcoop[CoProdStyle=coproduct,leftarmlen=1,rightarmlen=1,handlelen=0.8,handlept=T1,leftarmpt=M3,]{(0,0)}{(1,0)}; \drawcoop[CoProdStyle=comodule,leftarmlen=0.1,rightarmlen=.1,handlelen=0.8,handlept=T2]{($(N1)!-1!(M3)$)}{(N1)}; \draw (T1)node[above]{$B$} (M1) --(M1 |- T1)node[above]{$X$} (T2)node[below]{$X$} ($(T2)+(0.5,0.25)$)node{$\alpha_{X}$} (M3)--(M3 |- T2)node[below]{$B$}; \end{tikzpicture}=\begin{tikzpicture}[xscale=.6,yscale=.8,samples=100,thick,font=\scriptsize,baseline=(current bounding box.west)] \drawcoop[CoProdStyle=coproduct,leftarmlen=.5,rightarmlen=.5,handlelen=0.6,handlept=T1,leftarmpt=P0]{(0,0)}{(1,0)}; \drawcrossing[crosstyle=mn,inlefthandlen=0.05,inrighthandlen=0.1,outlefthandlen=0.05,outrighthandlen=0.05,inrightnm=P1,yscale=0.5]{(N1)}{($(N1)+(1,0)$)}; \drawcrossing[crosstyle=mn,inlefthandlen=0.05,inrighthandlen=0.05,outlefthandlen=0.05,outrighthandlen=0.05,outrightnm=P2,yscale=0.5]{(N2)}{(M2)}; \drawcrossing[crosstyle=mn,inlefthandlen=0,inrighthandlen=0,outlefthandlen=0.05,outrighthandlen=0.05,yscale=0.5]{($(N2)+(-1,0)$)}{(N2)}; \draw (M1) -- (P0); \drawcoop[CoProdStyle=comodule,leftarmlen=.6,rightarmlen=0,handlelen=.8,handlept=T2]{($(M2)+(1,0)$)}{(M2)}; \draw (T1)node[above]{$B$} (P1) --(P1 |- T1)node[above]{$X$} (T2)node[below]{$X$} (N2)--(N2 |- T2)node[below]{$B$} ($(T2)+(0.5,0.25)$)node{$\alpha_{X}$}; \end{tikzpicture}.\label{eq cocomm}%
\end{equation}

Note that $\left(  B,\Delta_{B}\right)  \in B$-$\operatorname*{Comod}%
_{\mathcal{C}}$. The next proposition shows that a subobject (subcomodule) of
$B\in B$-$\operatorname*{Comod}_{\mathcal{C}}$ is also a subcoalgebra of $B$
in $\mathcal{C}$.

\begin{theorem}
\label{thm D is coalg}Let $i:\left(  D,\rho_{D}\right)  \rightarrow\left(
B,\Delta_{B}\right)  $ be a subobject of $B\in B$-$\operatorname*{Comod}%
_{\mathcal{C}}$. Then

\begin{enumerate}
\item there exists a unique $\mathcal{C}$-coalgebra structure on $D$ such that
$i$ is a coalgebra morphism (i.e., $D$ is a subcoalgebra of $B$),

\item the category $D\text{-}\operatorname*{Comod}\nolimits_{\mathcal{C}}$ is
a $\mathcal{C}$-module subcategory of $B$-$\operatorname*{Comod}_{\mathcal{C}%
}$.
\end{enumerate}
\end{theorem}

\begin{proof}
\begin{enumerate}
\item Let $\left(  E,f\right)  $ be the cokernel of $i$ in $\mathcal{C}$. Then
$f:B\rightarrow E$ and $\ker f=i$. Since $i$ is a $B$-comodule morphism, we
have $\left(  id_{B}\otimes f\right)  \Delta_{B}i=\left(  id_{B}\otimes
fi\right)  \rho_{D}=0$.

We claim that $\left(  f\otimes id_{B}\right)  \Delta_{B}i=0$. Since
$\mathcal{C}$ is rigid, there exists a natural isomorphism
\begin{align}
\zeta:\operatorname{Hom}_{\mathcal{C}}\left(  D,E\otimes B\right)   &
\rightarrow\operatorname{Nat}\left(  D\otimes id_{\mathcal{C}},E\otimes
id_{\mathcal{C}}\right) \nonumber\\
\tikz[baseline=(current bounding box.west),scale=1,samples=100,thick,font=\scriptsize] { \drawcoop[CoProdStyle=comodule,leftarmlen=1,rightarmlen=1,handlelen=1,yscale=.5]{(0,0)}{(0.5,0)}; \draw (T) node[above] {$D$} (M1) node[below]{$E$} (N1)node[below]{$B$}; \draw[fill=white,thin] (-0.2,0.1) rectangle(0.7,-0.3); \draw ($(-0.2,0.1)!0.5!(0.7,-0.3)$) node {$t$};}
& \mapsto
\tikz[baseline=(current bounding box.west),scale=1,samples=100,thick,font=\scriptsize] { \drawcoop[CoProdStyle=comodule,leftarmlen=2,rightarmlen=1.6,handlelen=1,yscale=.5]{(0,0)}{(0.5,0)}; \draw (T) node[above] {$D$} (M1) node[below]{$E$}; \draw[fill=white,thin] (-0.2,0.1) rectangle(0.7,-0.3); \draw ($(-0.2,0.1)!0.5!(0.7,-0.3)$) node {$t$}; \drawcoop[CoProdStyle=comodule,leftarmlen=2.05,rightarmlen=0.5,handlelen=0.8,yscale=.5]{($(N1)+(0.5,0)$)}{(N1)}; \draw (T) node[below] {$\bullet$} (M1) node[above]{$\bullet$}; }\ ,\label{eq na iso}%
\end{align}
via the composition of following isomorphisms
\begin{subequations}
\begin{align*}
\operatorname{Hom}_{\mathcal{C}}\left(  D,E\otimes B\right)   &
-\!\!-\!\!\!\longrightarrow\operatorname{Hom}_{\mathcal{C}}\left(  E^{\ast
}\otimes D,B\right) \\
& \overset{\theta_{E^{\ast}\otimes D}}{-\!\!-\!\!\!\longrightarrow
}\operatorname{Nat}\left(  E^{\ast}\otimes D\otimes id_{\mathcal{C}%
},id_{\mathcal{C}}\right) \\
& -\!\!-\!\!\!\longrightarrow\operatorname{Nat}\left(  D\otimes
id_{\mathcal{C}},E\otimes id_{\mathcal{C}}\right)  .
\end{align*}
Denoted $\rho_{D}$ by
$\tikz[baseline=(current bounding box.west),xscale=0.5,samples=100,thick,font=\scriptsize] {	\drawcoop[CoProdStyle=coproduc,leftarmlen=0.3,rightarmlen=0.3,handlelen=.5]{(0,0)}{(1,0)};
\draw (T) node[above] {$D$} (M1) node[below] {$B$} (N1) node[below] {$D$};
}$, then $\begin{tikzpicture}[xscale=.6,yscale=.8,samples=100,thick,font=%
\scriptsize,baseline=(current bounding box.west)]
\drawcoop[CoProdStyle=coproduct,leftarmlen=.5,rightarmlen=.5,handlelen=1.5,handletext={rectangle/0.5/i}]{(0,0)}{(1,0)}; \draw (T)node[above]{$D$} (M1)node[below]{$B$} (N1)node[below]{$B$}; \end{tikzpicture}=\begin{tikzpicture}[xscale=.6,yscale=.8,samples=100,thick,font=%
\scriptsize,baseline=(current bounding box.west)]
\drawcoop[CoProdStyle=comodule,leftarmlen=0.9,rightarmlen=0.9,handlelen=0.5,rightarmtext={rectangle/0.5/i}]{(0,0)}{(1,0)}; \draw (T)node[above]{$D$} (M1)node[below]{$B$} (N1)node[below]{$B$}; \end{tikzpicture}
$. So we have
\end{subequations}
\[
\begin{tikzpicture}[xscale=.6,samples=100,thick,font=\scriptsize,baseline=(current bounding box.west)] \drawcoop[CoProdStyle=coproduct,leftarmlen=1.4,rightarmlen=0.7,handlelen=1.4,handlept=T1,leftarmpt=M3,handletext={rectangle/0.5/i},leftarmtext={rectangle/0.6/f}]{(0,0)}{(1,0)}; \drawcoop[CoProdStyle=comodule,leftarmlen=0.1,rightarmlen=.1,handlelen=1.2,handlept=T2]{($(N1)+(1,0)$)}{(N1)}; \draw (T1)node[above]{$D$} (M1) --(M1 |- T1)node[above]{$\bullet$} (T2)node[below]{$\bullet$} (M3)--(M3 |- T2)node[below]{$E$}; \end{tikzpicture}\overset
{(\ref{eq cocomm})}{=\!\!=\!\!=}%
\begin{tikzpicture}[xscale=.6,yscale=.8,samples=100,thick,font=\scriptsize,baseline=(current bounding box.west)] \drawcoop[CoProdStyle=coproduct,leftarmlen=.5,rightarmlen=.5,handlelen=1.2,handlept=T1,leftarmpt=P0,handletext={rectangle/0.5/i}]{(0,0)}{(1,0)}; \drawcrossing[crosstyle=mn,inlefthandlen=0.05,inrighthandlen=0.1,outlefthandlen=0.05,outrighthandlen=0.05,inrightnm=P1,yscale=0.5]{(N1)}{($(N1)+(1,0)$)}; \drawcrossing[crosstyle=mn,inlefthandlen=0.05,inrighthandlen=0.05,outlefthandlen=0.05,outrighthandlen=0.05,outrightnm=P2,yscale=0.5]{(N2)}{(M2)}; \drawcrossing[crosstyle=mn,inlefthandlen=0,inrighthandlen=0,outlefthandlen=0.05,outrighthandlen=0.05,yscale=0.5]{($(N2)+(-1,0)$)}{(N2)}; \draw (M1) -- (P0); \drawcoop[CoProdStyle=comodule,leftarmlen=.6,rightarmlen=0,handlelen=1.4,handlept=T2]{($(M2)+(1,0)$)}{(M2)}; \draw (T1)node[above]{$D$} (P1) --(P1 |- T1)node[above]{$\bullet$} (T2)node[below]{$\bullet$} (N2 |- T2)node[below]{$E$}; \linewithtext[text={rectangle/0.45/f}]{(N2)}{(N2 |- T2)}; \end{tikzpicture}=\begin{tikzpicture}[xscale=.6,yscale=.8,samples=100,thick,font=\scriptsize,baseline=(current bounding box.west)] \drawcoop[CoProdStyle=comodule,leftarmlen=.5,rightarmlen=0.6,handlelen=0.4,handlept=T1,leftarmpt=P0,rightarmtext={rectangle/0.5/i}]{(0,0)}{(1,0)}; \drawcrossing[crosstyle=mn,inlefthandlen=0.05,inrighthandlen=0.1,outlefthandlen=0.05,outrighthandlen=0.05,inrightnm=P1,yscale=0.5]{(N1)}{($(N1)+(1,0)$)}; \drawcrossing[crosstyle=mn,inlefthandlen=0.05,inrighthandlen=0.05,outlefthandlen=0.05,outrighthandlen=0.05,outrightnm=P2,yscale=0.5]{(N2)}{(M2)}; \drawcrossing[crosstyle=mn,inlefthandlen=0,inrighthandlen=0,outlefthandlen=0.05,outrighthandlen=0.05,yscale=0.5]{($(N2)+(-1,0)$)}{(N2)}; \draw (M1) -- (P0); \drawcoop[CoProdStyle=comodule,leftarmlen=.6,rightarmlen=0,handlelen=1.4,handlept=T2]{($(M2)+(1,0)$)}{(M2)}; \draw (T1)node[above]{$D$} (P1) --(P1 |- T1)node[above]{$\bullet$} (T2)node[below]{$\bullet$} (N2 |- T2)node[below]{$E$}; \linewithtext[text={rectangle/0.45/f}]{(N2)}{(N2 |- T2)}; \end{tikzpicture}=\begin{tikzpicture}[xscale=.6,yscale=.8,samples=100,thick,font=\scriptsize,baseline=(current bounding box.west)] \drawcoop[CoProdStyle=comodule,leftarmlen=.5,rightarmlen=0.2,handlelen=0.4,handlept=T1,leftarmpt=P0]{(0,0)}{(1,0)}; \drawcrossing[crosstyle=mn,inlefthandlen=0.05,inrighthandlen=0.1,outlefthandlen=0.05,outrighthandlen=0.05,inrightnm=P1,yscale=0.5]{(N1)}{($(N1)+(1,0)$)}; \drawcrossing[crosstyle=mn,inlefthandlen=0.05,inrighthandlen=0.05,outlefthandlen=0.05,outrighthandlen=0.05,outrightnm=P2,yscale=0.5]{(N2)}{(M2)}; \drawcrossing[crosstyle=mn,inlefthandlen=0,inrighthandlen=0,outlefthandlen=0.05,outrighthandlen=0.05,yscale=0.5]{($(N2)+(-1,0)$)}{(N2)}; \draw (M1) -- (P0); \drawcoop[CoProdStyle=comodule,leftarmlen=.6,rightarmlen=0,handlelen=2.4,handlept=T2]{($(M2)+(1,0)$)}{(M2)}; \draw (T1)node[above]{$D$} (P1) --(P1 |- T1)node[above]{$\bullet$} (T2)node[below]{$\bullet$} (N2 |- T2)node[below]{$E$}; \linewithtext[text={{rectangle/0.7/f},{rectangle/0.2/i}}]{(N2)}{(N2 |- T2)}; \end{tikzpicture}=0,
\]
which shows that $\zeta\left(  \left(  f\otimes id_{B}\right)  \Delta
_{B}i\right)  =0$, where $\zeta$ is the isomorphism (\ref{eq na iso}). Thus
$\left(  f\otimes id_{B}\right)  \Delta_{B}i=0$. By Lemma \ref{lemma
intersection}, there exists a unique $\Delta_{D}:D\rightarrow D\otimes D$,
such that $\Delta_{B}i=\left(  i\otimes i\right)  \Delta_{D}$.

Let $\varepsilon_{D}=\varepsilon_{B}i:D\rightarrow1$. We need to check that
$\left(  D,\Delta_{D},\varepsilon_{D}\right)  $ is a $\mathcal{C}$-coalgebra.
First, it follows from the counit axiom of $B$ that
\[
\left(  \varepsilon_{D}\otimes i\right)  \Delta_{D}=\left(  \varepsilon
_{B}\otimes id_{B}\right)  \left(  i\otimes i\right)  \Delta_{D}=\left(
\varepsilon_{B}\otimes id_{B}\right)  \Delta_{B}i=i.
\]
Since $i$ is monic, $\left(  \varepsilon_{D}\otimes id_{D}\right)  \Delta
_{D}=id_{D}$. Similarly, $\left(  id_{D}\otimes\varepsilon_{D}\right)
\Delta_{D}=id_{D}$. To show the coassociativity, it suffices to show that%
\[
\left(  i\otimes i\otimes i\right)  \left(  \Delta_{D}\otimes id_{D}\right)
\Delta_{D}=\left(  i\otimes i\otimes i\right)  \left(  id_{D}\otimes\Delta
_{D}\right)  \Delta_{D},
\]
which follows directly from the coassociativity of $B$. Consequently, $\left(
D,\Delta_{D},\varepsilon_{D}\right)  $ is a coalgebra in $\mathcal{C} $, and
$i:D\rightarrow B$ is a coalgebra map.

\item Let $\left(  V,\tilde{\rho}_{V}\right)  $ be a left $D$-comodule in
$\mathcal{C}$. Then $V$ is a left $B$-comodule via $\rho_{V}=\left(  i\otimes
id_{V}\right)  \tilde{\rho}_{V}$. For $V,W\in D$-$\operatorname*{Comod}%
_{\mathcal{C}},$ $\operatorname{Hom}_{\mathcal{C}}^{D}\left(  V,W\right)
=\operatorname{Hom}_{\mathcal{C}}^{B}\left(  V,W\right)  .$ So the category
$D$-$\operatorname*{Comod}_{\mathcal{C}}$ is a full subcategory of
$B$-$\operatorname*{Comod}_{\mathcal{C}}$, and it's clearly closed under the
$\mathcal{C}$-module product. Thus $D$-$\operatorname*{Comod}_{\mathcal{C}}$
is a $\mathcal{C}$-module subcategory of $B$-$\operatorname*{Comod}%
_{\mathcal{C}}$.
\end{enumerate}
\end{proof}

In the following proposition, we give some equivalence conditions for the
indecomposability of $D$-$\operatorname*{Comod}\nolimits_{\mathcal{C}}$.

\begin{proposition}
\label{prop indecom equivs} Let $i:\left(  D,\rho_{D}\right)  \rightarrow
\left(  B,\Delta_{B}\right)  $ be a subobject of $B\in B$%
-$\operatorname*{Comod}_{\mathcal{C}}$. Then the following statements are equivalent.

\begin{enumerate}
\item $D$ is indecomposable in $B$-$\operatorname*{Comod}_{\mathcal{C}}$.

\item $D$ is indecomposable in $D$-$\operatorname*{Comod}_{\mathcal{C}}$.

\item $D$ is an indecomposable $\mathcal{C}$-coalgebra.

\item The $\mathcal{C}$-module category $D$-$\operatorname*{Comod}%
\nolimits_{\mathcal{C}}$ is indecomposable.
\end{enumerate}
\end{proposition}

\begin{proof}
Obviously, in $\mathcal{C}$, each subcoalgebra of $D$ is a $D$-subcomodule of
$D$, and each $D$-subcomodule of $D$ is a $B$-subcomodule of $D$, so the
implications (1)$\Rightarrow$(2)$\Rightarrow$(3) are clear.

Given a $B$-subcomodule $\left(  D_{1},\rho_{1}\right)  $ of $D$ with
monomorphism $j:D_{1}\rightarrow D$, $D_{1}$ is clearly a $B$-subcomodule of
$B$, and thus it's a subcoalgebra of $B$ by Theorem \ref{thm D is coalg}. Thus
there exists a coproduct $\Delta_{1}:D_{1}\rightarrow D_{1}\otimes D_{1}$ in
$\mathcal{C}$, such that
\[
\left(  ij\otimes ij\right)  \Delta_{1}=\Delta_{B}ij,\text{ }\rho_{1}=\left(
ij\otimes id_{D_{1}}\right)  \Delta_{1},
\]
and a counit $\varepsilon_{1}=\varepsilon_{B}ij=\varepsilon_{D}j$. It's easy
to see that $j:D_{1}\rightarrow D$ is a coalgebra map.

If we assume further that $j$ splits in $B$-$\operatorname*{Comod}%
_{\mathcal{C}}$ and $p$ is a retraction of $j$, then we have $\Delta
_{1}p=\left(  p\otimes p\right)  \Delta_{D}$. It follows that if $D$ is
decomposable in $B$-$\operatorname*{Comod}_{\mathcal{C}}$, then $D$ is
decomposable as $\mathcal{C}$-coalgebras. So we get (3)$\Rightarrow$(1).

(2)$\Rightarrow$(4). Assume that $D$-$\operatorname*{Comod}_{\mathcal{C}%
}=\mathcal{M}_{1}\oplus\mathcal{M}_{2}$, where $\mathcal{M}_{1},\mathcal{M}%
_{2}$ are nontrivial $\mathcal{C}$-module subcategories of $D$%
-$\operatorname*{Comod}_{\mathcal{C}}$.

For any $M\in D$-$\operatorname*{Comod}_{\mathcal{C}}$, there exist $M_{1}%
\in\mathcal{M}_{1},\ M_{2}\in\mathcal{M}_{2}$ such that $M=M_{1}\oplus M_{2}$.
If $M,N\in D$-$\operatorname*{Comod}_{\mathcal{C}}$, and $f\in
\operatorname{Hom}_{\mathcal{C}}^{D}\left(  M,N\right)  $, then $f=\left(
f_{1},f_{2}\right)  $, where $f_{1}\in\operatorname{Hom}_{\mathcal{M}_{1}%
}\left(  M_{1},N_{1}\right)  ,$ $f_{2}\in\operatorname{Hom}_{\mathcal{M}_{2}%
}\left(  M_{2},N_{2}\right)  ,$ $M=M_{1}\oplus M_{2}$, $N=N_{1}\oplus N_{2}$.

As an object of $D$-$\operatorname*{Comod}_{\mathcal{C}}$, $D=D_{1}\oplus
D_{2}$, where $D_{1}\in\mathcal{M}_{1}$, $D_{2}\in\mathcal{M}_{2}$. Take a
nonzero object $N\in\mathcal{M}_{1}$, then the object $\left(  D\otimes
N,\rho_{D}\otimes id_{N}\right)  \in D$-$\operatorname*{Comod}_{\mathcal{C}}$,
and $\rho_{N}:N\rightarrow D\otimes N$ is a $D$-comodule map. For $i=1,2$,
$\left(  D\otimes N\right)  _{i}=D_{i}\otimes N\in\mathcal{M}_{i}$, as
$\mathcal{M}_{i}$ is closed under right $\mathcal{C}$-module product. So
$\rho_{N}=\left(  \rho_{N}\right)  _{1}:N\rightarrow D_{1}\otimes N$. Note
that $\rho_{N}$ is monic, so $D_{1}\neq0$. Similarly, $D_{2}\neq0$. Hence, $D$
is decomposable in $D$-$\operatorname*{Comod}_{\mathcal{C}}$.

(4)$\Rightarrow$(1). Assume that $D=D_{1}\oplus D_{2}$ is a direct sum of $B
$-subcomodules in $\mathcal{C}$. For $j=1,2$, let $i_{j}:D_{j}\rightarrow D $
and $p_{j}:D\rightarrow D_{j}$ be the canonical injections and projections.
Then the direct sum $D=D_{1}\oplus D_{2}$ can be viewed as in category
$D$-$\operatorname*{Comod}_{\mathcal{C}}$, and also as $\mathcal{C}$-coalgebras.

Given a left $D$-comodule $\left(  M,\rho_{M}\right)  $ in $\mathcal{C}$,
define maps%
\[
f_{j}=\left(  \varepsilon_{D_{j}}p_{j}\otimes id_{M}\right)  \rho_{M},\text{
}j=1,2.
\]
Since $\varepsilon_{D_{j}}=\varepsilon_{D}i_{j}$, $f_{1}+f_{2}=\left(
\varepsilon_{D}\otimes id_{M}\right)  \rho_{M}=id_{M}$. We easily get that
\begin{align*}
\rho_{M}f_{j}  & =\left(  \varepsilon_{D_{j}}p_{j}\otimes id_{D}\otimes
id_{M}\right)  \left(  id_{D}\otimes\rho_{M}\right)  \rho_{M}\\
& =\left(  \varepsilon_{D_{j}}p_{j}i_{j}p_{j}\otimes i_{j}p_{j}\otimes
id_{M}\right)  \left(  \Delta_{D}\otimes id_{M}\right)  \rho_{M}\\
& =\left(  i_{j}p_{j}\otimes id_{M}\right)  \rho_{M},
\end{align*}
thus $f_{j}f_{l}=\delta_{jl}f_{j},$ for $j,l=1,2$. It's then easy to verify
that $f_{1}$, $f_{2}$ are $D$-colinear. Therefore $\left\{  f_{1}%
,f_{2}\right\}  $ is a complete set of orthogonal idempotents in
$\operatorname{End}_{\mathcal{C}}^{D}\left(  M\right)  $.

Now setting $M_{j}=\operatorname{Im}f_{j}$, we have $M=M_{1}\oplus M_{2}$ as
$D$-comodules. Let $\rho_{j}$ be the $D$-coaction on $M_{j}$. Then one may
check that $M_{j}\in D_{j}$-$\operatorname*{Comod}_{\mathcal{C}}$ via
\[
\tilde{\rho}_{_{j}}:M_{j}\overset{\rho_{_{j}}}{\longrightarrow}D\otimes
M_{j}\overset{p_{j}\otimes id_{M_{j}}}{\longrightarrow}D_{j}\otimes M_{j},
\]
and that $\rho_{_{j}}=\left(  i_{j}\otimes id_{M_{j}}\right)  \tilde{\rho}_{j}
$.

Let $\left(  N_{1},\tilde{\rho}_{N_{1}}\right)  \in D_{1}$%
-$\operatorname*{Comod}_{\mathcal{C}}$, $\left(  N_{2},\tilde{\rho}_{N_{2}%
}\right)  \in D_{2}$-$\operatorname*{Comod}_{\mathcal{C}}$. Then $N_{j}$ can
be viewed as a natural left $D$-comodule via $\rho_{N_{j}}=\left(
i_{j}\otimes id_{N_{j}}\right)  \tilde{\rho}_{N_{j}}$. For any $f\in
\operatorname{Hom}_{\mathcal{C}}^{D}\left(  N_{1},N_{2}\right)  $, we have
$\rho_{N_{2}}f=\left(  id_{D}\otimes f\right)  \rho_{N_{1}}$. Applying
$i_{2}p_{2}\otimes id_{N_{2}}$ to both side, we get
\[
\rho_{N_{2}}f=\left(  i_{2}p_{2}\otimes f\right)  \rho_{N_{1}}=\left(
i_{2}p_{2}i_{1}\otimes f\right)  \tilde{\rho}_{N_{1}}=0,
\]
and thus $f=0$ and $\operatorname{Hom}_{\mathcal{C}}^{D}\left(  N_{1}%
,N_{2}\right)  =0$. Similarly, $\operatorname{Hom}_{\mathcal{C}}^{D}\left(
N_{2},N_{1}\right)  =0$. So
\[
D\text{-}\operatorname*{Comod}\nolimits_{\mathcal{C}}=D_{1}\text{-}%
\operatorname*{Comod}\nolimits_{\mathcal{C}}\oplus D_{2}\text{-}%
\operatorname*{Comod}\nolimits_{\mathcal{C}}%
\]
as $\mathcal{C}$-module categories, and (4)$\Rightarrow$(1) is done.
\end{proof}

Now assume that $k$ is an algebraically closed field of characteristic zero,
and $\mathcal{C}$ is a finite braided multifusion category. Note that
$\mathcal{C}$ has a natural module category structure over
$\mathcal{C\boxtimes C}^{op}$, and the dual category is the Drinfeld center
$\mathcal{Z}_{l}\left(  \mathcal{C}\right)  $ (see \cite[Corollary
3.37]{Etingof2004finite}). It due to Etingof, Nikshych and Ostrik
\cite[Theorem 2.18]{Etingof2005On} that for any module category $\mathcal{M}$
over a multifusion category $\mathcal{C}$ the dual category $\mathcal{C}%
_{\mathcal{M}}^{\ast}$ is semisimple. In particular, the Drinfeld center
$\mathcal{Z}_{l}\left(  \mathcal{C}\right)  $ of $\mathcal{C}$ is semisimple.
By Theorem \ref{theorem center iso H-comod}, the category $B$%
-$\operatorname*{Comod}_{\mathcal{C}}\cong\mathcal{Z}_{l}\left(
\mathcal{C}\right)  $ is semisimple.

As an object of $B$-$\operatorname*{Comod}_{\mathcal{C}}$, $B$ is a direct sum
of simple subobjects. By Proposition \ref{prop indecom equivs} each simple
subobject of $B$ is an indecomposable coalgebra in $\mathcal{C}$, and the
following proposition is immediate.

\begin{proposition}
\label{decom of B-comod}Let $\mathcal{C}$ be a finite braided multifusion
category over an algebraically closed field $k$ of characteristic zero, and
$B\cong D_{1}\oplus D_{2}\oplus\cdots\oplus D_{r}$ be a direct sum of simple
objects in $B$-$\operatorname*{Comod}_{\mathcal{C}}$, then $\mathcal{Z}%
_{l}\left(  \mathcal{C}\right)  \cong B$-$\operatorname*{Comod}%
\nolimits_{\mathcal{C}}\cong%
{\textstyle\bigoplus_{j=1}^{r}}
D_{j}$-$\operatorname*{Comod}\nolimits_{\mathcal{C}}$ is a direct sum of
indecomposable $\mathcal{C}$-module subcategories.
\end{proposition}

If $H$ is a semisimple quasi-triangular Hopf algebra and $\mathcal{C}={}%
_{H}\mathcal{M}$ is the category of finite dimensional representations, this
decomposition has already appeared in the authors' paper \cite{LiuZhu2019On},
as in the following example.

\begin{example}
[\cite{LiuZhu2019On}]\label{ex QT HA}Let $\left(  H,R\right)  $ be a
semisimple quasi-triangular Hopf algebra. The automorphism braided group
$H_{R}$ of $\mathcal{C}={}_{H}\mathcal{M}$ is constructed as follows. As an
$H$-module algebra, $H_{R}=H$ with the left adjoint action $\cdot_{ad}$. The
comultiplication and antipode are defined by
\[
\Delta_{R}\left(  h\right)  =h_{\left(  1\right)  }S\left(  R^{2}\right)
\otimes R^{1}\cdot_{ad}h_{\left(  2\right)  },\ S_{R}\left(  h\right)
=R^{2}S\left(  R^{1}\cdot_{ad}h\right)  ,\ h\in H.
\]
The decomposition of the automorphism braided group $H_{R}$ is the unique
decomposition $H_{R}=D_{1}\oplus\cdots\oplus D_{r}$ of the minimal
$H$-adjoint-stable subcoalgebras $D_{1},...,D_{r}$ of $H_{R}$, and
the~category${}$ $_{H}^{H}\mathcal{YD=Z}_{l}\left(  \mathcal{C}\right)
\cong{}_{H}^{H_{R}}\mathcal{M}\cong{}_{H}^{D_{1}}\mathcal{M}\oplus\cdots
\oplus{}_{H}^{D_{r}}\mathcal{M}$ is a direct sum of indecomposable right
$\mathcal{C}$-module subcategories.
\end{example}

In literature \cite{Ostrik2003module,Etingof2004finite}, the concept of
internal Hom plays a crucial role in the study of module categories. Once the
internal Hom is determined, Theorem~\ref{thm Ostrik}, Theorem~\ref{thm liu
zhu} can be applied to characterize indecomposable $\mathcal{C}$-module subcategories.

Now let $\mathcal{C}$ be a multitensor category and $D$ be a coalgebra in
$\mathcal{C}$. Naturally, $D$-$\operatorname*{Comod}_{\mathcal{C}}$ is a right
$\mathcal{C}$-module category. We end this section by presenting a
characterization of the internal Hom for $D$-$\operatorname*{Comod}%
_{\mathcal{C}}$.

First we need the notion of cotensor product over a coalgebra $D$ in
$\mathcal{C}$.

\begin{definition}
\label{def cotensor}Let $M,N$ be respectively a right $D$-comodule and a left
$D$-comodule in $\mathcal{C}$ with structure maps $\rho_{M},\rho_{N}$. The
cotensor product $M\square_{D}^{\mathcal{C}}N$ of $M$ and $N$ over $D$ is the
equalizer of the diagram%
\begin{equation}
M\square_{D}^{\mathcal{C}}N\subseteq M\otimes
N\;\tikz[baseline=-.3ex] {\draw[->] (0,.8ex) -- node[above]{$\rho _{M}\otimes id_{N}$}(3cm,0.8ex); \draw[->] (0,0ex) -- node[below]{$id_{M}\otimes\rho _{N}$}(3cm,0ex);}\;M\otimes
D\otimes N.\label{cotensor product}%
\end{equation}
That is, $M\square_{D}^{\mathcal{C}}N$ is the kernel of the morphism $\rho
_{M}\otimes id_{N}-id_{M}\otimes\rho_{N}$.
\end{definition}

Let $\left(  M,\rho_{M}\right)  \in D$-$\operatorname*{Comod}_{\mathcal{C}}$.
Then ${^{\ast}\hspace{-0.5ex}M}$ has a natural right $D$-comodule structure
$\rho_{{^{\ast}\hspace{-0.5ex}M}}$, which is the image of $\rho_{M}$ under the
composition of the isomorphisms%
\[
\operatorname{Hom}_{\mathcal{C}}\left(  M,D\otimes M\right)  \overset{\cong
}{\longrightarrow}\operatorname{Hom}_{\mathcal{C}}\left(  D^{\ast}\otimes
M,M\right)  \overset{\cong}{\longrightarrow}\operatorname{Hom}_{\mathcal{C}%
}\left(  {^{\ast}\hspace{-0.5ex}M},{^{\ast}\hspace{-0.5ex}M}\otimes D\right)
.
\]
The graphical representation of $\rho_{{^{\ast}\hspace{-0.5ex}M}}$ is
\[
\rho_{{^{\ast}\hspace{-0.5ex}M}}%
=\tikz[baseline=(current bounding box.west),xscale=1.2,samples=100,thick,font=\scriptsize] { \coordinate (X) at (0,0); \coordinate (Y) at (0.5,0); \drawcoop[CoProdStyle=comodule,leftarmlen=2.5,rightarmlen=1.3,handlelen=1.3,yscale=.5]{(X)}{(Y)}; \draw ($(T)+(0.2,0.15)$) node {$M$} (M1) node[below]{$D$}; \draw[fill=white,thin] ($(X)+(-0.2,0.1)$) rectangle($(Y)+(0.2,-0.3)$); \draw ($($(X)+(-0.1,0.1)$)!0.5!($(Y)+(0.1,-0.3)$)$) node {$\rho _{M}$}; \halfcircle{(T)}{($(T)+(-0.75,0)$)}; \draw ($(T)+(-1,0.15)$)node{${^{*}\hspace{-0.5ex}M}$}; \halfcircle{(N1)}{($(N1)+(0.75,0)$)}; \draw ($(N1)+(-0.15,-0.2)$) node {$M$} ($(N1)+(0.75,0)$)--($(N1)+(0.75,1.8)$)node[above]{${^{*}\hspace{-0.5ex}M}$}; \coordinate (Z) at ($(T)+(-0.75,0)$); \draw(Z)--(Z|-M1)node[below]{${^{*}\hspace{-0.5ex}M}$}; }.
\]

\begin{proposition}
\label{prop ihom for D-comod}Let $M,N\in D$-$\operatorname*{Comod}%
_{\mathcal{C}}$. Then $\underline{\operatorname{Hom}}\left(  M,N\right)
={^{\ast}\hspace{-0.5ex}M}\square_{D}^{\mathcal{C}}N$, i.e., the functor
${^{\ast}\hspace{-0.5ex}M}\square_{D}^{\mathcal{C}}\bullet{}:D$%
-$\operatorname*{Comod}_{\mathcal{C}}\rightarrow\mathcal{C}$ is a right
adjoint of $M\otimes\bullet$.
\end{proposition}

\begin{proof}
It suffices to show that there is a natural isomorphism
\[
\operatorname{Hom}_{\mathcal{C}}\left(  X,{^{\ast}\hspace{-0.5ex}M}\square
_{D}^{\mathcal{C}}N\right)  \cong\operatorname{Hom}_{\mathcal{C}}^{D}\left(
M\otimes X,N\right)  ,
\]
natural in $X\in\mathcal{C}$. We will show that the required isomorphism can
be deduced from the composition
\begin{equation}
\operatorname{Hom}_{\mathcal{C}}\left(  X,{^{\ast}\hspace{-0.5ex}M}\square
_{D}^{\mathcal{C}}N\right)  \overset{j_{\ast}}{\longrightarrow}%
\operatorname{Hom}_{\mathcal{C}}\left(  X,{^{\ast}\hspace{-0.5ex}M}\otimes
N\right)  \overset{\cong}{\longrightarrow}\operatorname{Hom}_{\mathcal{C}%
}\left(  M\otimes X,N\right)  ,\label{in Hom iso}%
\end{equation}
where $j:{^{\ast}\hspace{-0.5ex}M}\square_{D}^{\mathcal{C}}N\rightarrow
{}{^{\ast}\hspace{-0.5ex}M}\otimes N$ is the natural monomorphism in
(\ref{cotensor product}). We need to show that the image of this composition
is equal to $\operatorname{Hom}_{\mathcal{C}}^{D}\left(  M\otimes X,N\right)
$. Let $f\in\operatorname{Hom}_{\mathcal{C}}\left(  X,{}{^{\ast}%
\hspace{-0.5ex}M}\otimes N\right)  $, then $f\in\operatorname{Im}j_{\ast}$ if
and only if
\begin{equation}
\left(  \rho_{^{\ast}\hspace{-0.5ex}M}\otimes id_{N}\right)  f=\left(
id_{{^{\ast}\hspace{-0.5ex}M}}\otimes\rho_{N}\right)  f.\label{eq3}%
\end{equation}
The graphical expression of (\ref{eq3}) is%
\[
\tikz[baseline=(current bounding box.west),scale=1,samples=100,thick,font=\scriptsize] { \coordinate (X) at (0,0); \coordinate (Y) at (0.5,0); \drawcoop[CoProdStyle=comodule,leftarmlen=2.2,rightarmlen=1.3,handlelen=1.3,yscale=.5]{(X)}{(Y)}; \draw ($(T)+(0.2,0.15)$) node {$M$} (M1) node[below]{$D$}; \draw[fill=white,thin] ($(X)+(-0.2,0.1)$) rectangle($(Y)+(0.2,-0.3)$); \draw ($($(X)+(-0.1,0.1)$)!0.5!($(Y)+(0.1,-0.3)$)$) node {$\rho _{M}$}; \halfcircle{(T)}{($(T)+(-0.75,0)$)}; \draw ($(T)+(-1,0.15)$)node{${^{*}\hspace{-0.5ex}M}$}; \halfcircle{(N1)}{($(N1)+(0.75,0)$)}; \draw ($(N1)+(-0.15,-0.2)$) node {$M$}; \coordinate (Z) at ($(T)+(-0.75,0)$); \draw(Z)--(Z|-M1)node[below]{${^{*}\hspace{-0.5ex}M}$}; \coordinate (X) at ($(N1)+(0.75,1)$); \coordinate (Y) at ($(X)+(0.5,0)$); \drawcoop[CoProdStyle=comodule,leftarmlen=2,rightarmlen=3,handlelen=1.9,yscale=.5]{(X)}{(Y)}; \draw[fill=white,thin] ($(X)+(-0.2,0.1)$) rectangle($(Y)+(0.2,-0.3)$); \draw ($($(X)+(-0.1,0.1)$)!0.5!($(Y)+(0.1,-0.3)$)$) node {$f$} (T)node[above]{$X$} (N1)node[below]{$N$}; }=\tikz[baseline=(current bounding box.west),scale=1,samples=100,thick,font=\scriptsize] { \coordinate (X) at (0,0); \coordinate (Y) at (0.5,0); \drawcoop[CoProdStyle=comodule,leftarmlen=2.2,rightarmlen=2.2,handlelen=2.5,yscale=.5]{(X)}{(Y)}; \draw (M1) node[below]{$D$} (N1)node[below]{$N$}; \draw[fill=white,thin] ($(X)+(-0.2,0.1)$) rectangle($(Y)+(0.2,-0.3)$); \draw ($($(X)+(-0.1,0.1)$)!0.5!($(Y)+(0.1,-0.3)$)$) node {$\rho _{N}$}; \coordinate (X) at ($(T)+(-0.75,0)$); \coordinate (Y) at (T); \drawcoop[CoProdStyle=comodule,leftarmlen=2.3,rightarmlen=0.5,handlelen=1.2,yscale=.5]{(X)}{(Y)}; \draw[fill=white,thin] ($(X)+(-0.2,0.1)$) rectangle($(Y)+(0.2,-0.3)$); \draw ($($(X)+(-0.1,0.1)$)!0.5!($(Y)+(0.1,-0.3)$)$) node {$f$} (T)node[above]{$X$} (M1)node[below]{$^{*}\hspace{-0.5EX}M$}; }\text{
}\iff
\tikz[baseline=(current bounding box.west),scale=1,samples=100,thick,font=\scriptsize] { \coordinate (X) at (0,0); \coordinate (Y) at (0.5,0); \drawcoop[CoProdStyle=comodule,leftarmlen=3.5,rightarmlen=2,handlelen=1.3,yscale=.5]{(X)}{(Y)}; \draw (M1) node[below]{$D$} (T)node[above]{$M$}; \draw[fill=white,thin] ($(X)+(-0.2,0.1)$) rectangle($(Y)+(0.2,-0.3)$); \draw ($($(X)+(-0.1,0.1)$)!0.5!($(Y)+(0.1,-0.3)$)$) node {$\rho _{M}$}; \halfcircle{(N1)}{($(N1)+(0.75,0)$)}; \draw ($(N1)+(-0.15,-0.2)$) node {$M$}; \coordinate (X) at ($(N1)+(0.75,0.5)$); \coordinate (Y) at ($(X)+(0.5,0)$); \drawcoop[CoProdStyle=comodule,leftarmlen=1,rightarmlen=2.5,handlelen=3.2,yscale=.5]{(X)}{(Y)}; \draw[fill=white,thin] ($(X)+(-0.2,0.1)$) rectangle($(Y)+(0.2,-0.3)$); \draw ($($(X)+(-0.1,0.1)$)!0.5!($(Y)+(0.1,-0.3)$)$) node {$f$} (T)node[above]{$X$} (N1)node[below]{$N$}; }=\tikz[baseline=(current bounding box.west),scale=1,samples=100,thick,font=\scriptsize] { \coordinate (X) at (0,0); \coordinate (Y) at (0.5,0); \drawcoop[CoProdStyle=comodule,leftarmlen=1.5,rightarmlen=1.5,handlelen=4.2,yscale=.5]{(X)}{(Y)}; \draw (M1) node[below]{$D$} (N1)node[below]{$N$}; \draw[fill=white,thin] ($(X)+(-0.2,0.1)$) rectangle($(Y)+(0.2,-0.3)$); \draw ($($(X)+(-0.1,0.1)$)!0.5!($(Y)+(0.1,-0.3)$)$) node {$\rho _{N}$}; \coordinate (X) at ($(T)+(-0.75,0)$); \coordinate (Y) at (T); \drawcoop[CoProdStyle=comodule,leftarmlen=0.6,rightarmlen=0.5,handlelen=1.2,yscale=.5]{(X)}{(Y)}; \draw[fill=white,thin] ($(X)+(-0.2,0.1)$) rectangle($(Y)+(0.2,-0.3)$); \coordinate (Z) at ($(M1)+(-0.75,0)$);\halfcircle{(Z)}{(M1)}; \draw ($($(X)+(-0.1,0.1)$)!0.5!($(Y)+(0.1,-0.3)$)$) node {$f$} (T)node[above]{$X$} (Z)--(Z|-T)node[above]{$M$} ($(M1)+(0.25,-0.2)$)node{$^{*}\hspace{-0.5EX}M$}; }
\]
equivalently,under the isomorphism $\operatorname{Hom}_{\mathcal{C}}\left(
X,{^{\ast}\hspace{-0.5ex}M}\otimes N\right)  \overset{\cong}{\longrightarrow
}\operatorname{Hom}_{\mathcal{C}}\left(  M\otimes X,N\right)  $ the image of
$f$ is in $\operatorname{Hom}_{\mathcal{C}}^{D}\left(  M\otimes X,N\right)
\subseteq\operatorname{Hom}_{\mathcal{C}}\left(  M\otimes X,N\right)  $.
Hence, we get the isomorphism\linebreak[4] $\operatorname{Hom}_{\mathcal{C}}\left(
X,{^{\ast}\hspace{-0.5ex}M}\square_{D}^{\mathcal{C}}N\right)  \overset{\cong
}{\longrightarrow}\operatorname{Hom}_{\mathcal{C}}^{D}\left(  M\otimes
X,N\right)  $, and the naturality in $X$ is obvious.
\end{proof}

\begin{remark}
The internal Hom for the category $\mathrm{Mod}_{\mathcal{C}}$-$A$ for a
$\mathcal{C}$-algebra $A$ was calculated by Etingof and Ostrik \cite[Example
3.19]{Etingof2004finite}. The proposition is a dual version of their result.
\end{remark}

\begin{remark}
[\cite{LiuZhu2019On}]If we take $\mathcal{C}={}_{H}\mathcal{M}$, the category
of finite dimensional representations of a Hopf algebra $H$, and take $D$ an
$H$-module coalgebra, then $\underline{\operatorname{Hom}}\left(  M,N\right)
=\operatorname{Hom}^{D}\left(  M,N\right)  ,$ for any $M,N\in D$%
-$\operatorname*{Comod}_{\mathcal{C}}$.
\end{remark}

Recall that when the appropriated internal Hom objects exist, there are
definitions of evaluation morphism, multiplication morphism of internal Homs.
If in the case that we can identified the internal Homs as the cotensors the
evaluation, the multiplication morphism has the following form.

For any $M_{1},M_{2},M_{3}\in D$-$\operatorname*{Comod}_{\mathcal{C}}$, the
evaluation morphism is the composition
\[
ev_{M_{1},M_{2}}^{\prime}:M_{1}\otimes\left(  {^{\ast}\hspace{-0.5ex}M}%
_{1}\square_{D}^{\mathcal{C}}M_{2}\right)  \hookrightarrow M_{1}\otimes
{^{\ast}\hspace{-0.5ex}M}_{1}\otimes M_{2}\overset{ev_{M_{1}}^{\prime}\otimes
id_{M_{2}}}{-\!\!\!-\!\!\!-\!\!\!-\!\!\!-\!\!\!-\!\!\!\longrightarrow}M_{2},
\]
and the multiplication morphism of internal Hom is defined as the preimage of
the map
\[
\left(  {^{\ast}\hspace{-0.5ex}M}_{1}\square_{D}^{\mathcal{C}}M_{2}\right)
\otimes\left(  {^{\ast}\hspace{-0.5ex}M}_{2}\square_{D}^{\mathcal{C}}%
M_{3}\right)  \hookrightarrow{^{\ast}\hspace{-0.5ex}M}_{1}\otimes M_{2}%
\otimes{}^{\ast}M_{2}\otimes M_{3}\overset{id_{{^{\ast}\hspace{-0.5ex}M}_{1}%
}\otimes ev_{M_{2}}^{\prime}\otimes id_{M_{3}}}%
{-\!\!\!-\!\!\!-\!\!\!-\!\!\!-\!\!\!-\!\!\!-\!\!\!-\!\!\!-\!\!\!-\!\!\!-\!\!\!-\!\!\!\longrightarrow
}{^{\ast}\hspace{-0.5ex}M}_{1}\otimes M_{3},
\]
under the map
\begin{align*}
j_{\ast}: & \operatorname{Hom}_{\mathcal{C}}\left(  \left(  {^{\ast}%
\hspace{-0.5ex}M}_{1}\square_{D}^{\mathcal{C}}M_{2}\right)  \otimes\left(
{^{\ast}\hspace{-0.5ex}M}_{2}\square_{D}^{\mathcal{C}}M_{3}\right)  ,{^{\ast
}\hspace{-0.5ex}M}_{1}\square_{D}^{\mathcal{C}}M_{3}\right) \\
{} & {\rightarrow} \operatorname{Hom}_{\mathcal{C}}\left(  \left(  ^{\ast
}M_{1}\square_{D}^{\mathcal{C}}M_{2}\right)  \otimes\left(  {^{\ast}%
\hspace{-0.5ex}M}_{2}\square_{D}^{\mathcal{C}}M_{3}\right)  ,{}^{\ast}%
M_{1}\otimes M_{3}\right)  ,
\end{align*}
where $j:{^{\ast}\hspace{-0.5ex}M}\square_{D}^{\mathcal{C}}N\rightarrow
{^{\ast}\hspace{-0.5ex}M}\otimes N$ is the natural monomorphism. It makes
$A={^{\ast}\hspace{-0.5ex}M}\square_{D}^{\mathcal{C}}M$ into an algebra in
$\mathcal{C}$, and ${^{\ast}\hspace{-0.5ex}M}\square_{D}^{\mathcal{C}}N$ a
left $A$-module, for $M,N\in D$-$\operatorname*{Comod}_{\mathcal{C}}$. Now
apply theorem \ref{thm Ostrik} and \ref{thm liu zhu} with $\mathcal{M}%
=D$-$\operatorname*{Comod}_{\mathcal{C}}$, we get:

\begin{proposition}
\label{prop D-comod equiv}Let $\mathcal{C}$ be a finite multitensor category,
and let $D$ be a cosemisimple coalgebra in $\mathcal{C}$. If $M$ is a
generator of $D$-$\operatorname*{Comod}_{\mathcal{C}}$, then $A={}{^{\ast
}\hspace{-0.5ex}M}\square_{D}^{\mathcal{C}}M$ is a semisimple algebra in
$\mathcal{C}$, and the functors
\begin{align*}
F  & ={^{\ast}\hspace{-0.5ex}M}\square_{D}^{\mathcal{C}}\bullet{}%
:D\text{-}\operatorname*{Comod}\nolimits_{\mathcal{C}}\rightarrow
A\text{-}\mathrm{Mod}_{\mathcal{C}}\\
\text{and }G  & =M\otimes_{A}\bullet{}:A\text{-}\mathrm{Mod}_{\mathcal{C}%
}\rightarrow D\text{-}\operatorname*{Comod}\nolimits_{\mathcal{C}}%
\end{align*}
establish an equivalence between $\mathcal{C}$-module categories
$D$-$\operatorname*{Comod}\nolimits_{\mathcal{C}}$ and $A$-$\mathrm{Mod}%
_{\mathcal{C}}$.
\end{proposition}

To sum up, we have the following theorem.

\begin{theorem}
Let $\mathcal{C}$ be a braided finite multitensor category, and $B$ be the
automorphism braided group of $\mathcal{C}$. As an object of $\mathcal{Z}%
_{l}\left(  \mathcal{C}\right)  =B$-$\operatorname*{Comod}_{\mathcal{C}}$,
write $B=B_{1}\oplus\cdots\oplus B_{r}$ as a direct sum of indecomposable subobjects.

\begin{enumerate}
\item Then the decomposition $B=B_{1}\oplus\cdots\oplus B_{r}$ is unique as a
direct sum of indecomposable subobjects, and it is also unique as a direct sum
of indecomposable $\mathcal{C}$-subcoalgebras.

\item The category $\mathcal{Z}_{l}\left(  \mathcal{C}\right)  $ admits a
unique decomposition
\[
\mathcal{Z}_{l}\left(  \mathcal{C}\right)  =B_{1}\text{-}\operatorname*{Comod}%
\nolimits_{\mathcal{C}}\oplus\cdots\oplus B_{r}\text{-}\operatorname*{Comod}%
\nolimits_{\mathcal{C}}%
\]
into the direct sum of indecomposable $\mathcal{C}$-module subcategories.

\item For each $1\leq i\leq r$, let $M_{i}\in B_{i}$-$\operatorname*{Comod}%
_{\mathcal{C}}$ be a nonzero object, and $A_{i}={}{^{\ast}\hspace{-0.5ex}%
M}_{i}\square_{B_{i}}^{\mathcal{C}}M_{i}$. Then $F_{i}={^{\ast}\hspace
{-0.5ex}M}_{i}\square_{B_{i}}^{\mathcal{C}}\bullet{}:B_{i}$%
-$\operatorname*{Comod}\nolimits_{\mathcal{C}}\rightarrow A_{i}$%
-$\mathrm{Mod}_{\mathcal{C}}$ is an equivalence between $\mathcal{C}$-module
categories $B_{i}$-$\operatorname*{Comod}\nolimits_{\mathcal{C}}$ and $A_{i}%
$-$\mathrm{Mod}_{\mathcal{C}}$.
\end{enumerate}
\end{theorem}

\begin{proof}
The statements 1) and 2) follows from Theorem \ref{thm D is coalg}, the proof
of Proposition \ref{prop indecom equivs}; 3) follows from Proposition
\ref{prop D-comod equiv}.
\end{proof}

\section{An Application to Weak Hopf Algebras}

\label{sec-appl to WHA}

In this section, we will visualize the results in the previous two sections by
using the theory of weak Hopf algebras. Weak Hopf algebras was introduced by
B\"{o}hm, Nill, and Szlach\'{a}nyi \cite{Bohm1999weak} and studied extensively
by Nikshych and Vainerman \cite{Nikshych2002Finite}. The category of
finite-dimensional representations of a semisimple weak Hopf algebra is a
multifusion category. On the other hand, it is due to Hayashi
\cite{Hayashi1999ACanonical} and Szlach\'{a}nyi \cite{Szlachanyi2000Finite}
that any multifusion category is equivalent to the category of
finite-dimensional representations of a regular semisimple weak Hopf algebra.
For this reason, the theory of weak Hopf algebras is not merely good examples
for categorical construction but also a helpful tool for discussing
multifusion categories.

This section is arranged as follows. We begin by recalling some preliminaries
of weak Hopf algebras in Section \ref{Sec-prelim-of-WHA}. We then discuss some
properties of module (co)algebras for weak Hopf algebras in Section
\ref{sec-Module-coalgebras}. Next, we consider the braided multitensor
category $\mathcal{C}={}_{H}\mathcal{M}$, where $H$ is a quasi-triangular weak
Hopf algebra. We present Majid's braided reconstruction with $\mathcal{C}%
={}_{H}\mathcal{M}$, and obtain the automorphism braided group $U\left(
\mathcal{C}\right)  $ in Section \ref{sec-ABG-for-QT}, and give the structure
of irreducible Yetter-Drinfeld modules ${}$over $H$ in Section
\ref{sec-stru-YD}.

\subsection{Preliminaries of Weak Hopf Algebras}

\label{Sec-prelim-of-WHA}

Now we recall the definition of weak Hopf algebra, quasi-triangular weak Hopf
algebra, and some basic properties. Our references are
\cite{Bohm1999weak,Nikshych2003Invariants}. We will use the sigma notation:
$\Delta\left(  h\right)  =h_{\left(  1\right)  }\otimes h_{\left(  2\right)
}$ for coproduct and $\rho\left(  m\right)  =m_{\left\langle 0\right\rangle
}\otimes m_{\left\langle 1\right\rangle }$ for right coaction (or $\rho\left(
m\right)  =m_{\left\langle -1\right\rangle }\otimes m_{\left\langle
0\right\rangle }$ for left coaction).

A weak Hopf algebra $H$ over $k$ is a $k$-algebra and also a $k$-coalgebra
with an antipode $S:H\rightarrow H$, such that

\begin{enumerate}
\item $\Delta\left(  xy\right)  =\Delta\left(  x\right)  \Delta\left(
y\right)  $,

\item $\Delta^{2}\left(  1\right)  =1_{\left(  1\right)  }\otimes1_{\left(
2\right)  }1_{\left(  1^{\prime}\right)  }\otimes1_{\left(  2^{\prime}\right)
}=1_{\left(  1\right)  }\otimes1_{\left(  1^{\prime}\right)  }1_{\left(
2\right)  }\otimes1_{\left(  2^{\prime}\right)  }$,

\item $\varepsilon\left(  xyz\right)  =\varepsilon\left(  xy_{\left(
1\right)  }\right)  \varepsilon\left(  y_{\left(  2\right)  }z\right)
=\varepsilon\left(  xy_{\left(  2\right)  }\right)  \varepsilon\left(
y_{\left(  1\right)  }z\right)  $,

\item $x_{\left(  1\right)  }S\left(  x_{\left(  2\right)  }\right)
=\varepsilon\left(  1_{\left(  1\right)  }x\right)  1_{\left(  2\right)  }$,

\item $S\left(  x_{\left(  1\right)  }\right)  x_{\left(  2\right)
}=1_{\left(  1\right)  }\varepsilon\left(  x1_{\left(  2\right)  }\right)  $,

\item $S\left(  x\right)  =S\left(  x_{\left(  1\right)  }\right)  x_{\left(
2\right)  }S\left(  x_{\left(  3\right)  }\right)  $,
\end{enumerate}

for all $x,y,z\in H$, where $\Delta$ is the coproduct and $\varepsilon$ is the counit.

The target and the source counital maps $\varepsilon_{t},\varepsilon
_{s}:H\rightarrow H$ are defined by
\[
\varepsilon_{t}\left(  x\right)  =\varepsilon\left(  1_{\left(  1\right)
}x\right)  1_{\left(  2\right)  },\quad\varepsilon_{s}\left(  x\right)
=1_{\left(  1\right)  }\varepsilon\left(  x1_{\left(  2\right)  }\right)  ,
\]
for all $x\in H$. The images of these counital maps, denoted by $H_{t}%
=\varepsilon_{t}\left(  H\right)  $ and $H_{s}=\varepsilon_{s}\left(
H\right)  $.

A weak Hopf algebra $H$ is regular if the restriction of $S^{2}$ on
$H_{t}H_{s}$ is identity map. We will always assume that the weak Hopf
algebras we considered are regular.

If $H$ is a weak Hopf algebra, for all $g,h\in H,$ the following conditions
hold:%
\begin{align}
& \varepsilon\left(  hg\right)  =\varepsilon\left(  \varepsilon_{s}\left(
h\right)  g\right)  =\varepsilon\left(  h\varepsilon_{t}\left(  g\right)
\right)  ,\label{WHArelation1}\\
& \varepsilon_{t}\circ S=\varepsilon_{t}\circ\varepsilon_{s}=S\circ
\varepsilon_{s},\label{WHArelation2}\\
& \varepsilon_{s}\circ S=\varepsilon_{s}\circ\varepsilon_{t}=S\circ
\varepsilon_{t},\label{WHArelation3}\\
& 1_{\left(  1\right)  }h\otimes1_{\left(  2\right)  }=h_{\left(  1\right)
}\otimes\varepsilon_{t}\left(  h_{\left(  2\right)  }\right)
,\label{WHArelation4}\\
& 1_{\left(  1\right)  }\otimes h1_{\left(  2\right)  }=\varepsilon_{s}\left(
h_{\left(  1\right)  }\right)  \otimes h_{\left(  2\right)  }%
.\label{WHArelation5}%
\end{align}

Let $H$ be a weak Hopf algebra, and $\mathcal{C}={}_{H}\mathcal{M}$ be the
category of finite dimensional left $H$-modules. For any $M,N\in{}%
_{H}\mathcal{M}$,
\[
M\otimes_{t}N=\Delta\left(  1\right)  \left(  M\otimes_{k}N\right)  \subseteq
M\otimes_{k}N
\]
with $H$-action given by $h\cdot\left(  1_{\left(  1\right)  }m\otimes
1_{\left(  2\right)  }n\right)  =h_{\left(  1\right)  }m\otimes h_{\left(
2\right)  }n$, for $h\in H$, $m\in M$, $n\in N$. The subalgebra $H_{t}$ is an
$H$-module via $h\cdot z=\varepsilon_{t}\left(  hz\right)  $, where $h\in H$,
$z\in H_{t}$. Furthermore, $H_{t}$ is the unit object of $\mathcal{C}$. The
functorial unit isomorphism $l_{X}:H_{t}\otimes_{t}X\rightarrow X$ and
$r_{X}:X\otimes_{t}H_{t}\rightarrow X$ are defined by
\[
l_{X}\left(  1_{\left(  1\right)  }z\otimes1_{\left(  2\right)  }x\right)
=zx,\quad r_{X}\left(  1_{\left(  1\right)  }x\otimes1_{\left(  2\right)
}z\right)  =S\left(  z\right)  x,\quad z\in H_{t},x\in X.
\]
Then $\left(  _{H}\mathcal{M},\otimes_{t},H_{t},l,r\right)  $ is a monoidal
category. Using the isomorphism $l_{X}$ and $r_{X}$ identifying $H_{t}%
\otimes_{t}X$, $X\otimes_{t}H_{t}$ and $X$, we see that the monoidal category
$_{H}\mathcal{M}$ is strict. In addition, if $M\in{}_{H}\mathcal{M} $, then
there is a canonical $H_{t}$-$H_{t}$-bimodule structure on $M$, such that
$M\otimes_{t}N\cong M\otimes_{H_{t}}N$ (see. \cite{Bohm2011weak}).

The monoidal category $_{H}\mathcal{M}$ has left duality. For any $X\in{}%
_{H}\mathcal{M}$, the left dual of $X$ is $X^{\ast}=\operatorname{Hom}%
_{k}\left(  X,k\right)  $, considered as an object of${}$ $_{H}\mathcal{M}$
via
\[
\left\langle hx^{\ast},x\right\rangle =\left\langle x^{\ast},S\left(
h\right)  x\right\rangle ,\quad\forall x\in X,x^{\ast}\in X^{\ast},h\in H.
\]
The evaluation map $ev_{X}:X^{\ast}\otimes_{t}X\rightarrow H_{t}$ and the
coevaluation map $coev_{X}:H_{t}\rightarrow X\otimes_{t}X^{\ast}$ are defined
as follows:
\[
ev_{X}\left(  1_{\left(  1\right)  }x^{\ast}\otimes1_{\left(  2\right)
}x\right)  =\left\langle x^{\ast},1_{\left(  1\right)  }x\right\rangle
1_{\left(  2\right)  },\quad coev_{X}\left(  z\right)  =\sum_{i}z_{\left(
1\right)  }x_{i}\otimes z_{\left(  2\right)  }x_{i}^{\ast},
\]
where $\left\{  \left(  x_{i},x_{i}^{\ast}\right)  \right\}  _{i}$ is a dual
basis of $X$.

Recall that a quasi-triangular weak Hopf algebra is a pair $\left(
H,R\right)  $, where $H$ is a weak Hopf algebra and $R\in\Delta^{op}\left(
1\right)  \left(  H\otimes_{k}H\right)  \Delta\left(  1\right)  $ satisfying
the following conditions:

\begin{enumerate}
\item $R\Delta\left(  h\right)  =\Delta^{op}\left(  h\right)  R,$ for all
$h\in H$.

\item $\left(  id\otimes\Delta\right)  R=R_{13}R_{12},$ $\left(  \Delta\otimes
id\right)  R=R_{13}R_{23}$, where $R_{12}=R\otimes1,$ $R_{23}=1\otimes R,$
$R_{13}=R^{1}\otimes1\otimes R^{2}$.

\item There exists $\bar{R}\in\Delta\left(  1\right)  \left(  H\otimes
_{k}H\right)  \Delta^{op}\left(  1\right)  $ with $R\bar{R}=\Delta^{op}\left(
1\right)  $, $\bar{R}R=\Delta\left(  1\right)  $.
\end{enumerate}

The element $R$ is called an R-matrix of $H$. We write $R=R_{i}={R_{i}}%
^{1}\otimes{R_{i}}^{2}$, $\forall i\in\mathbb{N}^{+}$.

Let $\left(  H,R\right)  $ be a quasi-triangular weak Hopf algebras, then the
following properties hold :
\begin{align}
{} &  \left(  \varepsilon_{s}\otimes id\right)  \left(  R\right)
=\Delta\left(  1\right)  ,\quad{} &  &  \left(  \varepsilon_{t}\otimes
id\right)  \left(  R\right)  =\Delta^{op}\left(  1\right)  ,\label{eR1}\\
{} &  \left(  id\otimes\varepsilon_{s}\right)  \left(  R\right)  =\left(
S\otimes id\right)  \Delta^{op}\left(  1\right)  ,\quad{} &  &  \left(
id\otimes\varepsilon_{t}\right)  \left(  R\right)  =\left(  S\otimes
id\right)  \Delta\left(  1\right)  ,\label{eR2}\\
{} &  \left(  S\otimes id\right)  \left(  R\right)  =\bar{R},\quad{} &  &
\left(  S\otimes S\right)  \left(  R\right)  =R.\label{SotSR}%
\end{align}

It's known from \cite[Proposition 5.2]{Nikshych2003Invariants} that if
$\left(  H,R\right)  $ is a quasi-triangular weak Hopf algebras, then the
monoidal category $\mathcal{C}={}_{H}\mathcal{M}$ is braided with a braiding%
\[
c_{M,N}\left(  1_{\left(  1\right)  }m\otimes1_{\left(  2\right)  }n\right)
=R^{2}1_{\left(  2\right)  }n\otimes R^{1}1_{\left(  1\right)  }m,\text{ for
}m\in M\text{, }n\in N,\ \text{where \ }M,N\in{}_{H}\mathcal{M}.
\]

\subsection{Module (Co)algebras and Weak Smash Products}

\label{sec-Module-coalgebras}

If $H$ is a Hopf algebra, a coalgebra $C$ in the monoidal category
$\mathcal{C}={}_{H}\mathcal{M}$ is simply a left $H$-module coalgebra, and
left $C$-comodules in $\mathcal{C}$ are left $\left(  C,H\right)  $-Hopf
modules. For weak Hopf algebras, some categorical notions and formulaic
notions, including cotensors, need to be reconciled. From now on, $H$ is a
finite dimensional weak Hopf algebra over $k$, and $\mathcal{C}$ is the
monoidal category ${}_{H}\mathcal{M}$.

A $k$-coalgebra $\left(  C,\Delta_{C},\varepsilon_{C}\right)  $ is a left
$H$-module coalgebra \cite{Bohm2000Doi-Hopf} if $C$ is a left $H$-module via
$h\otimes c\mapsto h\cdot c$ and for all $h\in H$, $c\in C$
\begin{align}
\Delta_{C}\left(  h\cdot c\right)   &  =h_{\left(  1\right)  }\cdot c_{\left(
1\right)  }\otimes h_{\left(  2\right)  }\cdot c_{\left(  2\right)
},\label{MC1}\\
\varepsilon_{C}\left(  h\cdot c\right)   &  =\varepsilon_{C}\left(
\varepsilon_{s}\left(  h\right)  \cdot c\right)  .\label{MC2}%
\end{align}

A left $H$-module $M$ is a left $\left(  C,H\right)  $-Hopf module
\cite{Bohm2000Doi-Hopf} if it is a left $C$-comodule such that the
compatibility condition
\begin{equation}
\rho\left(  h\cdot m\right)  =h_{\left(  1\right)  }\cdot m_{\left\langle
-1\right\rangle }\otimes h_{\left(  2\right)  }\cdot m_{\left\langle
0\right\rangle }\label{HM}%
\end{equation}
holds for any $m\in M$, $h\in H$.

Observe that (\ref{MC1}) implies $\Delta_{C}\left(  c\right)  =1_{\left(
1\right)  }\cdot c_{\left(  1\right)  }\otimes1_{\left(  2\right)  }\cdot
c_{\left(  2\right)  }\in C\otimes_{t}C$, hence (\ref{MC1}) holds
$\Leftrightarrow$ $\Delta_{C}$ is a morphism in $_{H}\mathcal{M}$.
Analogously, (\ref{HM}) holds $\Leftrightarrow$ $\rho$ is a morphism in
$\mathcal{C}$.

Let $M$ be a left $H$-module, the invariants of $M$ is the subspace
\[
\operatorname{Inv}M=\left\{  m\in M\mid h\cdot m=\varepsilon_{t}\left(
h\right)  \cdot m,\forall h\in H\right\}  .
\]
Note that $M^{\ast}=\operatorname{Hom}_{k}\left(  M,k\right)  $ is also a left
$H$-module via $\left\langle h\cdot m^{\ast},m\right\rangle =\left\langle
m^{\ast},S\left(  h\right)  \cdot m\right\rangle $, for $h\in H$, $m\in M$,
$m^{\ast}\in M^{\ast}$. It is routine to check that%
\[
\operatorname{Inv}M^{\ast}=\left\{  m^{\ast}\in M^{\ast}\mid\left\langle
m^{\ast},h\cdot m\right\rangle =\left\langle m^{\ast},\varepsilon_{s}\left(
h\right)  \cdot m\right\rangle ,\forall h\in H,m\in M\right\}  .
\]

\begin{lemma}
For any $M\in\mathcal{C}$, the linear map $\beta_{M}:\operatorname{Inv}%
M^{\ast}\rightarrow\operatorname{Hom}_{H}\left(  M,H_{t}\right)  $ defined by
\[
\beta_{M}\left(  m^{\ast}\right)  \left(  m\right)  =\left\langle m^{\ast
},1_{\left(  1\right)  }\cdot m\right\rangle 1_{\left(  2\right)  },\text{
}\forall m\in M,m^{\ast}\in\operatorname{Inv}M^{\ast},
\]
is an isomorphism.

\begin{proof}
First $\beta_{M}$ is an $H$-module map, since for any $m^{\ast}\in
\operatorname{Inv}M^{\ast}$, $h\in H$, $m\in M$,
\begin{align*}
\beta_{M}\left(  m^{\ast}\right)  \left(  h\cdot m\right)   & =\left\langle
m^{\ast},\left(  1_{\left(  1\right)  }h\right)  \cdot m\right\rangle
1_{\left(  2\right)  }=\left\langle m^{\ast},h_{\left(  1\right)  }\cdot
m\right\rangle \varepsilon_{t}\left(  h_{\left(  2\right)  }\right) \\
& =\left\langle m^{\ast},\varepsilon_{s}\left(  h_{\left(  1\right)  }\right)
\cdot m\right\rangle \varepsilon_{t}\left(  h_{\left(  2\right)  }\right)
=\left\langle m^{\ast},1_{\left(  1\right)  }\cdot m\right\rangle
\varepsilon_{t}\left(  h1_{\left(  2\right)  }\right) \\
& =h\cdot\left(  \beta_{M}\left(  m^{\ast}\right)  \left(  m\right)  \right)
.
\end{align*}

On the other hand, for all $f\in\operatorname{Hom}_{H}\left(  M,H_{t}\right)
$,
\begin{align*}
\varepsilon\left(  f\left(  h\cdot m\right)  \right)   & =\varepsilon\left(
h\cdot f\left(  m\right)  \right)  =\varepsilon\left(  hf\left(  m\right)
\right) \\
& =\varepsilon\left(  \varepsilon_{s}\left(  h\right)  f\left(  m\right)
\right)  =\varepsilon\left(  f\left(  \varepsilon_{s}\left(  h\right)  \cdot
m\right)  \right)  .
\end{align*}
So $\beta_{M}$ has a well-defined inverse $\beta_{M}^{-1}:\operatorname{Hom}%
_{H}\left(  M,H_{t}\right)  \rightarrow\operatorname{Inv}M^{\ast}$ by the
formula
\[
\left\langle \beta_{M}^{-1}\left(  f\right)  ,m\right\rangle =\varepsilon
\left(  f\left(  m\right)  \right)  .
\]

\end{proof}
\end{lemma}

\begin{remark}
For $M,N\in\mathcal{C}$, and $m^{\ast}\in\operatorname{Inv}M^{\ast},\ m\in
M,\ n\in N$, we have%
\begin{align}
\left(  \beta_{M}\left(  m^{\ast}\right)  \left(  m\right)  \right)  \cdot n
& =\left\langle m^{\ast},1_{\left(  1\right)  }\cdot m\right\rangle 1_{\left(
2\right)  }\cdot n,\label{left counit}\\
n\cdot\left(  \beta_{M}\left(  m^{\ast}\right)  \left(  m\right)  \right)   &
=1_{\left(  1\right)  }\cdot n\left\langle m^{\ast},1_{\left(  2\right)
}\cdot m\right\rangle .\label{right counit}%
\end{align}

\end{remark}

Now we are able to reconcile the notion of $\mathcal{C}$-coalgebras with the
notion of left $H$-module coalgebra.

\begin{lemma}
A triple $\left(  C,\Delta_{C},\varepsilon_{C}\right)  $, where $C\in
\mathcal{C}$, $\Delta_{C}\in\operatorname{Hom}_{\mathcal{C}}\left(
C,C\otimes_{t}C\right)  $, $\varepsilon_{C}\in\operatorname{Hom}_{\mathcal{C}%
}\left(  C,H_{t}\right)  $, is a coalgebra in $\mathcal{C}$ if and only if
$\left(  C,\Delta_{C},\beta_{C}^{-1}\left(  \varepsilon_{C}\right)  \right)  $
is a left $H$-module coalgebra. In addition, if $C$ is a $\mathcal{C}%
$-coalgebra, then the category $C$-$\operatorname*{Comod}_{\mathcal{C}}$ of
left $C$-comodules in $\mathcal{C}$ is equal to the category $_{H}%
^{C}\mathcal{M}$ of left relative $\left(  C,H\right)  $-Hopf modules.
\end{lemma}

\begin{proof}
The equivalence of the counit axioms follows from (\ref{left counit}) and
(\ref{right counit}).
\end{proof}

Let $D$ be a $\mathcal{C}$-coalgebra. For $\left(  M,\rho_{M}\right)
\in\operatorname*{Comod}_{\mathcal{C}}$-$D$, $\left(  N,\rho_{N}\right)  \in
D$-$\operatorname*{Comod}_{\mathcal{C}}$, we will show that the cotensor
product $M\square_{D}^{\mathcal{C}}N$ in $\mathcal{C}$ (see
Definition~\ref{def cotensor}) is exactly the classical cotensor product
$M\square_{D}N$ with the diagonal $H$-module action.

\begin{proposition}
\label{prop cot. is compa.}Let $D$ be a coalgebra in $\mathcal{C}$. Let
$\left(  M,\rho_{M}\right)  $ be a right $D$-comodule in $\mathcal{C}$, and
$\left(  N,\rho_{N}\right)  $ be a left $D$-comodule in $\mathcal{C}$. Then
$M\square_{D}N$ is an $H$-submodule of $M\otimes_{t}N$, and $M\square
_{D}^{\mathcal{C}}N=M\square_{D}N$.
\end{proposition}

\begin{proof}
For any $x=\sum_{i}m_{i}\otimes n_{i}\in M\square_{D}N$, we have
\begin{align*}
x  & =\sum_{i}m_{i\left\langle 0\right\rangle }\cdot\varepsilon_{D}\left(
m_{i\left\langle 1\right\rangle }\right)  \otimes n_{i}\\
& =\sum_{i}m_{i}\cdot\varepsilon_{D}\left(  1_{\left(  1\right)  }\cdot
n_{i\left\langle -1\right\rangle }\right)  \otimes1_{\left(  2\right)  }\cdot
n_{i\left\langle 0\right\rangle }\\
& =\sum_{i}m_{i}\cdot\varepsilon_{t}\left(  1_{\left(  1\right)  }%
\varepsilon_{D}\left(  n_{i\left\langle -1\right\rangle }\right)  \right)
\otimes1_{\left(  2\right)  }\cdot n_{i\left\langle 0\right\rangle }\\
& =\sum_{i}\left(  m_{i}\cdot\varepsilon_{D}\left(  n_{i\left\langle
-1\right\rangle }\right)  \right)  \cdot S\left(  1_{\left(  1\right)
}\right)  \otimes1_{\left(  2\right)  }\cdot n_{i\left\langle 0\right\rangle
}\\
& =\sum_{i}1_{\left(  1\right)  }\cdot\left(  m_{i}\cdot\varepsilon_{D}\left(
n_{i\left\langle -1\right\rangle }\right)  \right)  \otimes1_{\left(
2\right)  }\cdot n_{i\left\langle 0\right\rangle }\\
& =\Delta\left(  1\right)  \cdot x.
\end{align*}
Thus $M\square_{D}N$ is a subspace of $M\otimes_{t}N$. Now we get the
following commutative diagram
\[%
\def\mleftdelim{.}\def\mrightdelim{.}\def\mrowsep{1cm}\def\mcolumnsep{1cm}%
\begin{tikzpicture}[scale=1,samples=100,baseline,line width=0.5pt]\matrix (m) [matrix of math nodes,left delimiter={\mleftdelim},right delimiter={\mrightdelim},row sep=\mrowsep,column sep=\mcolumnsep]{ M\square _{D}N & M\otimes _{t}N& & & & & M\otimes _{t}D\otimes _{t}N \\ M\square _{D}N & M\otimes _{k}N & & & & & M\otimes _{k}D\otimes _{k}N \\}; \begin{scope}[every node/.style={midway,auto}] \draw[right hook-to] (m-1-1) --(m-1-2); \draw[right hook-to] (m-2-1) --(m-2-2); \draw[->] (m-1-2) --node{$\rho _{M}\otimes_{t} id_{N}-id_{M}\otimes_{t} \rho _{N}$}(m-1-7); \draw[->] (m-2-2) --node{$\rho _{M}\otimes id_{N}-id_{M}\otimes \rho _{N}$}(m-2-7); \draw[right hook-to] (m-1-2) --(m-2-2); \draw[right hook-to] (m-1-7) --(m-2-7); \draw[double distance=2pt] (m-1-1) --(m-2-1);\end{scope} \end{tikzpicture},
\]
and the result is clear.
\end{proof}

Let $\left(  C,\Delta_{C},\varepsilon_{C}\right)  $ be a left $H$-module
coalgebra and $F:{}_{H}^{C}\mathcal{M}\rightarrow{}^{C}\mathcal{M}$ be the
forgetful functor. Then by references \cite[Proposition 3.3]{Bohm2000Doi-Hopf}
and \cite[Proposition 2.1]{Caenepeel2000Modules}, the functor $F$ has a left
adjoin functor $\operatorname{Ind}:{}^{C}\mathcal{M}\rightarrow{}_{H}%
^{C}\mathcal{M}$. For completeness, we include the structure here.

Naturally, $C^{\ast}$ can be considered as a left $H^{\ast}$-comodule $\left(
C^{\ast},\rho_{C^{\ast}}\right)  $. Define $\operatorname{Ind}:{}%
^{C}\mathcal{M}\rightarrow{}{}_{H}^{C}\mathcal{M}$ as follows:
$\operatorname{Ind}\left(  W\right)  =\left(  H\otimes W\right)  \rho
_{C^{\ast}}\left(  \varepsilon_{C}\right)  =H\leftharpoonup\varepsilon
_{C\left\langle -1\right\rangle }\otimes W\cdot\varepsilon_{C\left\langle
0\right\rangle }$ as $k$-space, which is the subspace of $H\otimes_{k}W$
generated by elements of the form $\left\langle \varepsilon_{C},h_{\left(
1\right)  }\cdot w_{\left\langle -1\right\rangle }\right\rangle h_{\left(
2\right)  }\otimes w_{\left\langle 0\right\rangle }$. The left $H$-action and
left $C$-coaction $\rho$ on $\operatorname{Ind}\left(  W\right)  $ are given
by the formulas%
\begin{align*}
x\left(  \left\langle \varepsilon_{C},h_{\left(  1\right)  }\cdot
w_{\left\langle -1\right\rangle }\right\rangle h_{\left(  2\right)  }\otimes
w_{\left\langle 0\right\rangle }\right)   & =\left\langle \varepsilon
_{C},h_{\left(  1\right)  }\cdot w_{\left\langle -1\right\rangle
}\right\rangle xh_{\left(  2\right)  }\otimes w_{\left\langle 0\right\rangle
}\\
& =\left\langle \varepsilon_{C},x_{\left(  1\right)  }h_{\left(  1\right)
}\cdot w_{\left\langle -1\right\rangle }\right\rangle x_{\left(  2\right)
}h_{\left(  2\right)  }\otimes w_{\left\langle 0\right\rangle },\\
\rho\left(  \left\langle \varepsilon_{C},h_{\left(  1\right)  }\cdot
w_{\left\langle -1\right\rangle }\right\rangle h_{\left(  2\right)  }\otimes
w_{\left\langle 0\right\rangle }\right)  \allowbreak & =h_{\left(  1\right)
}\cdot w_{\left\langle -1\right\rangle }\otimes h_{\left(  2\right)  }\otimes
w_{\left\langle 0\right\rangle }\allowbreak,
\end{align*}
where $x,h\in H$, $w\in W$.

\begin{lemma}
[{cf. \cite[Proposition 3.3]{Bohm2000Doi-Hopf}, \cite[Proposition
2.1]{Caenepeel2000Modules}}]\label{lemma Ind is a left adj}Let $C$ be a left
$H$-module coalgebra. Then the functor $\operatorname{Ind}:{}^{C}%
\mathcal{M}\rightarrow{}_{H}^{C}\mathcal{M}$ is a left adjoint of the
forgetful functor $F:{}_{H}^{C}\mathcal{M}\rightarrow{}^{C}\mathcal{M}$.
\end{lemma}


Let $D$ be a coalgebra in $\mathcal{C}$. We have already known that $D$ can be
viewed as a $k$-coalgebra. If $D$ is a cosemisimple $\mathcal{C}$-coalgebra,
one may ask whether $D$ is cosemisimple as a $k$-coalgebra. We will show that
it's true under the assumption that $H$ is cosemisimple, and this result will
be used to present the structure of the irreducible Yetter-Drinfeld modules.

Let $A$ be a left $H$-module algebra. Then the smash product $A\#H$ is defined
on the $k$-space $A\otimes_{H_{t}}H$, where $H$ is a left $H_{t}$-module via
multiplication and $A$ is a right $H_{t}$ module via
\[
a\cdot z=S^{-1}\left(  z\right)  \cdot a,\quad a\in A,z\in H_{t}.
\]
Let $a\#h$ denote the class of $a\otimes h$ in $A\#H$. The multiplication of
$A\#H$ is given by
\[
\left(  a\#h\right)  \left(  b\#g\right)  =a\left(  h_{\left(  1\right)
}\cdot b\right)  \#h_{\left(  2\right)  }g,\quad a,b\in A,h,g\in H,
\]
and the unit is $1_{A}\#1_{H}$.

Observed that $A\#H$ is a left $H^{\ast}$-module algebra via
\[
f\cdot\left(  a\#h\right)  =a\#\left(  f\rightharpoonup h\right)  ,\quad f\in
H^{\ast},a\in A,h\in H.
\]
The following duality theorem was shown by Nikshych \cite{Nikshych2000A}.

\begin{lemma}
[{\cite[Theorem 3.3]{Nikshych2000A}}]\label{dual thm}There is an algebra
isomorphism between the algebras $\left(  A\#H\right)  \#H^{\ast}$ and
$\operatorname{End}\left(  A\#H\right)  _{A}$, where $A\#H $ is a right
$A$-module via multiplication.
\end{lemma}

In the case when $H$ is a Hopf algebra, it has been proved by Blattner and
Montgomery \cite{BlattnerMontgomery1985A} that $\left(  A\#H\right)
\#H^{\ast}\cong M_{n}\left(  A\right)  $, where $n=\dim H$. While if a weak
Hopf algebra $H$ is not free over $H_{t}$, $\left(  A\#H\right)  \#H^{\ast}$
might not be isomorphic to a matrix algebra over $A$. However, we have that
$\left(  A\#H\right)  \#H^{\ast}$ is Morita-equivalent to $A$.

Consider $A$ as a regular right $A$-module and a left $H_{t}$-module via the
left $H$-action. For $z\in H_{t}$, $a,b\in A$,
\[
z\cdot\left(  ab\right)  =\left(  z_{\left(  1\right)  }\cdot a\right)
\left(  z_{\left(  2\right)  }\cdot b\right)  =\left(  \left(  1_{\left(
1\right)  }z\right)  \cdot a\right)  \left(  1_{\left(  2\right)  }\cdot
b\right)  =\left(  z\cdot a\right)  b,
\]
hence $A$ is an $H_{t}$-$A$ bimodule.

\begin{proposition}
\label{Prop Morita eq}Let $H$ be a finite dimensional weak Hopf algebra, and
$A$ be a left $H$-module algebra. Then as right $A$-modules
\[
A\#H\cong H\otimes_{H_{t}}A,
\]
and the algebra $\left(  A\#H\right)  \#H^{\ast}$ is Morita-equivalent to $A$.

\begin{proof}
Define a map
\[
\Phi:H\otimes_{H_{t}}A\rightarrow A\#H\quad\text{via }h\otimes a\mapsto\left(
1_{A}\#h\right)  \left(  a\#1_{H}\right)  .
\]
$\Phi$ is well-defined, since for $z\in H_{t}$, $a\in A$ and $h\in H$,
\begin{align*}
\left(  1_{A}\#hz\right)  \left(  a\#1_{H}\right)   & =\left(  \left(
h_{\left(  1\right)  }z_{\left(  1\right)  }\right)  \cdot a\right)
\#h_{\left(  2\right)  }z_{\left(  2\right)  }\\
& =\left(  \left(  h_{\left(  1\right)  }1_{\left(  1\right)  }z\right)  \cdot
a\right)  \#h_{\left(  2\right)  }1_{\left(  2\right)  }\\
& =\left(  h_{\left(  1\right)  }\cdot\left(  z\cdot a\right)  \right)
\#h_{\left(  2\right)  }\\
& =\left(  1_{A}\#h\right)  \left(  \left(  z\cdot a\right)  \#1_{H}\right)  .
\end{align*}
Clearly $\Phi$ is an $A$-module map. Observe that for all $a\in A$, $h\in H$,
\begin{align*}
\left(  \left(  1_{A}\#h_{\left(  2\right)  }\right)  \left(  \left(
S^{-1}\left(  h_{\left(  1\right)  }\right)  \cdot a\right)  \#1_{H}\right)
\right)   & =\left(  h_{\left(  2\right)  }S^{-1}\left(  h_{\left(  1\right)
}\right)  \cdot a\right)  \#h_{\left(  3\right)  }\\
& =\left(  S^{-1}\left(  \varepsilon_{t}\left(  h_{\left(  1\right)  }\right)
\right)  \cdot a\right)  \#h_{\left(  2\right)  }\\
& =a\#\varepsilon_{t}\left(  h_{\left(  1\right)  }\right)  h_{\left(
2\right)  }\\
& =a\#h.
\end{align*}
Then $\Phi$ has a well-defined inverse, namely, $a\#h\mapsto$ $h_{\left(
2\right)  }\otimes S^{-1}\left(  h_{\left(  1\right)  }\right)  \cdot a$.

Since $H_{t}$ is semisimple and $H$ is a faithful $H_{t}$-module, $H$ is a
progenerator of $\mathcal{M}_{H_{t}}$. Hence, $H\otimes_{H_{t}}A$ is a
progenerator of $\mathcal{M}_{A}$. Now by Lemma \ref{dual thm}, we have
$\left(  A\#H\right)  \#H^{\ast}\cong\operatorname{End}\left(  A\#H\right)
_{A}\cong\operatorname{End}\left(  H\otimes_{H_{t}}A\right)  _{A}$, which
implies that $\left(  A\#H\right)  \#H^{\ast}$ is Morita-equivalent to $A$.
\end{proof}
\end{proposition}

\begin{corollary}
\label{cor A is s.s.}Let $H$ be a finite dimensional cosemisimple weak Hopf
algebra, and $A$ be a left $H$-module algebra. If the algebra $A\#H$ is
semisimple, then $A$ is also semisimple.

\begin{proof}
Since $A\#H$ and $H^{\ast}$ are semisimple, $\left(  A\#H\right)  \#H^{\ast}$
is semisimple by a Maschke-type theorem for weak Hopf algebras \cite[Theorem
1]{Zhang2006Maschke-type}. Thus $A$ is semisimple by Proposition \ref{Prop
Morita eq}.
\end{proof}
\end{corollary}

\subsection{Automorphism Braided Group for Quasi-triangular Weak Hopf
Algebras}

\label{sec-ABG-for-QT}

Let $\left(  H,R\right)  $ be a quasi-triangular weak Hopf algebra. The goal
of this subsection is to present Majid's braided reconstruction with
$\mathcal{C}={}_{H}\mathcal{M}$, and characterize the automorphism braided
group $U\left(  \mathcal{C}\right)  $ of the braided rigid monoidal category
$\mathcal{C}$.

We need some preliminary steps. First, take $B=C_{H}\left(  H_{s}\right)  $,
the centralizer of $H_{s}$. It's known from \cite[Proposition 2.11]%
{Bohm1999weak} that $1_{\left(  1\right)  }\otimes S\left(  1_{\left(
2\right)  }\right)  \in H_{s}\otimes_{k}H_{s}$ is a separable idempotent of
$H_{s}$, then
\[
B=C_{H}\left(  H_{s}\right)  =\left\{  1_{\left(  1\right)  }hS\left(
1_{\left(  2\right)  }\right)  \mid h\in H\right\}  .
\]
Now we can consider $B$ as an object of $\mathcal{C}$ via the left $H$-adjoint
action $\cdot_{ad}$, namely, $h\cdot_{ad}b=h_{\left(  1\right)  }bS\left(
h_{\left(  2\right)  }\right)  $, $\forall h\in H,b\in B$. We will show that
$B=U\left(  \mathcal{C}\right)  $, the automorphism braided group of
${}\mathcal{C}$.

The next step is to find an action $\alpha\in\operatorname{Nat}\left(
B\otimes_{t}id_{\mathcal{C}},id_{\mathcal{C}}\right)  $. For any
$X\in\mathcal{C}$, define a map
\[
\alpha_{X}:B\otimes_{t}X\rightarrow X,\quad1_{\left(  1\right)  }\cdot
_{ad}b\otimes1_{\left(  2\right)  }x\mapsto bx.
\]
Since for any $b\in B,x\in X$, $bx=b1_{\left(  1\right)  }S\left(  1_{\left(
2\right)  }\right)  1_{\left(  3\right)  }x=\left(  1_{\left(  1\right)
}\cdot_{ad}b\right)  1_{\left(  2\right)  }x$, $\alpha_{X}$ is well-defined.
Next, we check that each $\alpha_{X}$ is a morphism in $\mathcal{C}$. In fact,
we have
\begin{align*}
\alpha_{X}\left(  h\left(  1_{\left(  1\right)  }\cdot_{ad}b\otimes1_{\left(
2\right)  }x\right)  \right)   & =\alpha_{X}\left(  h_{\left(  1\right)
}\cdot_{ad}b\otimes h_{\left(  2\right)  }x\right)  =h_{\left(  1\right)
}bS\left(  h_{\left(  2\right)  }\right)  h_{\left(  3\right)  }x\\
& =h_{\left(  1\right)  }S\left(  h_{\left(  2\right)  }\right)  h_{\left(
3\right)  }bx=h\left(  bx\right)  =h\alpha_{X}\left(  1_{\left(  1\right)
}\cdot_{ad}b\otimes1_{\left(  2\right)  }x\right)  ,
\end{align*}
for all $h\in H$, $b\in B$, $x\in X$. The naturality of $\alpha$ is obvious.
Now given $f\in\operatorname{Hom}_{\mathcal{C}}\left(  V,B\right)  ,$ we
define $\theta_{V}\left(  f\right)  \in\operatorname{Nat}\left(  V\otimes
_{t}id_{\mathcal{C}},id_{\mathcal{C}}\right)  $ via $\theta_{V}\left(
f\right)  _{X}=\alpha_{X}\left(  f\otimes_{t}id_{X}\right)  ,$ for all
$X\in\mathcal{C}$.

\begin{lemma}
\label{lemma theta is iso}The natural transformation
\begin{align*}
& \theta:\operatorname{Hom}_{\mathcal{C}}\left(  \bullet,B\right)
\rightarrow\operatorname{Nat}\left(  \bullet\otimes_{t}id_{\mathcal{C}%
},id_{\mathcal{C}}\right) \\
& V\leadsto\theta_{V}:\operatorname{Hom}_{\mathcal{C}}\left(  V,B\right)
\rightarrow\operatorname{Nat}\left(  V\otimes_{t}id_{\mathcal{C}%
},id_{\mathcal{C}}\right)  ,\quad\theta_{V}\left(  f\right)  _{X}=\alpha
_{X}\left(  f\otimes_{t}id_{X}\right)  ,
\end{align*}
is an isomorphism with inverse given by%
\begin{align*}
\theta_{V}^{-1}:\operatorname{Nat}\left(  V\otimes_{t}id_{\mathcal{C}%
},id_{\mathcal{C}}\right)   & \rightarrow\operatorname{Hom}_{\mathcal{C}%
}\left(  V,B\right)  ,\\
\delta & \mapsto\delta_{H}\left(  1_{\left(  1\right)  }\left(  \bullet
\right)  \otimes1_{\left(  2\right)  }\right)  ,
\end{align*}
where $H\in\mathcal{C}$ is considered as the left regular representation.
\end{lemma}

\begin{proof}
For $V\in\mathcal{C}$, $\delta\in\operatorname{Nat}\left(  V\otimes
_{t}id_{\mathcal{C}},id_{\mathcal{C}}\right)  $, we get a linear map
$\delta_{H}\left(  1_{\left(  1\right)  }\left(  \bullet\right)
\otimes1_{\left(  2\right)  }\right)  :V\rightarrow H,$ $v\mapsto\delta
_{H}\left(  1_{\left(  1\right)  }v\otimes1_{\left(  2\right)  }\right)  $. We
first show that the image of $\delta_{H}\left(  1_{\left(  1\right)  }\left(
\bullet\right)  \otimes1_{\left(  2\right)  }\right)  $ lies in $B$. For any
$X\in\mathcal{C}$ and $x\in X$, the map $x_{r}:H\rightarrow X$, $h\mapsto hx$
is a morphism in $\mathcal{C}$. Since $\delta$ is natural under the morphism
$H\rightarrow X$, we have
\begin{equation}
\delta_{X}\left(  1_{\left(  1\right)  }v\otimes1_{\left(  2\right)
}x\right)  =\delta_{H}\left(  1_{\left(  1\right)  }v\otimes1_{\left(
2\right)  }\right)  x.\label{eq_naturaldelta1}%
\end{equation}
Specially take $X=H$, then for any $y\in H_{s}$ we have
\begin{align*}
\delta_{H}\left(  1_{\left(  1\right)  }v\otimes1_{\left(  2\right)  }\right)
y  & =\delta_{H}\left(  1_{\left(  1\right)  }v\otimes1_{\left(  2\right)
}y\right)  =\delta_{H}\left(  1_{\left(  1\right)  }v\otimes y1_{\left(
2\right)  }\right) \\
& \overset{(\ref{WHArelation5})}{=}\delta_{H}\left(  y_{\left(  1\right)
}v\otimes y_{\left(  2\right)  }\right)  =\delta_{H}\left(  y\left(
1_{\left(  1\right)  }v\otimes1_{\left(  2\right)  }\right)  \right) \\
& =y\delta_{H}\left(  1_{\left(  1\right)  }v\otimes1_{\left(  2\right)
}\right)  ,
\end{align*}
and thus $\delta_{H}\left(  1_{\left(  1\right)  }v\otimes1_{\left(  2\right)
}\right)  \in B$, for all $v\in V$. Also by (\ref{eq_naturaldelta1}), for
$h\in H$, $v\in V$,
\begin{align*}
& \delta_{H}\left(  1_{\left(  1\right)  }hv\otimes1_{\left(  2\right)
}\right)  \overset{(\ref{WHArelation4})}{=}\delta_{H}\left(  h_{\left(
1\right)  }v\otimes\varepsilon_{t}\left(  h_{\left(  2\right)  }\right)
\right)  =\delta_{H}\left(  h_{\left(  1\right)  }v\otimes h_{\left(
2\right)  }S\left(  h_{\left(  3\right)  }\right)  \right) \\
& =\delta_{H}\left(  h_{\left(  1\right)  }v\otimes h_{\left(  2\right)
}\right)  S\left(  h_{\left(  3\right)  }\right)  =\delta_{H}\left(
h_{\left(  1\right)  }\left(  1_{\left(  1\right)  }v\otimes1_{\left(
2\right)  }\right)  \right)  S\left(  h_{\left(  2\right)  }\right)
=h\cdot_{ad}\delta_{H}\left(  1_{\left(  1\right)  }v\otimes1_{\left(
2\right)  }\right)  ,
\end{align*}
hence $\delta_{H}\left(  1_{\left(  1\right)  }\left(  \bullet\right)
\otimes1_{\left(  2\right)  }\right)  \in\operatorname{Hom}_{\mathcal{C}%
}\left(  V,B\right)  $. Now define
\begin{align*}
\xi_{V}:\operatorname{Nat}\left(  V\otimes_{t}id_{\mathcal{C}},id_{\mathcal{C}%
}\right)   & \rightarrow\operatorname{Hom}_{\mathcal{C}}\left(  V,B\right)
,\\
\delta & \mapsto\delta_{H}\left(  1_{\left(  1\right)  }\left(  \bullet
\right)  \otimes1_{\left(  2\right)  }\right)  .
\end{align*}
Finally, we show that $\xi_{V}$ is the inverse for $\theta_{V}$. If
$f\in\operatorname{Hom}_{\mathcal{C}}\left(  V,B\right)  $, then for $v\in V$,%
\begin{align*}
\xi_{V}\left(  \theta_{V}\left(  f\right)  \right)  \left(  v\right)   &
=\theta_{V}\left(  f\right)  _{H}\left(  1_{\left(  1\right)  }v\otimes
1_{\left(  2\right)  }\right)  =f\left(  1_{\left(  1\right)  }v\right)
1_{\left(  2\right)  }=\left(  1_{\left(  1\right)  }\cdot_{ad}f\left(
v\right)  \right)  1_{\left(  2\right)  }\\
& =1_{\left(  1\right)  }f\left(  v\right)  S\left(  1_{\left(  2\right)
}\right)  1_{\left(  3\right)  }=f\left(  v\right)  1_{\left(  1\right)
}S\left(  1_{\left(  2\right)  }\right)  1_{\left(  3\right)  }=f\left(
v\right)  .
\end{align*}
Conversely, if $\delta\in\operatorname{Nat}\left(  V\otimes id_{\mathcal{C}%
},id_{\mathcal{C}}\right)  $, then%
\begin{align*}
\theta_{V}\left(  \xi_{V}\left(  \delta\right)  \right)  _{X}\left(
1_{\left(  1\right)  }v\otimes1_{\left(  2\right)  }x\right)   & =\alpha
_{X}\left(  \xi_{V}\left(  \delta\right)  \left(  1_{\left(  1\right)
}v\right)  \otimes1_{\left(  2\right)  }x\right)  =\alpha_{X}\left(
1_{\left(  1\right)  }\cdot_{ad}\xi_{V}\left(  \delta\right)  \left(
v\right)  \otimes1_{\left(  2\right)  }x\right) \\
& =\xi_{V}\left(  \delta\right)  \left(  v\right)  x=\delta_{H}\left(
1_{\left(  1\right)  }v\otimes1_{\left(  2\right)  }\right)  x=\delta
_{X}\left(  1_{\left(  1\right)  }v\otimes1_{\left(  2\right)  }x\right)  .
\end{align*}
Thus $\xi_{V}=\theta_{V}^{-1}$.
\end{proof}

We will show that $B=U\left(  \mathcal{C}\right)  $. In fact, the
representable conditions for modules, as stated in Section \ref{Section Center
of BRC}, are satisfied. To present the reconstruction, we need the inverse of
the natural transformation $\theta^{2}$ determined by the diagram (\ref{eq
deftheta}).

\begin{lemma}
For any $V\in\mathcal{C}$, the morphism%
\begin{align*}
& \theta_{V}^{2}:\operatorname{Hom}_{\mathcal{C}}\left(  V,B\otimes
_{t}B\right)  \rightarrow\operatorname{Nat}\left(  V\otimes_{t}id_{\mathcal{C}%
}{}^{\otimes_{t}2},id_{\mathcal{C}}{}^{\otimes_{t}2}\right) \\
& \theta_{V}^{2}\left(  f\right)  _{X,Y}=\left(  \alpha_{X}\otimes_{t}%
\alpha_{Y}\right)  \left(  id_{B}\otimes_{t}c_{B,X}\otimes_{t}id_{Y}\right)
\left(  f\otimes_{t}id_{X}\otimes_{t}id_{Y}\right)  ,\text{ }X,Y\in\mathcal{C}%
\end{align*}
is an isomorphism with inverse
\begin{align*}
& \xi_{V}^{2}:\operatorname{Nat}\left(  V\otimes_{t}id_{\mathcal{C}}%
{}^{\otimes_{t}2},id_{\mathcal{C}}{}^{\otimes_{t}2}\right)  \rightarrow
\operatorname{Hom}_{\mathcal{C}}\left(  V,B\otimes_{t}B\right) \\
& \xi_{V}^{2}\left(  \delta\right)  \left(  v\right)  =v_{\left[  1\right]
}S\left(  R^{2}\right)  \otimes R^{1}\cdot_{ad}v_{\left[  2\right]  },
\end{align*}
where $v_{\left[  1\right]  }\otimes v_{\left[  2\right]  }=\delta
_{H,H}\left(  1_{\left(  1\right)  }v\otimes1_{\left(  2\right)  }%
\otimes1_{\left(  3\right)  }\right)  \in H\otimes_{t}H$.
\end{lemma}

\begin{proof}
If $f\in\operatorname{Hom}_{\mathcal{C}}\left(  V,B\otimes_{t}B\right)  $, for
$X,Y\in\mathcal{C}$, $x\in X$, $y\in Y$, $v\in V$,
\begin{align*}
\theta_{V}^{2}\left(  f\right)  _{X,Y}\left(  1_{\left(  1\right)  }%
v\otimes1_{\left(  2\right)  }x\otimes1_{\left(  3\right)  }y\right)   &
=1_{\left(  1\right)  }f\left(  v\right)  ^{\left[  1\right]  }R^{2}%
x\otimes1_{\left(  2\right)  }\left(  R^{1}\cdot_{ad}f\left(  v\right)
^{\left[  2\right]  }\right)  y\\
& =f\left(  v\right)  ^{\left[  1\right]  }1_{\left(  1\right)  }R^{2}%
x\otimes1_{\left(  2\right)  }S\left(  1_{\left(  3\right)  }\right)  \left(
R^{1}\cdot_{ad}f\left(  v\right)  ^{\left[  2\right]  }\right)  y\\
& =f\left(  v\right)  ^{\left[  1\right]  }1_{\left(  1\right)  }R^{2}%
x\otimes\left(  \left(  1_{\left(  2\right)  }R^{1}\right)  \cdot_{ad}f\left(
v\right)  ^{\left[  2\right]  }\right)  y\\
& =f\left(  v\right)  ^{\left[  1\right]  }R^{2}x\otimes\left(  R^{1}%
\cdot_{ad}f\left(  v\right)  ^{\left[  2\right]  }\right)  y,
\end{align*}
where $f\left(  v\right)  =f\left(  v\right)  ^{\left[  1\right]  }\otimes
f\left(  v\right)  ^{\left[  2\right]  }\in B\otimes_{t}B$. Given $\delta
\in\operatorname{Nat}\left(  V\otimes_{t}id_{\mathcal{C}}{}^{\otimes_{t}%
2},id_{\mathcal{C}}{}^{\otimes_{t}2}\right)  $, from the naturality of
$\delta$, we have
\begin{equation}
\delta_{X,Y}\left(  1_{\left(  1\right)  }v\otimes1_{\left(  2\right)
}x\otimes1_{\left(  3\right)  }y\right)  =v_{\left[  1\right]  }x\otimes
v_{\left[  2\right]  }y.\label{eq_naturaldelta2}%
\end{equation}
Applying (\ref{eq_naturaldelta2}) with $X=Y=H$, we have
\begin{align}
\left(  hv\right)  _{\left[  1\right]  }\otimes\left(  hv\right)  _{\left[
2\right]  }  & =\delta_{H,H}\left(  1_{\left(  1\right)  }hv\otimes1_{\left(
2\right)  }\otimes1_{\left(  3\right)  }\right)  =\delta_{H,H}\left(
h_{\left(  1\right)  }v\otimes h_{\left(  2\right)  }S\left(  h_{\left(
5\right)  }\right)  \otimes h_{\left(  3\right)  }S\left(  h_{4}\right)
\right) \nonumber\\
& =h_{\left(  1\right)  }\delta_{H,H}\left(  1_{\left(  1\right)  }%
v\otimes1_{\left(  2\right)  }S\left(  h_{\left(  3\right)  }\right)
\otimes1_{\left(  3\right)  }S\left(  h_{2}\right)  \right)  =h_{\left(
1\right)  }\left(  v_{\left[  1\right]  }S\left(  h_{\left(  3\right)
}\right)  \otimes v_{\left[  2\right]  }S\left(  h_{2}\right)  \right)
\nonumber\\
& =h_{\left(  1\right)  }v_{\left[  1\right]  }S\left(  h_{\left(  3\right)
}\right)  \otimes h_{\left(  2\right)  }\cdot_{ad}v_{\left[  2\right]
},\label{eq5}%
\end{align}
for $h\in H$. Take $h=1$, we get
\begin{equation}
v_{\left[  1\right]  }\otimes v_{\left[  2\right]  }=1_{\left(  1\right)
}v_{\left[  1\right]  }S\left(  1_{\left(  3\right)  }\right)  \otimes
1_{\left(  2\right)  }\cdot_{ad}v_{\left[  2\right]  }.\label{eq6}%
\end{equation}
Then
\begin{align*}
1_{\left(  1\right)  }\cdot_{ad}\left(  v_{\left[  1\right]  }S\left(
R^{2}\right)  \right)  \otimes1_{\left(  2\right)  }\cdot_{ad}\left(
R^{1}\cdot_{ad}v_{\left[  2\right]  }\right)   & =1_{\left(  1\right)
}v_{\left[  1\right]  }S\left(  1_{\left(  2\right)  }R^{2}\right)
\otimes\left(  1_{\left(  3\right)  }R^{1}\right)  \cdot_{ad}v_{\left[
2\right]  }\\
& =1_{\left(  1\right)  }v_{\left[  1\right]  }S\left(  1_{\left(  3\right)
}\right)  S\left(  R^{2}\right)  \otimes R^{1}\cdot_{ad}\left(  1_{\left(
2\right)  }\cdot_{ad}v_{\left[  2\right]  }\right) \\
& \overset{(\ref{eq6})}{=}v_{\left[  1\right]  }S\left(  R^{2}\right)  \otimes
R^{1}\cdot_{ad}v_{\left[  2\right]  },
\end{align*}
and thus $\xi_{V}^{2}\left(  \delta\right)  \left(  v\right)  \in B\otimes
_{t}B$.

Next, we show that $\xi_{V}^{2}\left(  \delta\right)  \in\operatorname{Hom}%
_{\mathcal{C}}\left(  V,B\otimes_{t}B\right)  $. For $h\in H$, $v\in V$,%
\begin{align*}
\xi_{V}^{2}\left(  \delta\right)  \left(  hv\right)   & =\left(  hv\right)  _{
\left[  1\right]  }S\left(  R^{2}\right)  \otimes R^{1}\cdot_{ad}\left(
hv\right)  _{\left[  2\right]  }\\
& \overset{(\ref{eq5})}{=}h_{\left(  1\right)  }v_{\left[  1\right]  }S\left(
h_{\left(  3\right)  }\right)  S\left(  R^{2}\right)  \otimes\left(
R^{1}h_{\left(  2\right)  }\right)  \cdot_{ad}v_{\left[  2\right]  }\\
& =h_{\left(  1\right)  }v_{\left[  1\right]  }S\left(  h_{\left(  2\right)
}R^{2}\right)  \otimes\left(  h_{\left(  3\right)  }R^{1}\right)  \cdot
_{ad}v_{\left[  2\right]  }\\
& =h_{\left(  1\right)  }\cdot_{ad}\left(  v_{\left[  1\right]  }S\left(
R^{2}\right)  \right)  \otimes h_{\left(  2\right)  }\cdot_{ad}\left(
R^{1}\cdot_{ad}v_{\left[  2\right]  }\right) \\
& =h\xi_{V}^{2}\left(  \delta\right)  \left(  v\right)  ,
\end{align*}
as expected, and thus $\xi_{V}^{2}:\operatorname{Nat}\left(  V\otimes
_{t}id_{\mathcal{C}}{}^{\otimes_{t}2},id_{\mathcal{C}}{}^{\otimes_{t}%
2}\right)  \rightarrow\operatorname{Hom}_{\mathcal{C}}\left(  V,B\otimes
_{t}B\right)  $ is well-defined.

Now, we only need to check that $\xi_{V}^{2}$ and $\theta_{V}^{2}$ are mutual
inverses. First, let $f\in\operatorname{Hom}_{\mathcal{C}}\left(
V,B\otimes_{t}B\right)  $, then $\theta_{V}^{2}\left(  f\right)  _{H,H}\left(
1_{\left(  1\right)  }v\otimes1_{\left(  2\right)  }\otimes1_{\left(
3\right)  }\right)  =f\left(  v\right)  ^{\left[  1\right]  }R^{2}%
\otimes\left(  R^{1}\cdot_{ad}f\left(  v\right)  ^{\left[  2\right]  }\right)
$. So we have that
\begin{align*}
\xi_{V}^{2}\left(  \theta_{V}^{2}\left(  f\right)  \right)  \left(  v\right)
& =f\left(  v\right)  ^{\left[  1\right]  }{R_{1}}^{2}S\left(  {R_{2}}%
^{2}\right)  \otimes{R_{2}}^{1}\cdot_{ad}\left(  {R_{1}}^{1}\cdot_{ad}f\left(
v\right)  ^{\left[  2\right]  }\right) \\
& =f\left(  v\right)  ^{\left[  1\right]  }S\left(  {R_{2}}^{2}{R_{1}}%
^{2}\right)  \otimes\left(  {R_{2}}^{1}S\left(  {R_{1}}^{1}\right)  \right)
\cdot_{ad}f\left(  v\right)  ^{\left[  2\right]  }\\
& =f\left(  v\right)  ^{\left[  1\right]  }S\left(  1_{\left(  1\right)
}\right)  \otimes1_{\left(  2\right)  }\cdot_{ad}f\left(  v\right)  ^{\left[
2\right]  }\\
& =1_{\left(  1\right)  }\cdot_{ad}f\left(  v\right)  ^{\left[  1\right]
}\otimes1_{\left(  2\right)  }\cdot_{ad}f\left(  v\right)  ^{\left[  2\right]
}=f\left(  v\right)  ,
\end{align*}
for $v\in V$. On the other hand, for $\delta\in\operatorname{Nat}\left(
V\otimes_{t}id_{\mathcal{C}}{}^{\otimes_{t}2},id_{\mathcal{C}}{}^{\otimes
_{t}2}\right)  $,
\begin{align*}
\theta_{V}^{2}\left(  \xi_{V}^{2}\left(  \delta\right)  \right)  _{X,Y}\left(
1_{\left(  1\right)  }v\otimes1_{\left(  2\right)  }x\otimes1_{\left(
3\right)  }y\right)   & =\left(  \xi_{V}^{2}\left(  \delta\right)  \left(
v\right)  \right)  ^{\left[  1\right]  }R^{2}x\otimes\left(  R^{1}\cdot
_{ad}\left(  \xi_{V}^{2}\left(  \delta\right)  \left(  v\right)  \right)
^{\left[  2\right]  }\right)  y\\
& =v_{\left[  1\right]  }S\left(  {R_{2}}^{2}\right)  {R_{1}}^{2}%
x\otimes\left(  {R_{1}}^{1}\cdot_{ad}\left(  {R_{2}}^{1}\cdot_{ad}v_{\left[
2\right]  }\right)  \right)  y\\
& =v_{\left[  1\right]  }S\left(  {R_{1}}^{2}{R_{2}}^{2}\right)
x\otimes\left(  \left(  S\left(  {R_{1}}^{1}\right)  {R_{2}}^{1}\right)
\cdot_{ad}v_{\left[  2\right]  }\right)  y\\
& =v_{\left[  1\right]  }S\left(  1_{\left(  2\right)  }\right)
x\otimes\left(  1_{\left(  1\right)  }\cdot_{ad}v_{\left[  2\right]  }\right)
y\\
& \overset{(\ref{eq6})}{=}1_{\left(  1^{\prime}\right)  }v_{\left[  1\right]
}S\left(  1_{\left(  3^{\prime}\right)  }\right)  S\left(  1_{\left(
2\right)  }\right)  x\otimes\left(  1_{\left(  1\right)  }\cdot_{ad}\left(
1_{\left(  2^{\prime}\right)  }\cdot_{ad}v_{\left[  2\right]  }\right)
\right)  y\\
& =1_{\left(  1^{\prime}\right)  }v_{\left[  1\right]  }S\left(  1_{\left(
3^{\prime}\right)  }\right)  x\otimes\left(  1_{\left(  2^{\prime}\right)
}\cdot_{ad}v_{\left[  2\right]  }\right)  y\\
& =v_{\left[  1\right]  }x\otimes v_{\left[  2\right]  }y\overset
{(\ref{eq_naturaldelta2})}{=}\delta_{X,Y}\left(  1_{\left(  1\right)
}v\otimes1_{\left(  2\right)  }x\otimes1_{\left(  3\right)  }y\right)  .
\end{align*}
Thus we have $\xi_{V}^{2}=\left(  \theta_{V}^{2}\right)  ^{-1}$.
\end{proof}

We summarize the above discussion in the next theorem, and provide the
concrete multiplication, comultiplication, etc.

\begin{theorem}
\label{thm aubragro of H}Let $\left(  H,R\right)  $ be a quasi-triangular weak
Hopf algebra. Then the automorphism braided group of $\mathcal{C}={}%
_{H}\mathcal{M}$ is the object $B=\left(  C_{H}\left(  H_{s}\right)
,\cdot_{ad}\right)  \in\mathcal{C}$ with Hopf algebra structure in
$\mathcal{C}$ defined as follows.

\begin{enumerate}
\item The multiplication $m_{B}:B\otimes_{t}B\rightarrow B$ and the unit
$u_{B}:H_{t}\rightarrow B$ are defined by
\[
m_{B}\left(  1_{\left(  1\right)  }\cdot_{ad}a\otimes1_{\left(  2\right)
}\cdot_{ad}b\right)  =ab,\text{ }u_{B}\left(  z\right)  =z,\text{\quad}\forall
a,b\in B,z\in H_{t}.
\]

\item The comultiplication $\Delta_{B}:B\rightarrow B\otimes_{t}B$ and the
counit $\varepsilon_{B}:B\rightarrow H_{t}$ are defined by%
\[
\Delta_{B}\left(  b\right)  =b_{\left(  1\right)  }S\left(  R^{2}\right)
\otimes R^{1}\cdot_{ad}b_{\left(  2\right)  },\text{ }\varepsilon_{B}\left(
b\right)  =\varepsilon_{t}\left(  b\right)  .
\]

\item The antipode $S_{B}:B\rightarrow B$ is defined by%
\[
S_{B}\left(  b\right)  ={R}^{2}S\left(  R^{1}\cdot_{ad}b\right)  .
\]

\end{enumerate}
\end{theorem}

\begin{proof}
To show the theorem, we use $\alpha,\theta,\theta^{2}$ to compute the Hopf
algebra structure on $B$, determined by the diagrams (\ref{Fig1 stru of
B}--\ref{Fig3 stru of B}). As before, we use the unit isomorphisms $l_{X}$ and
$r_{X}$ identifying $H_{t}\otimes_{t}X$ and $X\otimes_{t}H_{t}$ with $X $, for
any $X\in\mathcal{C}$. According to (\ref{Fig1 stru of B}), the multiplication
and the unit are characterized by%
\[
m_{B}\left(  1_{\left(  1\right)  }\cdot_{ad}a\otimes1_{\left(  2\right)
}\cdot_{ad}b\right)  x=a\left(  bx\right)  ,\quad u_{B}\left(  z\right)
x=zx,\text{\quad}\forall a,b\in B,\text{\quad}z\in H_{t},\text{\quad}x\in X.
\]
Take $X=H$ and $x=1$ be the unit of $H$, then
\[
m_{B}\left(  1_{\left(  1\right)  }\cdot_{ad}a\otimes1_{\left(  2\right)
}\cdot_{ad}b\right)  =ab,\quad u_{B}\left(  z\right)  =z.
\]
And the counit is characterized according to (\ref{Fig2 stru of B}) by$\quad$%
\[
\quad\varepsilon_{B}\left(  b\right)  =b\cdot1=\varepsilon_{t}\left(
b\right)  ,\text{ }\forall b\in B.
\]
Again, the comultiplication is characterized according to (\ref{Fig2 stru of
B}) by
\[
\theta_{B}^{2}\left(  \Delta_{B}\right)  _{X,Y}\left(  1_{\left(  1\right)
}\cdot_{ad}b\otimes1_{\left(  2\right)  }x\otimes1_{\left(  3\right)
}y\right)  =b_{\left(  1\right)  }x\otimes b_{\left(  2\right)  }y,\text{
}\forall x\in X,y\in Y,
\]
so we apply $\xi_{B}^{2}$, and get
\[
\Delta_{B}\left(  b\right)  =\xi_{B}^{2}\left(  \theta_{B}^{2}\left(
\Delta_{B}\right)  \right)  \left(  b\right)  =b_{\left(  1\right)  }S\left(
R^{2}\right)  \otimes R^{1}\cdot_{ad}b_{\left(  2\right)  },\text{ }\forall
b\in B.
\]
Finally, the antipode $S_{B}$ is characterized according to (\ref{Fig3 stru of
B}) by
\begin{align*}
& \theta_{B}\left(  S_{B}\right)  _{X}\left(  1_{\left(  1\right)  }\cdot
_{ad}b\otimes1_{\left(  2\right)  }x\right) \\
& =r_{X}\left(  id_{X}\otimes_{t}ev_{X}\right)  \left(  id_{X}\otimes
_{t}\alpha_{X^{\ast}}\otimes_{t}id_{X}\right)  \left(  c_{B,X}\otimes
_{t}id_{X^{\ast}}\otimes_{t}id_{X}\right)  \left(  1_{\left(  1\right)  }%
\cdot_{ad}b\otimes_{t}coev_{X}\left(  1_{\left(  2\right)  }\cdot1\right)
\otimes_{t}1_{\left(  3\right)  }x\right) \\
& =r_{X}\left(  id_{X}\otimes_{t}ev_{X}\right)  \left(  id_{X}\otimes
_{t}\alpha_{X^{\ast}}\otimes_{t}id_{X}\right)  \left(  \sum_{i}c_{B,X}\left(
1_{\left(  1\right)  }\cdot_{ad}b\otimes1_{\left(  2\right)  }x_{i}\right)
\otimes1_{\left(  3\right)  }x_{i}^{\ast}\otimes1_{\left(  4\right)  }x\right)
\\
& =r_{X}\left(  id_{X}\otimes_{t}ev_{X}\right)  \left(  \sum_{i}1_{\left(
1\right)  }R^{2}x_{i}\otimes\alpha_{X^{\ast}}\left(  \left(  1_{\left(
2\right)  }R^{1}\right)  \cdot_{ad}b\otimes1_{\left(  3\right)  }x_{i}^{\ast
}\right)  \otimes1_{\left(  4\right)  }x\right) \\
& =r_{X}\left(  \sum_{i}1_{\left(  1\right)  }R^{2}x_{i}\otimes ev_{X}\left(
1_{\left(  2\right)  }\left(  R^{1}\cdot_{ad}b\right)  x_{i}^{\ast}%
\otimes1_{\left(  3\right)  }x\right)  \right) \\
& =r_{X}\left(  \sum_{i}1_{\left(  1\right)  }R^{2}x_{i}\otimes1_{\left(
2\right)  }\cdot\left\langle \left(  R^{1}\cdot_{ad}b\right)  x_{i}^{\ast
},1_{\left(  1^{\prime}\right)  }x\right\rangle 1_{\left(  2^{\prime}\right)
}\right) \\
& =r_{X}\left(  1_{\left(  1\right)  }{R_{1}}^{2}{R_{2}}^{2}S\left(  {R_{1}%
}^{1}bS\left(  {R_{2}}^{1}\right)  \right)  1_{\left(  1^{\prime}\right)
}x\otimes1_{\left(  2\right)  }\cdot1_{\left(  2^{\prime}\right)  }\right) \\
& =S\left(  1_{\left(  2\right)  }\right)  {R_{1}}^{2}{R_{2}}^{2}S\left(
{R_{1}}^{1}bS\left(  {R_{2}}^{1}\right)  \right)  1_{\left(  1\right)  }x\\
& =S^{2}\left(  1_{\left(  1\right)  }\right)  {R_{1}}^{2}{R_{2}}^{2}S\left(
bS\left(  {R_{1}}^{2}\right)  \right)  S\left(  {R_{1}}^{1}\right)  S\left(
1_{\left(  2\right)  }\right)  x\\
& =1_{\left(  1\right)  }{R_{1}}^{2}{R_{2}}^{2}S\left(  bS\left(  {R_{1}}%
^{2}\right)  \right)  S\left(  1_{\left(  2\right)  }{R_{1}}^{1}\right)  x\\
& ={R_{1}}^{2}{R_{2}}^{2}S\left(  {R_{1}}^{1}bS\left(  {R_{1}}^{2}\right)
\right)  x\\
& =R^{2}S\left(  R^{1}\cdot_{ad}b\right)  x,
\end{align*}
where $\left\{  \left(  x_{i},x_{i}^{\ast}\right)  \right\}  _{i}$ is a dual
basis of $X$. Then one can apply $\xi_{B}$ to $\theta_{B}\left(  S_{B}\right)
$, getting the formula as stated.
\end{proof}

\begin{remark}
A braided Hopf algebra structure on $B=C_{H}\left(  H_{s}\right)  $ was
constructed explicitly in \cite{LiuGHBraided2012}. By Theorem \ref{thm
aubragro of H} we know that the braided Hopf algebra $B$ they given is exactly
the automorphism braided group of the braided multitensor category
$\mathcal{C}$.
\end{remark}

\subsection{Structure of Yetter-Drinfeld Modules over Quasi-triangular Weak
Hopf Algebras}

\label{sec-stru-YD}

In this last subsection, we study the structure of Yetter-Drinfeld modules
over a finite dimensional quasi-triangular weak Hopf algebra $\left(
H,R\right)  $. We will characterize the simple Yetter-Drinfeld modules in
${}_{H}^{H}\mathcal{YD}$ in the case when the category ${}_{H}^{H}\mathcal{YD}
$ is semisimple, extending the results in \cite{LiuZhu2019On} to a weak Hopf
algebra version.

B\"{o}hm \cite{Bohm2000Doi-Hopf} generalized the notion of Yetter-Drinfeld
modules to weak Hopf algebras. A left-left $H$-Yetter-Drinfeld modules $M$ is
a vector space with an $H$-action and an $H$-coaction satisfying the following
conditions:%
\begin{align*}
\rho\left(  m\right)   & =m_{\left\langle -1\right\rangle }\otimes
m_{\left\langle 0\right\rangle }\in H\otimes_{t}M,\\
h_{\left(  1\right)  }m_{\left\langle -1\right\rangle }\otimes h_{\left(
2\right)  }m_{\left\langle 0\right\rangle }  & =\left(  h_{\left(  1\right)
}m\right)  _{\left\langle -1\right\rangle }h_{\left(  2\right)  }%
\otimes\left(  h_{\left(  1\right)  }m\right)  _{\left\langle 0\right\rangle
},
\end{align*}
for all $h\in H,m\in M$. Denote by ${}_{H}^{H}\mathcal{YD}$ the category of
the left-left Yetter-Drinfeld module over $H$.

If $H$ is weak Hopf algebra with bijective antipode, then the category ${}%
_{H}^{H}\mathcal{YD}$ is a braided monoidal category, and it's isomorphic to
the left center $\mathcal{Z}_{l}\left(  _{H}\mathcal{M}\right)  $ as braided
monoidal categories (see \cite{Caenepeel2005Yetter}). Here is a brief
description of the connecting functors. For an object $\left(  M,\gamma
_{M,\bullet}\right)  \in\mathcal{Z}_{l}\left(  _{H}\mathcal{M}\right)  $, the
map
\begin{equation}
\rho:M\rightarrow H\otimes_{t}M,\text{ }\rho\left(  m\right)  =\gamma
_{M,H}\left(  1_{\left(  1\right)  }m\otimes1_{\left(  2\right)  }\right)
\label{eqrho}%
\end{equation}
gives a left $H$-coaction on $M$, which makes $M$ into a Yetter-Drinfeld
module in $_{H}^{H}\mathcal{YD}$. Conversely, for $M\in{}_{H}^{H}\mathcal{YD}%
$, define a natural transformation $\gamma_{M,\bullet}$ by
\begin{equation}
\gamma_{M,X}\left(  1_{\left(  1\right)  }m\otimes1_{\left(  2\right)
}x\right)  =m_{\left\langle -1\right\rangle }x\otimes m_{\left\langle
0\right\rangle },\text{ for }X\in{}_{H}\mathcal{M},x\in X,m\in
M,\label{eqgama}%
\end{equation}
then $\left(  M,\gamma_{M,\bullet}\right)  $ is an object of $\mathcal{Z}%
_{l}\left(  _{H}\mathcal{M}\right)  $.

Now let $\left(  H,R\right)  $ be a quasi-triangular weak Hopf algebra. In
Section \ref{sec-ABG-for-QT} we have proved that $B=C_{H}\left(  H_{s}\right)
$ is the automorphism braided group of $\mathcal{C}={}_{H}\mathcal{M}$. It was
shown in \cite{Zhu2015Braided} that the Yetter-Drinfeld module category
${}_{H}^{H}\mathcal{YD}$ is isomorphic to the category of left $B$-comodules
for the braided Hopf algebra $B=C_{H}\left(  H_{s}\right)  $. We now give a
categorical interpretation, as an application of Theorem \ref{theorem center
iso H-comod} and Theorem \ref{thm aubragro of H}.

\begin{proposition}
[{cf. \cite[Theorem 2.5]{Zhu2015Braided}}]\label{YD iso Bcomod}There is an
equivalence from the category $B$-$\operatorname*{Comod}_{\mathcal{C}}$ of
left $B$-comodules to the category ${}_{H}^{H}\mathcal{YD}$:%
\[
\mathcal{F}:B\text{-}\operatorname*{Comod}\nolimits_{\mathcal{C}}\rightarrow
{}_{H}^{H}\mathcal{YD},\text{ }\left(  M,\rho_{R}\right)  \mapsto\left(
M,\rho\right)  ,\text{ }%
\]
where the left $H$-coaction $\rho:M\rightarrow H\otimes_{t}M$ is defined by
$\rho\left(  m\right)  =\sum m^{\left\langle -1\right\rangle }R^{2}\otimes
R^{1}m^{\left\langle 0\right\rangle }$, for $m\in M$. The quasi-inverse of
$\mathcal{F}$ is
\[
\mathcal{G}:{}_{H}^{H}\mathcal{YD}\rightarrow{}B\text{-}\operatorname*{Comod}%
\nolimits_{\mathcal{C}},\text{ }\left(  M,\rho\right)  \mapsto\left(
M,\rho_{R}\right)  ,
\]
where the left $B$-coaction $\rho_{R}:M\rightarrow B\otimes_{t}M$ is defined
by $\rho_{R}\left(  m\right)  =\sum m_{\left\langle -1\right\rangle }S\left(
R^{2}\right)  \otimes R^{1}m_{\left\langle 0\right\rangle }$, for $m\in M$.
Here we use the notation $\rho_{R}\left(  m\right)  =m^{\left\langle
-1\right\rangle }\otimes m^{\left\langle 0\right\rangle }$ for left
$B$-coaction to distinguish the $H$-coaction $\rho\left(  m\right)
=m_{\left\langle -1\right\rangle }\otimes m_{\left\langle 0\right\rangle }$.
\end{proposition}

\begin{proof}
Note that $\mathcal{Z}_{l}\left(  \mathcal{C}\right)  \cong{}_{H}%
^{H}\mathcal{YD}$. For any $\left(  M,\rho_{R}\right)  \in B$%
-$\operatorname*{Comod}\nolimits_{\mathcal{C}}$, one can easily check that
\[
\varphi_{M,M}\left(  \rho_{R}\right)  _{H}\left(  1_{\left(  1\right)
}m\otimes1_{\left(  2\right)  }\right)  =m^{\left\langle -1\right\rangle
}R^{2}\otimes R^{1}m^{\left\langle 0\right\rangle }%
\]
for $m\in M$, where $\varphi_{M,M}$ is defined as (\ref{eq defphi}). Then
combining Theorem \ref{theorem center iso H-comod} with (\ref{eqrho}),
$\mathcal{F}$ is an equivalence. Conversely, to show $\mathcal{G}$ is the
quasi-inverse of $\mathcal{F}$, it is enough to verify that for any $\left(
M,\rho\right)  \in{}_{H}^{H}\mathcal{YD}$, $\rho_{R}=\varphi_{M,M}^{-1}\left(
\gamma_{M,\bullet}\right)  $, where $\gamma_{M,\bullet}$ is defined as
(\ref{eqgama}). For $m\in M$, $x\in X$, we have%
\begin{align*}
\varphi_{M,M}\left(  \rho_{R}\right)  _{X}\left(  1_{\left(  1\right)
}m\otimes1_{\left(  2\right)  }x\right)   & =m_{\left\langle -1\right\rangle
}S\left(  {R_{1}}^{2}\right)  {R_{2}}^{2}x\otimes{R_{2}}^{1}{R_{1}}%
^{1}m_{\left\langle 0\right\rangle }\\
& =m_{\left\langle -1\right\rangle }S\left(  {R_{2}}^{2}{R_{1}}^{2}\right)
x\otimes S\left(  {R_{2}}^{1}\right)  {R_{1}}^{1}m_{\left\langle
0\right\rangle }\\
& =m_{\left\langle -1\right\rangle }S\left(  1_{\left(  2\right)  }\right)
x\otimes1_{\left(  1\right)  }m_{\left\langle 0\right\rangle }\\
& =1_{\left(  1^{\prime}\right)  }m_{\left\langle -1\right\rangle }S\left(
1_{\left(  2\right)  }\right)  x\otimes1_{\left(  2^{\prime}\right)
}1_{\left(  1\right)  }m_{\left\langle 0\right\rangle }\\
& =1_{\left(  1\right)  }m_{\left\langle -1\right\rangle }S\left(  1_{\left(
3\right)  }\right)  x\otimes1_{\left(  2\right)  }m_{\left\langle
0\right\rangle }\\
& =m_{\left\langle -1\right\rangle }x\otimes m_{\left\langle 0\right\rangle
}\\
& =\gamma_{M,X}\left(  1_{\left(  1\right)  }m\otimes1_{\left(  2\right)
}x\right)  .
\end{align*}
Now the result follows from Theorem \ref{theorem center iso H-comod}.
\end{proof}

The coproduct $\Delta_{B}$ of the braided group $B$ can be considered
canonically as a coassociative coproduct in $\mathrm{Vec}_{k}$, via
\[
\Delta_{B}:B\rightarrow B\otimes_{t}B\hookrightarrow B\otimes_{k}B.
\]
We have known in Section \ref{sec-Module-coalgebras} that $\left(
B,\Delta_{B},\varepsilon|_{B}\right)  $ is a left $H$-module coalgebra, and
$B$-$\operatorname*{Comod}\nolimits_{\mathcal{C}}={}_{H}^{B}\mathcal{M}$. We
will use the notation $\Delta_{B}\left(  b\right)  =b^{\left(  1\right)
}\otimes b^{\left(  2\right)  }$ to distinguish the original coproduct
$\Delta\left(  b\right)  =b_{\left(  1\right)  }\otimes b_{\left(  2\right)
}$ of $H$.

From now on, we assume that $k$ is algebraically closed of characteristic
zero, and $\left(  H,R\right)  $ is a semisimple quasi-triangular weak Hopf
algebra over $k$. It was shown by Etingof, Nikshych and Ostrik that $H$ is
cosemisimple \cite{Etingof2005On} and the category ${}_{H}^{H}\mathcal{YD}$ is
also semisimple. Then by the dual of Corollary \ref{cor A is s.s.}, the $k
$-coalgebra $\left(  B,\Delta_{B},\varepsilon|_{B}\right)  $ is cosemisimple.

A subcoalgebra $D$ of $\left(  B,\Delta_{B},\varepsilon|_{B}\right)  $ is
called $H$-adjoint-stable if $H\cdot_{ad}D\subseteq D$. Clearly, $\left(
B,\cdot_{ad},\Delta\right)  \in{}_{H}^{H}\mathcal{YD}$. For any
Yetter-Drinfeld submodule $D$ of $B$, it follows from Theorem \ref{thm D is
coalg}, that $D$ is an $H$-adjoint-stable subcoalgebra of $B$.

\begin{proposition}
Let $\left(  H,R\right)  $ be a quasi-triangular weak Hopf algebra. Then there
is a unique decomposition%
\[
B=D_{1}\oplus\cdots\oplus D_{r}%
\]
of minimal $H$-adjoint-stable subcoalgebras $D_{1},...,D_{r}$ of $B$. It
coincide with the decomposition of simple Yetter-Drinfeld modules.

Moreover, the decomposition of $B$ as direct sum of simple Yetter-Drinfeld
modules is unique, and the category ${}$%
\[
_{H}^{H}\mathcal{YD}\cong B\text{-}\operatorname*{Comod}\nolimits_{\mathcal{C}%
}=%
{\textstyle\bigoplus_{j=1}^{r}}
D_{j}\text{-}\operatorname*{Comod}\nolimits_{\mathcal{C}}%
\]
is a direct sum of indecomposable $\mathcal{C}$-module subcategories.
\end{proposition}

\begin{proof}
This follows from Proposition \ref{YD iso Bcomod} and Proposition \ref{decom
of B-comod}.
\end{proof}

Let $D$ be a minimal $H$-adjoint-stable subcoalgebra of $B$. Next we give the
structure of the indecomposable right $\mathcal{C}$-module category
$D$-$\operatorname*{Comod}\nolimits_{\mathcal{C}}$, applying Proposition
\ref{prop ihom for D-comod} and \ref{prop D-comod equiv} on $\mathcal{C}%
={}_{H}\mathcal{M}$.

For finite dimensional vector space $M,N$, one usually identifies $M^{\ast
}\otimes_{k}N$ with $\operatorname{Hom}_{k}\left(  M,N\right)  $ via
\[
\left(  m^{\ast}\otimes n\right)  \left(  m\right)  =\left\langle m^{\ast
},m\right\rangle n,\forall m\in M,n\in N,m^{\ast}\in M^{\ast},
\]
where $M^{\ast}=\operatorname{Hom}_{k}\left(  M,k\right)  $ is the dual vector
space. Then by Proposition \ref{prop cot. is compa.} for two objects
$M_{1},M_{2}\in{}D$-$\operatorname*{Comod}\nolimits_{\mathcal{C}}$, the
internal Hom $\underline{\operatorname{Hom}}\left(  M_{1},M_{2}\right)
={^{*}\hspace{-0.5ex}M}_{1}\square_{D}^{\mathcal{C}}M_{2}={^{*}\hspace
{-0.5ex}M}_{1}\square_{D}M_{2}\cong\operatorname{Hom}^{D}\left(  M_{1}%
,M_{2}\right)  $, is the set of $D$-comodule map from $M_{1}$ to $M_{2}$. As
an object of $_{H}\mathcal{M}$, the left $H$-action on $\operatorname{Hom}%
^{D}\left(  M_{1},M_{2}\right)  $ is given by
\[
\left(  h\cdot f\right)  \left(  m_{1}\right)  =\sum h_{\left(  2\right)
}f\left(  S^{-1}\left(  h_{\left(  1\right)  }\right)  m_{1}\right)  ,
\]
where $h\in H$, $f\in\operatorname{Hom}^{D}\left(  M_{1},M_{2}\right)  $,
$m_{1}\in M_{1}$. It's not difficult to verify that the evaluation map
$ev_{M_{1},M_{2}}^{\prime}:M_{1}\otimes\operatorname{Hom}^{D}\left(
M_{1},M_{2}\right)  \rightarrow M_{2}$ is exactly the regular evaluation map.
In particular, the internal endomorphism $\underline{\operatorname{Hom}%
}\left(  M_{1},M_{1}\right)  =\operatorname{End}^{D}\left(  M_{1}\right)
^{op}$. As a consequence of Proposition \ref{prop D-comod equiv}, we have:

\begin{proposition}
Let $D$ be a minimal $H$-adjoint-stable subcoalgebra of $B$. For any nonzero
$M\in{}D$-$\operatorname*{Comod}\nolimits_{\mathcal{C}}$, the algebra
$A=\operatorname{End}^{D}\left(  M\right)  ^{op}$ in $\mathcal{C}$ is
semisimple, and the functors
\begin{align*}
F &  =\operatorname{Hom}^{D}\left(  M,\bullet\right)  :{}D\text{-}%
\operatorname*{Comod}\nolimits_{\mathcal{C}}\rightarrow A\text{-}%
\mathrm{Mod}_{\mathcal{C}},\\
G &  =M\otimes_{A}\bullet{}:A\text{-}\mathrm{Mod}_{\mathcal{C}}\rightarrow
D\text{-}\operatorname*{Comod}\nolimits_{\mathcal{C}}%
\end{align*}
establish an equivalence of $\mathcal{C}$-module categories between
$D$-$\operatorname*{Comod}\nolimits_{\mathcal{C}}$ and $A$-$\mathrm{Mod}%
_{\mathcal{C}}\mathrm{.}$
\end{proposition}

We will give another characterization of the category $D$%
-$\operatorname*{Comod}\nolimits_{\mathcal{C}}$ by viewing it as a left module
category over the tensor category $\mathrm{Vec}_{k}$ with $X\otimes
M=X\otimes_{k}M$, for any $X\in\mathrm{Vec}_{k}$, $M\in D$%
-$\operatorname*{Comod}\nolimits_{\mathcal{C}}$. $X\otimes_{k}M$ is an object
of $D$-$\operatorname*{Comod}\nolimits_{\mathcal{C}}$ via the $H$-action and
$D$-coaction on the right tensorand $M$. For objects $M_{1},M_{2}\in
D$-$\operatorname*{Comod}\nolimits_{\mathcal{C}}$, and $X\in\mathrm{Vec}_{k}$,
the restriction of the canonical isomorphism
\[
\operatorname{Hom}_{k}\left(  X,\operatorname{Hom}_{k}\left(  M_{1}%
,M_{2}\right)  \right)  \cong\operatorname{Hom}_{k}\left(  X\otimes_{k}%
M_{1},M_{2}\right)
\]
on $\operatorname{Hom}_{k}\left(  X,\operatorname{Hom}_{H}^{D}\left(
M_{1},M_{2}\right)  \right)  $ induces a natural isomorphism
\[
\operatorname{Hom}_{k}\left(  X,\operatorname{Hom}_{H}^{D}\left(  M_{1}%
,M_{2}\right)  \right)  \cong\operatorname{Hom}_{H}^{D}\left(  X\otimes
_{k}M_{1},M_{2}\right)  ,
\]
Thus the internal Hom $\underline{\operatorname{Hom}}\left(  M_{1}%
,M_{2}\right)  =\operatorname{Hom}_{H}^{D}\left(  M_{1},M_{2}\right)  $, and
the evaluation map $ev_{M_{1},M_{2}}$ is indeed the regular evaluation map.

Applying Theorem \ref{thm liu zhu} to the module category $D$%
-$\operatorname*{Comod}\nolimits_{\mathcal{C}}$ over $\mathrm{Vec}_{k}$, we get:

\begin{theorem}
\label{theo-C=Vec-equiva}Let $D$ be a minimal $H$-adjoint-stable subcoalgebra
of $B$. If $0\neq W\in{}^{D}\mathcal{M}$, then $A_{W}=\operatorname{End}%
_{H}^{D}\left(  \operatorname{Ind}\left(  W\right)  \right)  $ is a semisimple
$k$-algebra, and the functors
\begin{align*}
F &  =\operatorname{Hom}_{H}^{D}\left(  \operatorname{Ind}\left(  W\right)
,\bullet\right)  :D\text{-}\operatorname*{Comod}\nolimits_{\mathcal{C}%
}\rightarrow\mathcal{M}_{A_{W}},\\
G &  =\bullet\otimes_{A_{W}}\left(  \operatorname{Ind}\left(  W\right)
\right)  :\mathcal{M}_{A_{W}}\rightarrow{}D\text{-}\operatorname*{Comod}%
\nolimits_{\mathcal{C}}%
\end{align*}
establish an equivalence of between $D$-$\operatorname*{Comod}%
\nolimits_{\mathcal{C}}$ and $\mathcal{M}_{A_{W}}$.

Furthermore, any irreducible object $V\in D$-$\operatorname*{Comod}%
\nolimits_{\mathcal{C}}$ is isomorphic to $U\otimes_{A_{W}}\operatorname{Ind}%
\left(  W\right)  $, for some simple right $A_{W}$-module $U$.
\end{theorem}

\begin{proof}
The proof we give here is similar to that of \cite[Proposition 5.2]%
{LiuZhu2019On}. To apply Theorem \ref{thm liu zhu} to the category
$\mathcal{M}=D$-$\operatorname*{Comod}\nolimits_{\mathcal{C}}$ and the object
$M=\operatorname{Ind}\left(  W\right)  \in\mathcal{M}$, we only need to verify
that $\operatorname{Ind}\left(  W\right)  $ generates the module category
$\mathcal{M}$ over $\mathrm{Vec}_{k}$. For any simple object $V\in\mathcal{M}%
$, we claim that the internal Hom $\underline{\operatorname{Hom}}\left(
M,V\right)  =\operatorname{Hom}_{H}^{D}\left(  M,V\right)  $ is nonzero. Since
$\operatorname{Ind}$ is the left adjoint of the forgetful functor ${}_{H}%
^{D}\mathcal{M}\rightarrow{}^{D}\mathcal{M}$, it suffices to show
$\operatorname{Hom}^{D}\left(  W,V\right)  \neq0$. Let $D^{\prime
}=\mathrm{span}\left\{  v^{\ast}\rightharpoonup v\mid v\in V,v^{\ast}\in
V^{\ast}\right\}  $ with $v^{\ast}\rightharpoonup v=\sum v^{\left\langle
-1\right\rangle }\left\langle v^{\ast},v^{\left\langle 0\right\rangle
}\right\rangle $. It is easy to check that $D^{\prime}$ is a nonzero left
coideal of $D$ and is also an $H$-submodule under the left $H$-action
$\cdot_{ad}$. Since $D$ is irreducible in ${}_{H}^{H}\mathcal{YD}$, then
$D^{\prime}=D$. So there exists a surjection $V^{\left(  n\right)
}\rightarrow D\rightarrow0$ in $^{D}\mathcal{M}$ for some $n\in\mathbb{N}^{+}%
$. Since $D$ is cosemisimple, there exists an injection $0\rightarrow
D\rightarrow$ $V^{\left(  n\right)  }$ in $^{D}\mathcal{M}$. Take a simple
$D$-subcomodule $W^{\prime}$ of $W$, then $W^{\prime}$ is isomorphic to a
simple left coideal of $D$. So there exists a left $D$-comodule injection
$j:W^{\prime}\rightarrow V$. Thus $\operatorname{Hom}^{D}\left(  W,V\right)
\neq0$, and the claim follows. Then by the isomorphism
\[
\operatorname{Hom}_{H}^{D}\left(  \operatorname{Hom}_{H}^{D}\left(
M,V\right)  \otimes M,V\right)  \cong\operatorname{Hom}\left(
\operatorname{Hom}_{H}^{D}\left(  M,V\right)  ,\operatorname{Hom}_{H}%
^{D}\left(  M,V\right)  \right)  \neq0,
\]
so the evaluation morphism $\operatorname{Hom}_{H}^{D}\left(  M,V\right)
\otimes M\rightarrow V$ is a surjection in $\mathcal{M}$. Hence $M$ is a
generator, and the result follows.
\end{proof}

\bibliographystyle{abbrv}

\end{document}